%% file: main.tex
\documentclass[11pt,reqno]{amsart}

\usepackage[T1]{fontenc}
\usepackage[utf8]{inputenc}
\usepackage{lmodern}
\usepackage{microtype}
\usepackage{booktabs}
\usepackage[margin=1.12in]{geometry}
\usepackage{setspace}
\usepackage{enumitem}

\usepackage{amsmath,amssymb,amsfonts,amsthm,mathtools}
\usepackage{mathrsfs}
\usepackage{bbm}
\usepackage{bm}

\numberwithin{equation}{section}
\allowdisplaybreaks

\usepackage{tikz}
\usetikzlibrary{arrows.meta,positioning,calc}

\usepackage[numbers,sort&compress]{natbib}

\usepackage[
    colorlinks=true,
    linkcolor=blue,
    citecolor=blue,
    urlcolor=blue
]{hyperref}

\usepackage[nameinlink,capitalize,noabbrev]{cleveref}

\newif\ifdraftversion
\draftversionfalse

\ifdraftversion
    \usepackage[pagewise]{lineno}
    \linenumbers
\fi

\newtheoremstyle{greenstyle}
  {3pt}                     
  {3pt}                     
  {\color{green!50!black!100!}}           
  {}                        
  {\bfseries\color{green!50!black!100!}}  
  {.}                       
  {.5em}                    
  {}                        

\theoremstyle{plain}
\newtheorem{theorem}{Theorem}[section]
\newtheorem{mainthm}{Main Theorem}[section]
\newtheorem{proposition}{Proposition}[section]
\newtheorem{corollary}{Corollary}[section]
\newtheorem{lemma}{Lemma}[section]

\newtheorem{definition}{Definition}[section]
\newtheorem{assumption}{Assumption}[section]
\newtheorem{example}{Example}[section]

\theoremstyle{remark}
\newtheorem{remark}{Remark}[section]
\newtheorem{notation}{Notation}[section]

\newcommand{\inputfile}[2][]{%
    \IfFileExists{#2}{%
        \input{#2}%
    }{%
        \ifx&#1&%
            \typeout{*** Missing file: #2 ***}%
        \else%
            \section*{#1}%
            \addcontentsline{toc}{section}{#1}%
            \noindent\textbf{Placeholder.}
            The file \texttt{#2} will be inserted here in the expanded manuscript.
            \typeout{*** Missing file: #2 ***}%
        \fi%
    }%
}

\title[Macroscopic Cycles and Poisson--Kingman Universality]%
{Macroscopic Feynman Cycles and Poisson--Kingman Universality in Bose Condensation}

\author{Wen SUN}
\address{School of Mathematical Sciences, University of Science and Technology of China, Jinzhai 96, 230026 Hefei}
\email{wensun.ustc@gmail.com}

\date{\today}



\begin{document}

\maketitle

\begin{abstract}
We prove a canonical limit theorem for the macroscopic Feynman cycles
of finite-volume ideal Bose gases.  Cycles carry marks in a general
Polish space $\mathsf{M}$, encoding spatial, geometric, spectral, or internal
data.  After removing a deterministic background density
$\rho_{\mathrm{bg}}$, the marked macroscopic cycle process converges
in the canonical ensemble to a marked Poisson--Kingman bridge of total
mass $\rho - \rho_{\mathrm{bg}}$.  The bridge is constructed from
a marked Poisson point process with intensity
$x^{-1}\eta_x(dm)\,dx$, conditioned on total mass~$\rho - \rho_{\mathrm{bg}}$, where the
kernel $x \mapsto \eta_x$ and its total-mass profile
$\phi(x) = \eta_x(\mathsf{M})$ are determined by the low-energy
spectral data visible on the scale $j \sim V_L$.

When $\phi$ is constant, the bridge reduces to a Gamma bridge and the
ranked cycle lengths follow the Poisson--Dirichlet law.  We verify
this for the ideal Bose gas in dimension $d > 2$ under periodic,
Dirichlet, and Neumann boundary conditions: in all three cases
$\phi \equiv 1$ and the ranked lengths converge to
$\mathrm{PD}(0,1)$, while the mark kernels distinguish the three
models through their winding, killed-bridge, and reflected-bridge
geometry.  When $\phi$ is not constant, the bridge is no longer Gamma
and the ranked lengths are not Poisson--Dirichlet.  As a concrete
example, a critical double-well potential whose tunnelling splitting
satisfies $V_L \Delta_L \to \gamma$ gives
$\phi_\gamma(x) = 1 + e^{-\beta\gamma x}$; more generally, a
finite-type visible spectrum with $Q$ components yields
$\phi(x) = \sum_{r=1}^{Q} \theta_r e^{-\beta\lambda_r x}$.  These
results identify Poisson--Kingman bridges as the canonical
universality class for marked macroscopic Bose cycles, with the
visible low-energy spectrum selecting the particular bridge.
\end{abstract}



\input{01-introduction.tex}

\input{02-model-cycle-representation-path-marks.tex}

\input{03-limit-objects-main-results.tex}

\input{04-poissonization-mode-decomposition.tex}

\input{05-brownian-loop-shape-winding.tex}

\input{06-proofs}



%






\section*{Acknowledgments}
 This work is supported by the National Key R\&D Program of China (No. 2022YFA1006500) and by the National Natural Science Foundation of China (No. 12401171).

\bibliographystyle{amsplain}
\bibliography{ref}
\bigskip
\end{document}

%% file: 01-introduction.tex
\section{Introduction}

\subsection{Bose--Einstein condensation and the Feynman cycle picture}
\label{subsec:intro-BEC-cycles}

Bose--Einstein condensation (BEC), predicted by Einstein in 1924--1925 following Bose's
quantum statistics~\cite{Bose1924,Einstein1924,Einstein1925}, is a macroscopic manifestation
of quantum indistinguishability.  Two complementary mathematical descriptions have played a
central role.  The first is spectral.  In the Penrose--Onsager
formulation~\cite{PenroseOnsager1956}, condensation is detected by a macroscopic eigenvalue
of the one-particle reduced density matrix; in homogeneous systems this is closely related to
off-diagonal long-range order in the sense of Yang~\cite{Yang1962}.  Rigorous accounts of the
spectral viewpoint, especially for dilute and trapped Bose gases, can be found in the work of
Lieb, Seiringer, Yngvason and
collaborators~\cite{LiebSeiringer2002,LiebSeiringerSolovejYngvason2005}.

The second description is geometric and goes back to Feynman's path-integral
picture~\cite{Feynman1953}.  The symmetrisation of the bosonic partition function can be
represented by particle-exchange cycles: a cycle of length~$j$ corresponds to a Brownian loop
of imaginary-time length~$\beta j$.  Short cycles describe thermal excitations, whereas cycles
whose lengths are comparable with the volume~$V_L=|\Lambda_L|$ carry a macroscopic number of
particles and provide the cycle-level signature of condensation.

This Feynman cycle picture has been made rigorous in a series of works.
S\"ut\H{o}~\cite{SutoPercolation1993,SutoPercolation2002} established the connection between
BEC and a percolation transition of permutation cycles in the ideal Bose gas.
Ueltschi~\cite{Ueltschi2006} further developed the rigorous cycle representation and clarified
the role of long cycles as a geometric manifestation of condensation.  Betz and Ueltschi
introduced spatial random permutations as a probabilistic model class inspired by the Bose gas
and analysed the occurrence of infinite cycles in regimes corresponding to
condensation~\cite{BetzUeltschi2009}.  Together, these works show that long or infinite cycles
form a natural geometric order parameter for BEC, complementary to the spectral
Penrose--Onsager criterion.

The present paper is motivated by the interface between these spectral and geometric
descriptions.  The standard cycle expansion provides a formal bridge between them, since the
one-particle Hamiltonian contributes to the weight of each exchange cycle.  However, this
connection by itself does not identify the joint scaling limit of the long cycles, nor does it
explain how the low-energy spectral structure is encoded by the macroscopic exchange loops.
We therefore ask a question finer than whether macroscopic or infinite cycles appear: what
random object do they form, what information beyond cycle length survives in the macroscopic
limit and how does the visible low-energy
spectrum select the limiting cycle process together with its marks?

The purpose of this paper is to identify the canonical limiting law of these macroscopic marked cycles and to show how it is selected by the visible low-energy spectral data.

\subsection{The classical Poisson--Dirichlet law and the spectral assumption behind it}
\label{subsec:intro-PD-spectral-assumption}

We now recall the classical Poisson--Dirichlet limit for macroscopic Feynman cycles and
reinterpret it in the spectral language used throughout this paper.  The main point is that the
usual Poisson--Dirichlet law is not merely a consequence of the existence of a condensate; it
also reflects a particular low-energy spectral regime.

In the homogeneous periodic ideal Bose gas in dimension~$d>2$, above the critical density,
Betz and Ueltschi~\cite{BetzUeltschiPD2011} proved the Poisson--Dirichlet limit in the
spatial-random-permutation framework.  If~$N_L$ particles live in the
torus~$\mathbb T_L^d$, $V_L=L^d$, and $N_L/V_L\to\rho>\rho_c$, then the ranked cycle
lengths $\ell_1^L\geq \ell_2^L\geq\cdots$ satisfy
\[
        \left(\frac{\ell^L_1}{V_L},\frac{\ell^L_2}{V_L},\ldots\right)
        \Longrightarrow
        (\rho-\rho_c)(P_1,P_2,\ldots),
        \qquad
        (P_i)_{i\geq1}\sim \mathrm{PD}(0,1).
\]

Betz and Ueltschi prove this result by reformulating the annealed spatial model in Fourier
space.  They introduce occupation numbers~$n_k$ and, conditional on these occupations, independent
non-spatial weighted permutations~$\pi_k$ within the individual Fourier modes.  Their
occupation-number estimates show that the zero mode carries the macroscopic fraction
$\nu=1-\rho_c/\rho$, while no nonzero mode contributes a macroscopic cycle.  The largest
spatial cycles therefore have the same limit as the largest cycles of a non-spatial weighted
permutation on~$n_0$ points.  They then apply the asymptotic cycle law for these non-spatial
weighted permutations and rank the resulting cycle lengths: under their hypotheses,
$\alpha_j\to\alpha$ gives
$\mathrm{PD}(0,e^{-\alpha})$.  In the ideal Bose gas $\alpha=0$, hence
$\mathrm{PD}(0,1)$.

More recently, K\"onig, Vogel and Zass~\cite{KonigVogelZass2025} proved a
$\mathrm{PD}(0,1)$ limit for the ranked long-loop lengths of the canonical free Bose gas under
periodic and diffusive boundary conditions.  Bai, K\"onig and
Vogel~\cite{BaiKonigVogel2025} obtained a related Poisson--Dirichlet limit for a
non-interacting mean-field trapped gas.  These results extend the classical one-component
Poisson--Dirichlet picture to canonical loop representations in a broader range of settings.

The purpose here is to isolate the spectral condition behind this classical picture and to
determine what replaces it when more than one low-energy component remains visible.  The
organising mechanism is the scale of the one-particle spectral gaps seen by cycles with
$j\asymp V_L$.
Concretely, we start from the canonical Feynman-cycle expansion and identify the part of the
one-particle spectrum that is visible to cycles with $j\sim V_L$.  Let $K_L$ be the
finite-volume one-particle Hamiltonian, shifted so that its ground-state energy is zero, and set
\[
        q_{L,j}=\operatorname{Tr} e^{-\beta jK_L}.
\]
The canonical partition function has the cycle expansion
\[
        Z_{L,N}
        =
        \sum_{\substack{(n_j)_{j\geq1}\\ \sum_{j\geq1} j n_j=N}}
        \prod_{j\geq1}
        \frac{1}{n_j!}\left(\frac{q_{L,j}}{j}\right)^{n_j}.
\]
Here $1/j$ is the universal cycle-combinatorial factor, while the one-cycle trace $q_{L,j}$
contains the model-dependent spectral information.  Thus the connection between the spectral
theory of BEC and the macroscopic Feynman-cycle process is already encoded in the
weight~$q_{L,j}/j$.

On the macroscopic scale $j/V_L\to x\in(0,\infty)$, only spectral gaps of
order~$V_L^{-1}$ or smaller can contribute nontrivially.  Suppose, for instance, that the
effective low-energy contribution has the form
\[
        q^{\mathrm{eff}}_{L,j}
        =
        \sum_r \theta_r \exp\{-\beta j\varepsilon_{r,L}\},
        \qquad
        V_L\varepsilon_{r,L}\longrightarrow \lambda_r\in[0,\infty],
\]
where
\[
        0=\varepsilon_{0,L}\leq \varepsilon_{1,L}\leq\varepsilon_{2,L}\leq\cdots
\]
are the low-lying eigenvalues of $K_L$, and where $\theta_r$ may encode degeneracy, an
internal weight, or a mark multiplicity.  Then the modes with finite $\lambda_r$ remain visible
to cycles of length $j\sim V_L$, whereas modes with $\lambda_r=\infty$ disappear on this
scale.  Consequently,
\[
        q^{\mathrm{eff}}_{L,j}
        \longrightarrow
        \phi(x)
        :=
        \sum_{\lambda_r<\infty} \theta_r e^{-\beta\lambda_r x}.
\]
The function~$\phi$ is the visible spectral profile.  It is the scalar datum that determines
the length-level universality class of the macroscopic condensate cycles; the corresponding
eigenvectors or low-energy components determine the mark kernels.

For the ordinary periodic ideal Bose gas in dimension $d>2$, this visible profile is constant.
Indeed, the gap above the zero-momentum ground state is of order~$L^{-2}$, while $V_L=L^d$,
and hence
\[
        V_L\varepsilon_{1,L}\asymp L^{d-2}\longrightarrow\infty.
\]
All excited modes are therefore invisible to cycles with $j\sim V_L$, and only the ground
state remains.  Thus
\[
        \phi(x)\equiv1.
\]
At the unmarked length level, the macroscopic cycle weights are then governed by the
logarithmic intensity $dx/x$.  After imposing the canonical condensate-mass constraint
$\sum_i x_i=\rho-\rho_c$, one obtains the Gamma bridge of total mass $\rho-\rho_c$, whose
ranked jumps are distributed as $(\rho-\rho_c)\,\mathrm{PD}(0,1)$.  Having identified the
spectral assumption behind this classical law, we now state the general result.

\subsection{Main result: marked Poisson--Kingman bridges}
\label{subsec:intro-main-result}

The main theorem turns the spectral principle described above into a canonical limit theorem
for macroscopic Feynman cycles.  As in the preceding subsection, the finite-volume cycle
process separates into an effective part, governed by the visible low-energy modes, and a
background part.  The background may carry a positive particle
density~$\rho_{\mathrm{bg}}$ (equal to the usual critical density~$\rho_c$ in the standard
examples), but it produces no atoms in any fixed macroscopic length window.

Cycles may carry marks in a Polish space~$\mathsf M$, encoding spatial, geometric, spectral,
or internal data.  Let
\[
  \Xi_{L,N_L}
  =
  \sum_c \delta_{(U_c,\,j_c/V_L,\,m_c)}
\]
be the marked point process of cycles in the canonical ensemble, where $j_c$ is the cycle
length, $m_c$ its mark, and $U_c\in[0,1]$ an auxiliary coordinate used only to separate
atoms.  The visible spectral data are encoded by a limiting marked kernel
$x\mapsto\eta_x$, $x>0$, whose total mass
\[
  \phi(x)=\eta_x(\mathsf M)
\]
is the visible length profile.

The limiting object is constructed as follows.  Consider the marked Poisson point
process~$\Pi^{(\kappa)}$ with intensity
\[
  du\,\frac{e^{-\kappa x}}{x}\,\eta_x(dm)\,dx,
  \qquad \kappa>0,
\]
and let $T^{(\kappa)}$ denote the sum of the $x$-coordinates of all atoms.  If
$N_L/V_L\to\rho>\rho_{\mathrm{bg}}$, then the density available to macroscopic visible
cycles is $\rho_{\mathrm{eff}}=\rho-\rho_{\mathrm{bg}}$, and we define the Poisson--Kingman
bridge
\[
  \Pi^{\mathrm{br}}_{\rho_{\mathrm{eff}}}
  =
  \mathcal L\bigl(\Pi^{(\kappa)}\,\big|\,
  T^{(\kappa)}=\rho_{\mathrm{eff}}\bigr).
\]
The conditioning is understood through disintegration; the auxiliary parameter~$\kappa$
disappears after conditioning, so the bridge law does not depend on its value.  The name comes
from the classical Poisson--Kingman construction of random
partitions~\cite{PermanPitmanYor1992,Pitman2003PoissonKingman,PitmanYor1997}: one takes the jumps of a
subordinator, conditions on their total mass, and ranks the result to obtain a
Poisson--Kingman partition.  Here we keep the full marked point process and condition on the
total visible mass fixed by the canonical ensemble; the word ``bridge'' refers to this endpoint
conditioning, in analogy with Gamma and stable bridges. 

The main theorem, \cref{thm:main-canonical-bridge-limit}, states that under the assumptions
formulated in \cref{sec:limit-objects-main-results},
\[
  \Xi_{L,N_L}
  \;\Longrightarrow\;
  \Pi^{\mathrm{br}}_{\rho_{\mathrm{eff}}}
\]
in the length-bounded topology, which observes all cycles with positive rescaled length and
ignores only those whose rescaled lengths vanish.  The background carries asymptotically the
deterministic density~$\rho_{\mathrm{bg}}$, while the remaining
density~$\rho_{\mathrm{eff}}$ is carried by the effective part.  All random macroscopic atoms
in the limit come from the visible spectrum.

Two pieces of limiting data play distinct roles in this theorem.  The scalar
profile~$\phi$ alone determines the macroscopic length law.  If $\phi\equiv\gamma$, the bridge
is a Gamma bridge and the ranked normalised jumps have law $\mathrm{PD}(0,\gamma)$; in
particular, $\phi\equiv1$ gives the classical $\mathrm{PD}(0,1)$ law.  A non-constant visible
profile, arising when low-energy spectral splittings of order~$V_L^{-1}$ survive the
thermodynamic limit, produces a Poisson--Kingman bridge whose ranked lengths are generally not
Poisson--Dirichlet.  The full marked kernel~$x\mapsto\eta_x$ describes the additional geometric,
boundary, spectral, or metastable information carried by macroscopic cycles; it distinguishes
models that share the same scalar profile but differ in their spatial or internal structure.

In this sense, the \(\mathrm{PD}(0,1)\) limit is universal for the ordinary ideal Bose gas:
as shown in \cref{sec:three-bc-ideal-bose}, the Weyl law makes the boundary-dependent
spectral details invisible on the macroscopic cycle scale for all standard boundary conditions,
collapsing the visible profile to $\phi\equiv 1$.  Non-Gamma Poisson--Kingman limits arise
only when the low-energy spectrum contains additional structure visible at
scale~$V_L^{-1}$; models of this type, including double-well loop marks and their finite-type
extensions, are discussed in \cref{sec:double-well-loop-marks}.

\subsection{Examples: Weyl-law universality, double well, and finite-type band model}
\label{subsec:intro-examples}

We illustrate the abstract framework through two complementary families of examples, treated
in \cref{sec:three-bc-ideal-bose,sec:double-well-loop-marks}.

Consider first the ideal Bose gas in a box $\Lambda_L\subset\mathbb R^d$, $d>2$, at inverse
temperature~$\beta$, with periodic, Dirichlet, or Neumann boundary conditions.  Let
$N_L/V_L\to\rho$.  For all three boundary conditions, the background density is the usual
critical density $\rho_{\mathrm{bg}}=\rho_c(\beta)$, and, after the ground-state shift, the
Weyl law collapses the part of the spectrum seen by cycles with $j\asymp V_L$ to the constant
profile $\phi(x)\equiv 1$.  Thus, for $\rho>\rho_c(\beta)$, the unmarked length bridge is the
Gamma bridge of total mass $\rho-\rho_c(\beta)$.  If
$\ell_1^L\geq \ell_2^L\geq\cdots$ are the ranked cycle lengths, then
\[
  \left(
    \frac{\ell_1^L}{V_L},
    \frac{\ell_2^L}{V_L},
    \ldots
  \right)
  \Longrightarrow
  (\rho-\rho_c(\beta))\,(P_1^\downarrow,P_2^\downarrow,\ldots),
  \qquad
  (P_i^\downarrow)_{i\geq1}\sim \mathrm{PD}(0,1),
\]
independently of the choice of boundary condition.  The discrete random-walk variant discussed
in \cref{sec:three-bc-ideal-bose} belongs to the same class.

The boundary condition does, however, change the marked limit.  For periodic boundary
conditions, macroscopic cycles carry winding numbers and Brownian-bridge geometry
(\cref{cor:periodic-marked-winding-bridge-limit}).  For Dirichlet boundary conditions, the
ground-state transform produces killed-bridge, or taboo-process, marks
(\cref{prop:dirichlet-empirical-local-process-limit,%
cor:dirichlet-empirical-marked-bridge}).  For Neumann boundary conditions, the analogous marks
converge to reflected Brownian motion
(\cref{prop:neumann-empirical-local-process-limit,%
cor:neumann-empirical-marked-bridge}).  The three models therefore share the same unmarked
macroscopic length law but have different marked condensates.

The second family of examples arises when low-energy spectral splittings are not washed out at
the $V_L^{-1}$ scale.  The simplest instance is the critical symmetric double well.  Let
$\Delta_L$ be the splitting between the two lowest one-particle energies and assume
$V_L\Delta_L\to\gamma\in(0,\infty)$.  After subtracting the ground-state energy, both levels
contribute to cycles of length $j\asymp V_L$, and the scalar profile becomes
\[
  \phi_\gamma(x)=1+e^{-\beta\gamma x}.
\]
Since $\phi_\gamma$ is not constant, the unmarked bridge is not Gamma and the ranked lengths
are not governed by $\mathrm{PD}(0,1)$.

Different choices of marks on the same double-well model answer different questions.  If the
aim is to distinguish the two energy levels, one uses the spectral label mark on $\{0,1\}$,
with limiting kernel $\eta_x^{\mathrm{sp}}=\delta_0+e^{-\beta\gamma x}\,\delta_1$.  If the
aim is to study metastable tunnelling between the two wells, one instead uses a closed
two-state well-loop mark describing the effective inter-well motion on the rescaled time
interval.  The two marked models share the same scalar profile~$\phi_\gamma$ but have different
mark kernels.  The well-loop bridge is proved in \cref{cor:double-well-well-loop-pk-limit},
while the spectral-label version is a special case of the finite-type result below.

More generally, the framework applies whenever a finite number of low-energy modes remain
visible, as in a multi-well potential or a system with finitely many internal states.  The
finite-type band extension replaces the doublet by $Q$ visible types with
multiplicities~$\theta_r>0$ and rescaled energies~$\lambda_r$.  For the spectral-label mark
$r\in\{1,\ldots,Q\}$, the limiting kernel and scalar profile are
\[
  \eta_x^{\mathrm{ft}}
  =
  \sum_{r=1}^Q
  \theta_r e^{-\beta\lambda_r x}\,\delta_r,
  \qquad
  \phi_{\mathrm{ft}}(x)
  =
  \sum_{r=1}^Q
  \theta_r e^{-\beta\lambda_r x}.
\]
Given a cycle of macroscopic length~$x$, the limiting distribution of its label is
\[
  \mathbb P(r\mid x)
  =
  \frac{\theta_r e^{-\beta\lambda_r x}}
       {\phi_{\mathrm{ft}}(x)}.
\]
The canonical marked bridge is proved in \cref{cor:finite-type-spectral-label-pk-limit}.  When
$\phi_{\mathrm{ft}}$ is constant, the length marginal reduces to a Gamma bridge; otherwise it
is a Poisson--Kingman bridge whose ranked lengths are generally not Poisson--Dirichlet.

\subsection{Related work on ODLRO and loop configurations}
\label{subsec:intro-related-work}

The functional-integral expansion underlying the cycle picture was developed systematically by
Ginibre~\cite{Ginibre1971} and supplies the Brownian-loop and cycle-weight representations used
below.  The relation between long cycles and the spectral order parameter is subtle rather than
automatic.  Ueltschi~\cite{UeltschiODLRO2006} related the density in infinite cycles to ODLRO
and emphasised that the two need not coincide without additional information on the long-loop
kernel.  For the ideal gas, Chevallier and Krauth~\cite{ChevallierKrauth2007} identified the
ODLRO parameter with the mass in cycles whose length is large compared with the diffusive
scale.  Benfatto, Cassandro, Merola and Presutti~\cite{BenfattoCassandroMerolaPresutti2005}
obtained precise asymptotics for canonical loop occupations in a mean-field Bose gas and
identified the mass in infinite loops with the condensed mass.

The closest probabilistic starting points for the present analysis are the works of K\"onig,
Vogel and Zass~\cite{KonigVogelZass2025} and of Bai, K\"onig and
Vogel~\cite{BaiKonigVogel2025}.  For the canonical free Bose gas, K\"onig, Vogel and Zass
developed a Feynman--Kac and Poisson-point-process proof of ODLRO under periodic and diffusive
boundary conditions.  Above the critical density, they established strong concentration of the
short-loop particle density, identified the complementary long-loop mass with the condensate
mass, and proved the $\mathrm{PD}(0,1)$ limit for the ranked long-loop lengths.  Bai, K\"onig
and Vogel adapted this approach to a non-interacting mean-field trapped gas with kinetic
prefactor~$a_N$.  In terms of $\chi=\lim_{N\to\infty}Na_N^{d/2}$ and the trap-dependent
threshold~$\rho_w$, they proved ODLRO for $\chi>\rho_w$ and its absence for $\chi<\rho_w$;
in the former regime they also identified the condensate fraction and the
$\mathrm{PD}(0,1)$ law of the normalised long-loop lengths.

A complementary line of research studies infinite-volume and interacting loop configurations.
Adams, Collevecchio and K\"onig~\cite{AdamsCollevecchioKonig2011} represented the interacting
many-particle system by an ensemble of Brownian bridges organised into cycles and derived a
variational formula via a marked Poisson point process.  Adams and
Vogel~\cite{AdamsVogel2020} introduced Bosonic loop measures and their space--time random-walk
approximation, while Armend\'ariz, Ferrari and
Yuhjtman~\cite{ArmendarizFerrariYuhjtman2021} constructed an infinite-volume Gaussian random
permutation from a Gaussian loop soup and, above criticality, Gaussian random interlacements;
its point marginal is the boson point process.  Vogel~\cite{Vogel2023Interlacements} proved that a supercritical
finite-volume Bose soup conditioned on its density converges locally to the superposition of
the critical loop soup and random interlacements with intensity equal to the excess density.
Dickson and Vogel~\cite{DicksonVogel2024} identified an interacting limiting loop measure with
a random-interlacement component above a shifted critical density, and Bai and
Vogel~\cite{BaiVogel2025} proved the existence of Gibbs measures for the Feynman representation
with non-negative interaction.  At criticality, Vogel~\cite{Vogel2026Critical} showed that the scale and
occurrence of large loops depend more delicately on the geometry and on boundary heat-kernel
coefficients.

The present paper develops this probabilistic cycle framework in a complementary direction.
The results of K\"onig, Vogel and Zass~\cite{KonigVogelZass2025} and Bai, K\"onig and
Vogel~\cite{BaiKonigVogel2025} provide the essential one-component benchmarks: they determine
ODLRO, the total long-loop mass and the ranked unmarked length law in their respective models.
Here the limiting object is instead the entire canonical point process of cycles with
$j\asymp V_L$, together with marks in a general Polish space.  The marked kernel records
spatial, geometric or spectral information that is lost after one retains only the ranked
lengths, while the visible spectral profile allows several low-energy components to survive on
the $V_L^{-1}$ scale.

In the one-visible-level box examples also covered by K\"onig, Vogel and
Zass~\cite{KonigVogelZass2025}, forgetting the marks gives the same classical
$\mathrm{PD}(0,1)$ limit, providing a consistency check on the abstract theorem.  Keeping the
marks distinguishes winding, killed-bridge and reflected-bridge
geometries even when the length marginal is unchanged.  With several visible spectral
components, the framework further yields length-dependent type mixtures and genuinely
non-Gamma bridges; the critical double-well model is the simplest example.  Thus the
contribution is a joint marked-process limit and a spectral classification of its universality
classes, extending the classical unmarked picture without replacing the ODLRO and long-loop
results on which it builds.

We use the term ``Poisson--Kingman bridge'' for the resulting conditioned marked Poisson
point process because ranking its conditioned jumps gives a Poisson--Kingman mass partition.
In the one-level case it reduces to the familiar Gamma bridge and $\mathrm{PD}(0,1)$ law,
whereas distinct visible spectral rates lead to the non-Gamma bridges described above.
\subsection*{Outline of the paper}

In \cref{sec:finite-volume-model} we introduce the finite-volume cycle representation,
including the canonical cycle expansion, the effective and background decomposition, the
one-particle Hamiltonian, and the marked canonical point process.  In
\cref{sec:limit-objects-main-results} we state the abstract assumptions and the main results,
culminating in the canonical marked Poisson--Kingman bridge limit and its basic consequences.
In \cref{sec:three-bc-ideal-bose} we apply the general theorem to the ideal Bose gas with
periodic, Dirichlet, and Neumann boundary conditions, showing that the same
Poisson--Dirichlet length law coexists with different marked limits depending on the boundary
condition.  In \cref{sec:double-well-loop-marks} we study visible finite-type extensions and
double-well loop marks, giving explicit examples with a non-constant visible profile and a
Poisson--Kingman, rather than Gamma, bridge.  Finally, \cref{sec:proofs-canonical-framework}
contains the proofs of the abstract canonical framework, including the Poisson representation,
effective and background splitting, unconditioned convergence, local limit estimates, and
bridge convergence.

%% file: 02-model-cycle-representation-path-marks.tex
\section{Finite-volume cycle representation and marked point processes}
\label{sec:finite-volume-model}

This section defines the finite-volume model used throughout the paper.  We begin by
describing the spatial domain, Hamiltonian, and spectral quantities of a single particle and recall the canonical cycle
expansion of the ideal Bose gas.  The cycle-count law is then lifted to a
marked point process by realizing each counted cycle as an abstract atom with a
macroscopic length, a mark, and an auxiliary time coordinate.  Finally, we introduce
an effective/background decomposition of the one-cycle measure; this separates
the part that will remain visible in the macroscopic marked limit from the
part that contributes only through a background density.

\subsection{Finite-volume Bose gas model}
\label{subsec:finite-volume-one-particle-data}

Let \(\Lambda_L\) be a finite-volume spatial domain and let
\[
    V_L := |\Lambda_L|
\]
denote its volume.  In continuum models, \(\Lambda_L\) may be a scaled bounded
domain or a flat torus; in lattice models, \(\Lambda_L\) is a finite set and
\(V_L\) denotes the number of sites.  We assume that \(V_L \to \infty\) as
\(L \to \infty\).
The inverse temperature is fixed and denoted by \(\beta > 0\).  The canonical
particle number is \(N_L \in \mathbb{N}\), and
\[
    \rho_L := \frac{N_L}{V_L}
\]
is the finite-volume particle density.

The one-particle Hilbert space is denoted by \(\mathcal{H}_L\).  Let \(H_L\)
be a self-adjoint one-particle Hamiltonian on \(\mathcal{H}_L\).  We assume
that \(H_L\) is bounded from below, has purely discrete spectrum, and that
\(e^{-tH_L}\) is trace class for every \(t>0\).  Its eigenvalues, counted with
multiplicity, are written as
\[
    E_{0,L} \leq E_{1,L} \leq E_{2,L} \leq \cdots .
\]

We shift the ground-state energy to zero by setting
\[
    K_L := H_L - E_{0,L}.
\]
This shift only multiplies the \(N\)-particle partition function by a common
factor and hence does not change the canonical probability law.  The
eigenvalues of \(K_L\) are
\[
    \varepsilon_{i,L} := E_{i,L} - E_{0,L},
    \qquad i \geq 0,
\]
so that \(\varepsilon_{0,L}=0\).  We also use the rescaled eigenvalues
\[
    \lambda_{i,L} := V_L \varepsilon_{i,L}.
\]
The shifted heat semigroup \(e^{-tK_L}\) remains trace class for every
\(t>0\).

\subsection{Canonical cycle expansion}
\label{subsec:canonical-cycle-expansion}

For each cycle length \(j \geq 1\), define the finite-volume one-cycle trace
\[
    q_{L,j}
    :=
    \operatorname{Tr}_{\mathcal{H}_L}\!\left(e^{-\beta j K_L}\right).
\]
Hence,
\[
    q_{L,j}
    =
    \sum_{i \geq 0} e^{-\beta j \varepsilon_{i,L}}
    =
    \sum_{i \geq 0} e^{-\beta (j/V_L)\lambda_{i,L}}.
\]
For notational convenience, set
\[
    a_{L,j} := \frac{q_{L,j}}{j}.
\]

For \(N \in \mathbb{N}\), the canonical \(N\)-particle partition function
associated with the shifted Hamiltonian \(K_L\) is
\[
    Z_{L,N}
    =
    \sum_{\substack{(n_j)_{j \geq 1} \\
                    n_j \in \mathbb{N},\;
                    \sum_{j \geq 1} j n_j = N}}
    \prod_{j \geq 1}
    \frac{a_{L,j}^{\,n_j}}{n_j!}.
\]
Here \(n_j\) is the number of cycles of length \(j\), and the constraint
\(\sum_{j\geq1} j n_j=N\) fixes the total particle number.

The corresponding canonical law on cycle-count configurations is
\begin{equation}\label{pcan}
    \mathbb{P}_{L,N}^{\mathrm{can}}
    \bigl((n_j)_{j \geq 1}\bigr)
    =
    \frac{1}{Z_{L,N}}
    \prod_{j \geq 1}
    \frac{a_{L,j}^{\,n_j}}{n_j!},
\end{equation}
for all non-negative integer sequences satisfying
\(\sum_{j\geq1} j n_j=N\).  When \(N=N_L\), we write
\[
    \mathbb{P}_{L}^{\mathrm{can}}
    :=
    \mathbb{P}_{L,N_L}^{\mathrm{can}}.
\]

We shall often speak about individual cycles rather than only their counts.
 Let \({S}_N\) be the symmetric group on
\(\{1,\ldots,N\}\).  For \(\pi\in{S}_N\), denote by
\(\mathcal{C}(\pi)\) the set of cycles in the disjoint-cycle decomposition of
\(\pi\), and write \(|c|\) for the length of a cycle
\(c\in\mathcal{C}(\pi)\).  The cycle counts associated with \(\pi\) are
\[
    n_j(\pi)
    :=
    \#\{c\in\mathcal{C}(\pi): |c|=j\},
    \qquad j\geq1.
\]

The cycle-count law above is the push-forward of the following probability
measure on \({S}_N\):
\begin{equation}\label{probperm}
    \mathbb{P}_{L,N}^{\mathrm{perm}}(\pi)
    =
    \frac{1}{N!\,Z_{L,N}}
    \prod_{c\in\mathcal{C}(\pi)} q_{L,|c|}.
\end{equation}
Indeed, the number of permutations with cycle counts
\((n_j)_{j\geq1}\) is
\[
    \frac{N!}{\prod_{j\geq1} j^{n_j} n_j!},
\]
and hence the induced law of \((n_j(\pi))_{j\geq1}\) is exactly $ \mathbb{P}_{L,N}^{\mathrm{can}}$~\eqref{pcan}.
Thus, whenever we refer to an individual cycle below, we mean a cycle
\(c\in\mathcal{C}(\pi)\) for a permutation \(\pi\) sampled from
\(\mathbb{P}_{L,N}^{\mathrm{perm}}\).  Passing from permutations to cycle
counts forgets the particle labels and keeps only the lengths of these cycles.

\subsection{Marked cycle measures and canonical point process}
\label{subsec:marked-canonical-process}

We now attach marks to the abstract cycles.  Let \(\mathsf{M}\) be a Polish
space.  For each \(L\) and \(j\geq1\), let \(\mu_{L,j}\) be a finite positive
Borel measure on \(\mathsf{M}\) with total mass
\[
    \mu_{L,j}(\mathsf{M}) = q_{L,j}.
\]
We call \(\mu_{L,j}\) the marked one-cycle measure of length \(j\).  If
\(q_{L,j}>0\), write
\[
    J_{L,j}(dm) := \frac{\mu_{L,j}(dm)}{q_{L,j}}
\]
for the normalized mark law; if \(q_{L,j}=0\), the choice of \(J_{L,j}\) is
irrelevant.
The mark space \(\mathsf{M}\) is model-dependent.  It may encode path-valued
information, winding numbers, boundary data, spectral labels, or well
histories.  The examples in~\cref{sec:three-bc-ideal-bose} and~\cref{sec:double-well-loop-marks} specify
\(\mathsf{M}\) and \(\mu_{L,j}\) concretely.  At this stage only
the normalization \(\mu_{L,j}(\mathsf{M})=q_{L,j}\) is used.

We next define the canonical marked cycle point process.  Fix
\(N\in\mathbb{N}\), and sample a permutation
\(\pi\in{S}_N\) with law~\eqref{probperm}.
Conditionally on \(\pi\), each cycle \(c\in\mathcal{C}(\pi)\) receives an
independent mark
\[
    m_c \sim J_{L,|c|}.
\]
We also attach to each cycle an independent auxiliary coordinate
\[
    U_c \sim \operatorname{Unif}[0,1].
\]
This coordinate has no physical meaning; it is only a device for separating
cycles which may have the same length and the same mark when they are
represented as atoms of a point process.
Set
\[
    \mathsf{E}:=[0,1]\times(0,\infty)\times\mathsf{M}.
\]
The canonical marked cycle point process is the random point measure on
\(\mathsf{E}\) defined by
\[
    \Xi_{L,N}
    :=
    \sum_{c\in\mathcal{C}(\pi)}
    \delta_{\bigl(U_c,\; |c|/V_L,\; m_c\bigr)}.
\]
Thus each permutation cycle contributes one atom whose second coordinate is
the macroscopic cycle length \(|c|/V_L\).
If
$   n_j=n_j(\pi),$
then the cycles of length \(j\) may be enumerated, purely for notational
convenience, as $   c_{j,1},\ldots,c_{j,n_j}$.
With $   x_{L,j}:=\frac{j}{V_L}$,
the same point process can be written as
\[
    \Xi_{L,N}
    =
    \sum_{j\geq1}\sum_{\ell=1}^{n_j}
    \delta_{\bigl(U_{j,\ell},\; x_{L,j},\; m_{j,\ell}\bigr)},
\]
where
$   U_{j,\ell}\sim \operatorname{Unif}[0,1]$ and
   $ m_{j,\ell}\sim J_{L,j}$
independently over all enumerated cycles.  

For a non-negative measurable function
\(F:\mathsf{E}\to[0,\infty)\), write
\[
    \langle F,\Xi_{L,N}\rangle
    :=
    \int_{\mathsf{E}} F(u,x,m)\,\Xi_{L,N}(du,dx,dm).
\]
Since the cycles of a permutation of \(N\) labels partition
\(\{1,\ldots,N\}\), one has
$   \sum_{c\in\mathcal{C}(\pi)} |c| = N$.
Therefore the canonical particle-number constraint becomes
\[
    V_L \int_{\mathsf{E}} x\,\Xi_{L,N}(du,dx,dm)=N
\]
almost surely.  For the prescribed particle number \(N_L\), we write
$   \Xi_L := \Xi_{L,N_L}$.
By a harmless abuse of notation, we use
\(\mathbb{P}_{L,N}^{\mathrm{can}}\) also for the enlarged law that includes
the sampled permutation, the marks, and the auxiliary coordinates, whenever
only the resulting marked cycle process is relevant.

\subsection{Effective and background parts}
\label{subsec:finite-volume-effective-background}

The preceding construction treats all cycles in the same
way.  In the thermodynamic limit, however, different parts of the one-cycle
measure $\mu_{L,\cdot}$ may play different roles.  The part we call effective is the part
whose marked macroscopic cycles are retained in the limiting point process;
the background part is allowed to carry a non-negligible amount of mass, but
its contribution will  be controlled only through its total particle
density.  This distinction is useful, for example, when a low-energy or
visible part carries the macroscopic random atoms, while the remaining modes
produce a deterministic density shift. See the examples in~\cref{sec:three-bc-ideal-bose} and~\cref{sec:double-well-loop-marks} for the motivation of this decomposition.

For each \(L\) and \(j \geq 1\), assume that the marked
one-cycle measure admits a decomposition into positive measures
\[
    \mu_{L,j}
    =
    \mu_{L,j}^{\mathrm{eff}}
    +
    \mu_{L,j}^{\mathrm{bg}},
\]
where both terms are finite positive Borel measures on \(\mathsf{M}\).
Define the corresponding total masses by
$   q_{L,j}^{\mathrm{eff}}
    :=
    \mu_{L,j}^{\mathrm{eff}}(\mathsf{M})$,
$    q_{L,j}^{\mathrm{bg}}
    :=
    \mu_{L,j}^{\mathrm{bg}}(\mathsf{M})$.
    Then \(q_{L,j} = q_{L,j}^{\mathrm{eff}} + q_{L,j}^{\mathrm{bg}}\). No mutual singularity between \(\mu_{L,j}^{\mathrm{eff}}\) and \(\mu_{L,j}^{\mathrm{bg}}\) is assumed. The required assumptions will be stated in~\cref{subsec:main-assumptions}.
    
When \(q_{L,j}^{\mathrm{eff}} > 0\), define
\[
    J_{L,j}^{\mathrm{eff}}(dm)
    :=
    \frac{\mu_{L,j}^{\mathrm{eff}}(dm)}{q_{L,j}^{\mathrm{eff}}}\,,
\]
and when \(q_{L,j}^{\mathrm{bg}} > 0\), define
\[
    J_{L,j}^{\mathrm{bg}}(dm)
    :=
    \frac{\mu_{L,j}^{\mathrm{bg}}(dm)}{q_{L,j}^{\mathrm{bg}}}\,.
\]
If the corresponding mass is zero, the choice of the kernel is irrelevant.

The total mark kernel decomposes as the mixture
\[
    J_{L,j}
    =
    \frac{q_{L,j}^{\mathrm{eff}}}{q_{L,j}}\,
    J_{L,j}^{\mathrm{eff}}
    \;+\;
    \frac{q_{L,j}^{\mathrm{bg}}}{q_{L,j}}\,
    J_{L,j}^{\mathrm{bg}},
    \qquad q_{L,j} > 0.
\]
Thus the canonical marked process may be realized by first assigning to each
cycle of length \(j\) a part label
\(\sigma \in \{\mathrm{eff}, \mathrm{bg}\}\)
with probabilities
\[
    \mathbb{P}(\sigma = \mathrm{eff})
    =
    \frac{q_{L,j}^{\mathrm{eff}}}{q_{L,j}}\,,
    \qquad
    \mathbb{P}(\sigma = \mathrm{bg})
    =
    \frac{q_{L,j}^{\mathrm{bg}}}{q_{L,j}}\,,
\]
and then sampling the mark from \(J_{L,j}^{\mathrm{eff}}\) or
\(J_{L,j}^{\mathrm{bg}}\) according to the assigned part.  After forgetting
the part label, the marginal mark law is again \(J_{L,j}\).

The canonical marked point process therefore decomposes as
\[
    \Xi_{L,N}
    =
    \Xi_{L,N}^{\mathrm{eff}}
    +
    \Xi_{L,N}^{\mathrm{bg}},
\]
where
\[
    \Xi_{L,N}^{\mathrm{eff}}
    :=
    \sum_{c:\,\sigma(c)=\mathrm{eff}}
    \delta_{\bigl(U_c,\; |c|/V_L,\; m(c)\bigr)}
\]
and
\[
    \Xi_{L,N}^{\mathrm{bg}}
    :=
    \sum_{c:\,\sigma(c)=\mathrm{bg}}
    \delta_{\bigl(U_c,\; |c|/V_L,\; m(c)\bigr)}.
\]

For the prescribed particle number \(N_L\), we write
$   \Xi_{L,\mathrm{eff}} := \Xi_{L,N_L}^{\mathrm{eff}}$,
 $   \Xi_{L,\mathrm{bg}} := \Xi_{L,N_L}^{\mathrm{bg}}$.
The corresponding effective and background particle numbers are
\[
    G_L
    :=
    V_L \int_{\mathsf{E}} x\, \Xi_{L,\mathrm{eff}}(du,dx,dm),
    \qquad
    B_L
    :=
    V_L \int_{\mathsf{E}} x\, \Xi_{L,\mathrm{bg}}(du,dx,dm).
\]
Under the canonical law at particle number \(N_L\),
\(G_L + B_L = N_L\) almost surely.

%% file: 03-limit-objects-main-results.tex

\section{Assumptions, limiting bridge, and main results}
\label{sec:limit-objects-main-results}
This section formulates the precise assumptions on the finite-volume model and
states the main convergence theorem.  We begin by introducing the
length-bounded topology on the space of point measures, which is the natural
framework for processes whose atoms may accumulate near zero macroscopic
length.  We then define the limiting effective kernel and the associated
limiting marked Poisson point process.  The canonical limit is obtained by conditioning this Poisson point process on
its total macroscopic mass; we refer to the resulting conditional law as the
marked Poisson--Kingman bridge. This limit is related to a Poisson--Kingman-type mass partition, see~\cite{PermanPitmanYor1992,Pitman2003PoissonKingman,PitmanYor1997}.
Next, we collect the assumptions on the finite-volume effective and background
parts: the effective part is required to converge, in a marked sense, to
the limiting kernel, while the background part concentrates on a deterministic
density and remains invisible at macroscopic scales.  Finally, we state the
main theorem, which asserts that the canonical marked cycle point process
converges weakly to this marked Poisson--Kingman bridge with total mass equal
to the effective particle density.
\subsection{Length-bounded topology}
\label{subsec:length-bounded-topology}

Let
$$
    \mathsf{E} := [0,1] \times (0,\infty) \times \mathsf{M},
$$
where $\mathsf{M}$ is a Polish mark space.  The second coordinate is always
interpreted as the macroscopic cycle length.  Since the canonical cycle process
may have infinitely many atoms with lengths tending to zero, we do not equip
the space of point measures on $\mathsf{E}$ with the usual vague or weak
topology.  Instead, we use a topology that tests the process only on length
windows bounded away from both zero and infinity.

For $0 < \delta < R < \infty$, set
$$
    \mathsf{E}_{\delta,R} := [0,1] \times [\delta, R] \times \mathsf{M}\,.
$$
Let $\mathcal{N}_\ell(\mathsf{E})$ be the space of Borel point measures
$\xi$ on $\mathsf{E}$ such that
$$
    \xi(\mathsf{E}_{\delta,R}) < \infty
    \qquad \text{for all } 0 < \delta < R < \infty.
$$
We call such measures \emph{length-boundedly finite}.  A measure
$\xi \in \mathcal{N}_\ell(\mathsf{E})$ may have infinitely many atoms with
lengths tending to zero, but it has only finitely many atoms in every fixed
macroscopic length window.

Let $C_b^\ell(\mathsf{E})$ denote the class of bounded continuous functions
$f : \mathsf{E} \to \mathbb{R}$ for which there exist
$0 < \delta < R < \infty$ such that
$$
    f(u,x,m) = 0
    \qquad \text{whenever } x \notin [\delta, R].
$$
The \emph{length-bounded topology} on $\mathcal{N}_\ell(\mathsf{E})$ is the
coarsest topology making all maps
$$
    \xi \longmapsto \langle f, \xi \rangle
    := \int_{\mathsf{E}} f\, d\xi,
    \qquad f \in C_b^\ell(\mathsf{E}),
$$
continuous.  Equivalently, $\xi_n \to \xi$ in
$\mathcal{N}_\ell(\mathsf{E})$ if and only if
$$
    \int_{\mathsf{E}} f\, d\xi_n
    \longrightarrow
    \int_{\mathsf{E}} f\, d\xi
    \qquad \text{for every } f \in C_b^\ell(\mathsf{E}).
$$

This is the boundedly finite random-measure topology associated with the
bornology generated by the length windows $\mathsf{E}_{\delta,R}$.  It is
the natural topology for the present problem because the limiting Poisson
point process is locally finite on each such window, whereas atoms with
microscopic lengths may accumulate near $x = 0$.  We write
$$
    \Xi_L \Longrightarrow \Xi
    \qquad \text{in } \mathcal{N}_\ell(\mathsf{E})
$$
for weak convergence with respect to this topology.  We use the standard
terminology and convergence criteria for random measures and point processes as
in Kallenberg~\cite{Kallenberg2017} and
Daley--Vere-Jones~\cite{DaleyVereJones2008}.

Let $\mathcal{H}$ denote the non-negative cone of $C_b^\ell(\mathsf{E})$:
the class of functions $h : \mathsf{E} \to [0,\infty)$ that are bounded,
continuous, and supported in some length window
$[0,1] \times [\delta_h, R_h] \times \mathsf{M}$.  For
$\xi \in \mathcal{N}_\ell(\mathsf{E})$ and $h \in \mathcal{H}$, write
$$
    \langle h, \xi \rangle
    :=
    \int_{\mathsf{E}} h(u,x,m)\, \xi(du,dx,dm).
$$
Since $h$ is supported on a fixed length window, this integral is finite for
every $\xi \in \mathcal{N}_\ell(\mathsf{E})$.

Convergence of point processes in $\mathcal{N}_\ell(\mathsf{E})$ will be
formulated through convergence of the corresponding laws on this space.  In
particular, Laplace functionals of the form
$$
    \mathbf{E}\,\exp\!\bigl\{-\langle h, \Xi \rangle\bigr\},
    \qquad h \in \mathcal{H},
$$
will be used in the proofs in~\cref{sec:proofs-canonical-framework} to identify the limiting point processes.

\subsection{Limiting effective kernel and Poisson--Kingman bridge}
\label{subsec:limiting-effective-bridge}
We introduce the limit of the marked Point process~$\Xi_L$.
We first define a marked Poisson point process with basic assumptions.
The canonical bridge is obtained by conditioning this process on its total
macroscopic mass.  We use the term marked Poisson--Kingman bridge for this
conditional law.

The limiting
effective part is encoded by a family of finite positive Borel measures
$ \{   \eta_x, x > 0\},$
on the mark space $\mathsf{M}$.  This family gives the limiting marked
one-cycle law at macroscopic length $x$, before canonical conditioning.

\begin{assumption}[Limiting effective kernel]
\label{ass:limiting-effective-kernel}
The family $x \mapsto \eta_x$ satisfies the following conditions.

\begin{enumerate}
\item The map $x \mapsto \eta_x$ is weakly continuous as a map from
$(0,\infty)$ into the space of finite positive Borel measures on
$\mathsf{M}$.  That is, for every $f \in C_b(\mathsf{M})$,
$$
    x \longmapsto \eta_x(f)
    := \int_{\mathsf{M}} f(m)\, \eta_x(dm)
$$
is continuous.

\item Setting
$$
    \phi(x) := \eta_x(\mathsf{M}), \qquad x > 0,
$$
there exists a finite positive measure $\Sigma(d\lambda)$ on
$[0,\infty)$ such that
$$
    \phi(x)
    =
    \int_{[0,\infty)} e^{-\beta x \lambda}\, \Sigma(d\lambda),
    \qquad x > 0.
$$

\item For every $\kappa > 0$,
$$
    \int_0^\infty (1 \wedge x)\, e^{-\kappa x}\,
    \frac{\phi(x)}{x}\, dx < \infty.
$$
\end{enumerate}
\end{assumption}

For $\kappa > 0$, define a measure $\nu^{(\kappa)}$ on $\mathsf{E}$ by
\begin{equation}\label{defnu}
    \nu^{(\kappa)}(du,dx,dm)
    :=
    du\; e^{-\kappa x}\, \frac{dx}{x}\; \eta_x(dm),
\end{equation}
where $du$ denotes Lebesgue measure on $[0,1]$.  For every length window
$\mathsf{E}_{\delta,R}$,
$\nu^{(\kappa)}(\mathsf{E}_{\delta,R})   < \infty.$
Hence \(\nu^{(\kappa)}\) defines a length-boundedly finite intensity measure.

Let
\[
    \Pi^{(\kappa)}
    \sim
    \operatorname{PPP}\bigl(\nu^{(\kappa)}\bigr)
\]
be the marked Poisson point process on \(\mathsf{E}\) with intensity
\(\nu^{(\kappa)}\).  It is an
\(\mathcal{N}_\ell(\mathsf{E})\)-valued random measure.  Its total effective
mass is defined by
\[
    T^{(\kappa)}
    :=
    \int_{\mathsf{E}} x\, \Pi^{(\kappa)}(du,dx,dm).
\]
The integrability condition in~\cref{ass:limiting-effective-kernel}
ensures that \(T^{(\kappa)} < \infty\) almost surely.

We shall condition this limiting Poisson process on the value of its total
mass.  Since this is a conditioning on a continuous random variable, it is
understood through density disintegration.

\begin{assumption}[Density for the limiting effective mass]
\label{ass:limiting-effective-density}
For every \(\kappa > 0\), the random variable \(T^{(\kappa)}\) admits a
continuous density on \((0,\infty)\).  We denote this density by
\(f_0^{(\kappa)}\).  The bridge at a value \(a > 0\) will be used only when
\(f_0^{(\kappa)}(a) > 0\).
\end{assumption}

For \(a > 0\) such that \(f_0^{(\kappa)}(a) > 0\), the \emph{marked
Poisson--Kingman bridge with total mass \(a\)} is the conditional law
\[
    \Pi_a^{\mathrm{br}}
    :=
    \mathcal{L}\!\left(
        \Pi^{(\kappa)}
        \;\middle|\;
        T^{(\kappa)} = a
    \right),
\]
where the conditioning is interpreted through density disintegration of
\(T^{(\kappa)}\).  Equivalently, for every bounded measurable functional
\(\Phi : \mathcal{N}_\ell(\mathsf{E}) \to \mathbb{R}\) and every bounded
measurable function \(g : (0,\infty) \to \mathbb{R}\),
\[
    \mathbf{E}\!\left[
        \Phi(\Pi^{(\kappa)})\, g(T^{(\kappa)})
    \right]
    =
    \int_0^\infty
    g(a)\,
    \mathbf{E}\!\left[
        \Phi(\Pi_a^{\mathrm{br}})
    \right]
    f_0^{(\kappa)}(a)\, da.
\]
This identity defines the bridge as a regular conditional law at density
points of \(T^{(\kappa)}\).

\begin{remark}[Terminology: Poisson--Kingman bridge]
We call this conditional law a marked Poisson--Kingman bridge because, after
forgetting the marks and ranking the atom sizes, it gives the corresponding
Poisson--Kingman-type mass partition~\cite{Pitman2003PoissonKingman}.  In the ideal Bose gas case in~\cref{sec:three-bc-ideal-bose},
this bridge reduces to the classical Gamma bridge, and the ranked jumps have
the Poisson--Dirichlet law.
\end{remark}

\subsection{Assumptions on the finite-volume model}
\label{subsec:main-assumptions}

We now state the assumptions on the finite-volume effective and background
parts.  The effective part is required to converge, on every macroscopic
length window, to the limiting kernel \(x \mapsto \eta_x\).  The background
part may carry a non-zero particle density, but this density is assumed to
be deterministic in the limit and invisible on the macroscopic length scale.

\begin{assumption}[Effective marked trace convergence]
\label{ass:effective-marked-trace-convergence}
For every \(0 < \delta < R < \infty\) and every
\(F \in C_b\bigl([0,1] \times [\delta,R] \times \mathsf{M}\bigr)\),
\[
    \sup_{x \in [\delta,R]}
    \left|
    \int_0^1 \!\int_{\mathsf{M}}
    F(u,x,m)\, \mu_{L,\lfloor xV_L\rfloor}^{\mathrm{eff}}(dm)\, du
    \;-\;
    \int_0^1 \!\int_{\mathsf{M}}
    F(u,x,m)\, \eta_x(dm)\, du
    \right|
    \longrightarrow 0.
\]
Moreover, the effective one-cycle masses are uniformly bounded:
\[
    C_{\mathrm{eff}}
    :=
    \sup_{L \geq 1}\, \sup_{j \geq 1}\,
    q_{L,j}^{\mathrm{eff}}
    < \infty.
\]
\end{assumption}

Taking \(F \equiv 1\) in~\cref{ass:effective-marked-trace-convergence} yields the scalar
effective trace convergence
\begin{equation}\label{cvgq}
    q_{L,\lfloor xV_L\rfloor}^{\mathrm{eff}}
    \longrightarrow
    \phi(x) := \eta_x(\mathsf{M})
\end{equation}
locally uniformly for \(x \in (0,\infty)\).

The next assumption is a verifiable spectral criterion used to derive the
effective local limit theorem in~\cref{sec:proofs-canonical-framework}.

\begin{assumption}[Effective spectral local-limit criterion]
\label{ass:effective-spectral-LLT-criterion}
At least one of the following two conditions holds.

\begin{enumerate}
\item[\(\mathrm{(A)}\)] \textbf{Absolute case.}
There exist finite positive measures
\(\Sigma_L(d\lambda)\), \(L \geq 1\), on \([0,\infty)\) such that, for every
\(L\) and every \(j \geq 1\),
\[
    q_{L,j}^{\mathrm{eff}}
    =
    \int_{[0,\infty)} e^{-\beta(j/V_L)\lambda}\, \Sigma_L(d\lambda).
\]
Write \(\Theta_L := \Sigma_L\bigl([0,\infty)\bigr)\).  We require:
\begin{enumerate}
\item[(i)] \(\sup_{L \geq 1} \Theta_L < \infty\).
\item[(ii)] For every \(0 < \delta < R < \infty\),
\[
    \sup_{x \in [\delta,R]}
    \left|
    \int_{[0,\infty)}
    e^{-\beta(\lfloor xV_L\rfloor/V_L)\lambda}\, \Sigma_L(d\lambda)
    -
    \int_{[0,\infty)}
    e^{-\beta x\lambda}\, \Sigma(d\lambda)
    \right|
    \longrightarrow 0,
\]
where \(\Sigma\) is the measure appearing in~\cref{ass:limiting-effective-kernel}.
\item[(iii)] There exists \(\Theta_* > 1\) such that, for all sufficiently
large \(L\), \(\Theta_L \geq \Theta_*\).
\item[(iv)] For every \(\kappa > 0\),
\[
    \sup_L
    \int_{[0,\infty)}
    \log(1 + \kappa + \beta\lambda)\, \Sigma_L(d\lambda)
    < \infty.
\]
\end{enumerate}

\item[\(\mathrm{(B)}\)] \textbf{Critical finite-type case.}
There exist an integer \(Q \geq 1\), weights
\(\theta_1, \ldots, \theta_Q > 0\), and parameters
\[
    \lambda_{L,r}:=V_L\varepsilon_{L,r} \longrightarrow \lambda_r \in [0,\infty),
    \qquad r = 1, \ldots, Q,
\]
such that, for every \(L\) and every \(j \geq 1\),
\[
    q_{L,j}^{\mathrm{eff}}
    =
    \sum_{r=1}^Q \theta_r\, e^{-\beta \lambda_{L,r}\, j/V_L}.
\]
Set \(\Theta := \sum_{r=1}^Q \theta_r\).  We assume the critical condition
\(\Theta = 1\).  In this case the limiting scalar profile is
\[
    \phi(x)
    =
    \sum_{r=1}^Q \theta_r\, e^{-\beta \lambda_r x}.
\]
\end{enumerate}
\end{assumption}

\begin{remark}[Finite type with \(\Theta > 1\)]
\label{rem:finite-type-theta-greater-than-one-implies-absolute}
The finite-type case with
\(\Theta = \sum_{r=1}^Q \theta_r > 1\)
is already covered by condition~\(\mathrm{(A)}\).  Indeed, set
\[
    \Sigma_L := \sum_{r=1}^Q \theta_r\, \delta_{\lambda_{L,r}},
    \qquad
    \Sigma := \sum_{r=1}^Q \theta_r\, \delta_{\lambda_r}.
\]
Then, for every \(j \geq 1\),
\[
    \sum_{r=1}^Q \theta_r e^{-\beta \lambda_{L,r} j/V_L}
    =
    \int_{[0,\infty)} e^{-\beta(j/V_L)\lambda}\,\Sigma_L(d\lambda).
\]
Moreover,
\(\Theta_L=\Sigma_L([0,\infty))=\Theta<\infty\), and since
\(\Theta>1\), one may choose \(1<\Theta_*<\Theta\).  Thus
conditions~\(\mathrm{(i)}\) and~\(\mathrm{(iii)}\) hold immediately.

It remains to check~\(\mathrm{(ii)}\) and~\(\mathrm{(iv)}\).  Fix
\(0<\delta<R<\infty\) and put
\[
    a_L(x):=\frac{\lfloor xV_L\rfloor}{V_L}.
\]
Then \(\sup_{x\in[\delta,R]}|a_L(x)-x|\leq V_L^{-1}\).  Hence
\[
\begin{aligned}
&\sup_{x\in[\delta,R]}
\left|
    \int e^{-\beta a_L(x)\lambda}\,\Sigma_L(d\lambda)
    -
    \int e^{-\beta x\lambda}\,\Sigma(d\lambda)
\right|  \\
&\qquad \leq
\sum_{r=1}^Q \theta_r
\sup_{x\in[\delta,R]}
\left|
    e^{-\beta a_L(x)\lambda_{L,r}}
    -
    e^{-\beta x\lambda_r}
\right|.
\end{aligned}
\]
Since \(\lambda_{L,r}\to\lambda_r\), each sequence
\((\lambda_{L,r})_L\) is bounded; write
\(M_r:=\sup_L\lambda_{L,r}<\infty\).  Using
\(|e^{-u}-e^{-v}|\leq |u-v|\) for \(u,v\geq0\),
\[
\begin{aligned}
\sup_{x\in[\delta,R]}
\left|
    e^{-\beta a_L(x)\lambda_{L,r}}
    -
    e^{-\beta x\lambda_r}
\right|
&\leq
\beta \sup_{x\in[\delta,R]}
\left|a_L(x)\lambda_{L,r}-x\lambda_r\right|  \\
&\leq
\beta\left(\frac{M_r}{V_L}
+
R|\lambda_{L,r}-\lambda_r|\right)
\longrightarrow 0 .
\end{aligned}
\]
Since \(Q<\infty\), summing over \(r\) gives condition~\(\mathrm{(ii)}\).

Finally, for every \(\kappa>0\),
\[
\begin{aligned}
\int_{[0,\infty)}
\log(1+\kappa+\beta\lambda)\,\Sigma_L(d\lambda)
&=
\sum_{r=1}^Q
\theta_r
\log(1+\kappa+\beta\lambda_{L,r})  \\
&\leq
\sum_{r=1}^Q
\theta_r
\log(1+\kappa+\beta M_r)
<\infty .
\end{aligned}
\]
The bound is independent of \(L\), so condition~\(\mathrm{(iv)}\) follows.
\end{remark}

We next state the assumption on the background part, formulated solely in
terms of the finite-volume background traces \(q_{L,j}^{\mathrm{bg}}\).

\begin{assumption}[Background density concentration]
\label{ass:background-density-concentration}
There exists a constant \(\rho_{\mathrm{bg}} \in [0,\infty)\)
and some \(\kappa > 0\), such that
\[
    m_{L,\mathrm{bg}}^{(\kappa)}
    :=
    \frac{1}{V_L}
    \sum_{j \geq 1} e^{-\kappa j/V_L}\, q_{L,j}^{\mathrm{bg}}
    \longrightarrow \rho_{\mathrm{bg}},
\]
and
\[
    v_{L,\mathrm{bg}}^{(\kappa)}
    :=
    \frac{1}{V_L^2}
    \sum_{j \geq 1} j\, e^{-\kappa j/V_L}\, q_{L,j}^{\mathrm{bg}}
    \longrightarrow 0.
\]
\end{assumption}
Finally, we specify the canonical density regime.

\begin{assumption}[Canonical density regime]
\label{ass:canonical-density-regime}
The canonical particle numbers satisfy
\[
    \frac{N_L}{V_L} \longrightarrow \rho
\]
for some \(\rho \in (0,\infty)\).  We assume that
\(\rho > \rho_{\mathrm{bg}}\).  Define the effective density by
\[
    \rho_{\mathrm{eff}} := \rho - \rho_{\mathrm{bg}} > 0.
\]
For the chosen \(\kappa > 0\), we also assume
\(f_0^{(\kappa)}(\rho_{\mathrm{eff}}) > 0\), where \(f_0^{(\kappa)}\) is the
density of the limiting effective mass \(T^{(\kappa)}\) from~\cref{ass:limiting-effective-density}.
\end{assumption}

\subsection{Main theorem}
\label{subsec:main-theorem}

We now state the canonical bridge limit.  The theorem asserts that, once the
deterministic background density is removed, the macroscopic effective cycles
converge to a marked Poisson--Kingman bridge.

\begin{theorem}[Canonical marked Poisson--Kingman bridge limit]
\label{thm:main-canonical-bridge-limit}
  Assume~\cref{ass:limiting-effective-kernel},
\cref{ass:limiting-effective-density},
\cref{ass:effective-marked-trace-convergence},
\cref{ass:effective-spectral-LLT-criterion}, and
\cref{ass:canonical-density-regime}. Fix \(\kappa > 0\) satisfies~\cref{ass:background-density-concentration}.  Let
\(\Xi_L = \Xi_{L,N_L}\)
be the canonical marked cycle point process at particle number \(N_L\).  Then,
under \(\mathbb{P}_{L,N_L}^{\mathrm{can}}\),
\[
    \Xi_L
    \Longrightarrow
    \Pi_{\rho_{\mathrm{eff}}}^{\mathrm{br}}
    \qquad
    \text{in } \mathcal{N}_\ell(\mathsf{E}).
\]
Here
\[
    \Pi_{\rho_{\mathrm{eff}}}^{\mathrm{br}}
    =
    \mathcal{L}\!\left(
        \Pi^{(\kappa)}
        \;\middle|\;
        T^{(\kappa)} = \rho_{\mathrm{eff}}
    \right)
\]
is the marked Poisson--Kingman bridge defined
in~\cref{subsec:limiting-effective-bridge}.

Moreover, the background part is invisible in every fixed macroscopic length
window: for every \(0 < \delta < R < \infty\),
\[
    \mathbb{P}_{L,N_L}^{\mathrm{can}}\!\left(
        \Xi_{L,\mathrm{bg}}
        \bigl([0,1] \times [\delta,R] \times \mathsf{M}\bigr) > 0
    \right)
    \longrightarrow 0.
\]
Thus the background contributes the deterministic density
\(\rho_{\mathrm{bg}}\), while the effective part carries the remaining
density \(\rho_{\mathrm{eff}}\): under
\(\mathbb{P}_{L,N_L}^{\mathrm{can}}\),
\[
    \left(
        \frac{B_L}{V_L},\;
        \frac{G_L}{V_L}
    \right)
    \longrightarrow
    \left(
        \rho_{\mathrm{bg}},\;
        \rho_{\mathrm{eff}}
    \right)
    \qquad
    \text{in probability}.
\]
\end{theorem}

\begin{remark}[Independence of \(\kappa\)]
\label{rem:kappa-auxiliary-main}
The canonical law \(\mathbb{P}_{L,N_L}^{\mathrm{can}}\) does not depend on
\(\kappa\).  The parameter \(\kappa\) enters only through the grand-canonical
Poisson representation used in the proof.  The bridge law at a fixed total
mass \(\rho_{\mathrm{eff}}\) is likewise independent of \(\kappa\): changing
\(\kappa\) exponentially tilts the Poisson intensity by the factor
\(e^{-\kappa x}\), but after conditioning on
\(T^{(\kappa)} = \rho_{\mathrm{eff}}\) this tilt becomes a multiplicative
constant in the conditional density and cancels upon normalization.  The
precise verifications are given in~\cref{lem:canonical-conditioning}
and~\cref{lem:bridge-kappa-independence}.
\end{remark}

\begin{corollary}[Convergence of ranked macroscopic cycle lengths]
\label{cor:ranked-effective-length-convergence}
Assume the assumptions of~\cref{thm:main-canonical-bridge-limit} hold.
Let
\[
    \ell^L
    =
    (\ell_1^L,\ell_2^L,\ldots)
\]
be the decreasing rearrangement of the length coordinates of  $\Xi^{\mathrm{eff}}$. 
Let
\[
    \ell
    =
    (\ell_1,\ell_2,\ldots)
\]
be the decreasing rearrangement of the length coordinates of the limiting
bridge \(\Pi_{\rho_{\mathrm{eff}}}^{\mathrm{br}}\).
Assume that \(\Pi_{\rho_{\mathrm{eff}}}^{\mathrm{br}}\) has almost surely no ties in the length
coordinate.  Then
\[
    \ell^L
    \Longrightarrow
    \ell
    \qquad
    \text{in } \ell^1_\downarrow .
\]
where
\[
    \ell^1_\downarrow
    :=
    \left\{
        x=(x_i)_{i\ge1}\in[0,\infty)^{\mathbb N}:
        x_1\ge x_2\ge\cdots\ge0,\ 
        \|x\|_1:=\sum_{i=1}^\infty x_i<\infty
    \right\}
\]
is equipped with the usual \(\ell^1\)-metric.
Finite ranked configurations are identified with elements of
\(\ell^1_\downarrow\) by appending zeros.
In particular, if the limiting length marginal of the Poisson intensity is
diffuse, then the no-ties assumption above is satisfied.
\end{corollary}

%% file: 04-poissonization-mode-decomposition.tex
\section{The ideal Bose gas under three boundary conditions}
\label{sec:three-bc-ideal-bose}

This section demonstrates that the abstract framework developed in~\cref{sec:finite-volume-model}--\cref{sec:limit-objects-main-results}
applies to the spatial ideal Bose gas, the canonical physical example of
Bose--Einstein condensation.  We treat three standard boundary conditions
(periodic, Dirichlet, and Neumann) and verify, in each case, the
full set of assumptions required by~\cref{thm:main-canonical-bridge-limit}.  Because the effective
part reduces to a single zero-energy mode, the scalar limit are identical
across all three boundary conditions, leading to the same
Poisson--Dirichlet \((0,1)\) law for macroscopic cycle lengths.  The
boundary condition manifests itself only through the mark structure:
periodic cycles carry Gaussian winding, Dirichlet cycles carry
killed-bridge marks with non-uniform roots, and Neumann cycles carry
reflected-bridge marks with uniform roots.  We also talk about  a discrete
random-walk variant that belongs to the same universality class.

\medskip

We consider the ideal Bose gas in the box
\[
    \Lambda_L = (0,L)^d \subset \mathbb{R}^d,
    \qquad
    V_L = L^d,
    \qquad
    d > 2,
\]
under one of three standard boundary conditions:
\[
    b \in \{\mathrm{per}, D, N\},
\]
namely periodic, Dirichlet, and Neumann.  Let \(h_L^b = -\Delta\) on
\(\Lambda_L\) with boundary condition \(b\), and let
\[
    K_L^b := h_L^b - E_{0,L}^b
\]
be the ground-state-shifted one-particle Hamiltonian, as in~\cref{subsec:finite-volume-one-particle-data}.  Its eigenvalues are
\[
    0 = \varepsilon_{0,L}^b
    \leq \varepsilon_{1,L}^b
    \leq \varepsilon_{2,L}^b
    \leq \cdots\,.
\]
\subsection{Spectra and the common length bridge}
\label{subsec:three-bc-scalar}
The spectra of the Laplacian on a rectangular box with periodic,
Dirichlet, and Neumann boundary conditions are standard consequences of
separation of variables; see, for example,~\cite{EvansPDE}. They are summarized as follows
\[
\renewcommand{\arraystretch}{1.45}
\begin{array}{c|c|c|c|c}
\hline
b
&
\text{index set}
&
E_{n,L}^b
&
E_{0,L}^b
&
\varepsilon_{1,L}^b
\\
\hline
\mathrm{per}
&
n \in \mathbb{Z}^d
&
\dfrac{4\pi^2|n|^2}{L^2}
&
0
&
\dfrac{4\pi^2}{L^2}
\\[1.2em]
D
&
n \in \mathbb{N}^d
&
\dfrac{\pi^2|n|^2}{L^2}
&
\dfrac{\pi^2 d}{L^2}
&
\dfrac{3\pi^2}{L^2}
\\[1.2em]
N
&
n \in \mathbb{N}_0^d
&
\dfrac{\pi^2|n|^2}{L^2}
&
0
&
\dfrac{\pi^2}{L^2}
\\
\hline
\end{array}
\]
where $\mathbb N$ denotes the positive integers and $\mathbb N_0$ denotes the non-negative integers.
Since \(d > 2\), the rescaled first excited eigenvalue satisfies
\[
    V_L\, \varepsilon_{1,L}^b \longrightarrow \infty,
    \qquad
    b \in \{\mathrm{per}, D, N\}.
\]
Thus all excited modes leave every bounded spectral window on the
\(V_L^{-1}\)-energy scale.  After the ground-state shift, the only mode
visible on this scale is the shifted ground state itself.

\medskip

In this subsection we verify all assumptions at the \emph{scalar level},
taking the trivial mark space
\[
    \mathsf{M}_0 := \{*\}
\]
with the deterministic mark kernel \(\eta_x(\{*\}) = 1\) for all
\(x > 0\).  This reduces every marked assumption to its scalar content:
the mark integral collapses to the total mass of the mark measure, and
convergence of marked traces becomes convergence~\eqref{cvgq} of the scalar traces
\(q_{L,j}^{b,\mathrm{eff}}\).  The non-trivial mark spaces
\(\mathsf{M}_{\mathrm{per}}\), \(\mathsf{M}_D\), \(\mathsf{M}_N\) and
their associated mark kernels will be treated separately in~\cref{subsec:periodic-marked}--\cref{subsec:DN-empirical-local-process}.

\medskip

We now verify the scalar assumptions of
\cref{thm:main-canonical-bridge-limit}.  For each
\(b\in\{\mathrm{per},D,N\}\), the effective part consists only of the
ground-state mode.  Hence
\[
    q_{L,j}^{b,\mathrm{eff}}=1,\qquad
    \Sigma_L^b=\Sigma^b=\delta_0,\qquad
    \phi_b(x)=1,\quad x>0 .
\]
Then the~\Cref{ass:limiting-effective-kernel} and~\Cref{ass:effective-marked-trace-convergence} are immediate. The critical finite-type condition (B) in~\Cref{ass:effective-spectral-LLT-criterion} holds with
$   Q=1$, $\theta_1=1$, $ \lambda_{L,1}=0$.
Since \(\phi_b\equiv1\), the \(\kappa\)-tilted limiting effective Poisson
process has length intensity
$   e^{-\kappa x}dx/x$.
Therefore, by the Laplace functional of a Poisson point process,
\[
    \mathbf E e^{-sT^{(\kappa)}}
    =
    \exp\left\{
        -\int_0^\infty
        \bigl(1-e^{-sx}\bigr)
        e^{-\kappa x}\frac{dx}{x}
    \right\}  
=
    \frac{\kappa}{\kappa+s},
    \qquad s\ge0 .
\]
Hence \(T^{(\kappa)}\sim\mathrm{Exp}(\kappa)\), and
$   f_0^{(\kappa)}(a)=\kappa e^{-\kappa a}>0$ for  $a>0$. The~\Cref{ass:limiting-effective-density} follows.

\begin{lemma}[Background density concentration]
\label{lem:background-density-concentration}
Assume \(d>2\) and fix \(\kappa\ge0\).  For each
\(b\in\{\mathrm{per},D,N\}\), let
\(\{\varepsilon_{i,L}^b\}_i\) be the shifted one-particle spectrum introduced
above, so that the ground-state energy is \(0\).  Then, under the
\(\kappa\)-tilted grand-canonical measure,
\[
    m_{L,\mathrm{bg}}^{b,(\kappa)}
    \longrightarrow
    \rho_{\mathrm c},
    \qquad
    v_{L,\mathrm{bg}}^{b,(\kappa)}
    \longrightarrow
    0,
    \qquad L\to\infty,
\]
where
\[
    \rho_{\mathrm c}
    :=
    \int_{\mathbb R^d}
    \frac{dp}{(2\pi)^d}
    \frac{1}{e^{\beta |p|^2}-1}.
\]
In particular,
\(\Cref{ass:background-density-concentration}\) holds with
\(\rho_{\mathrm{bg}}=\rho_{\mathrm c}\), independently of the boundary
condition \(b\).
\end{lemma}

\begin{proof}
The background consists of the non-ground-state modes.  Hence
\[
    q_{L,j}^{b,\mathrm{bg}}
    =
    \sum_{i\ge1} e^{-\beta j\varepsilon_{i,L}^b}.
\]
Therefore
\[
    m_{L,\mathrm{bg}}^{b,(\kappa)}
    =
    \frac{1}{V_L}
    \sum_{j\ge1}
    e^{-\kappa j/V_L}q_{L,j}^{b,\mathrm{bg}},
    \qquad
    v_{L,\mathrm{bg}}^{b,(\kappa)}
    =
    \frac{1}{V_L^2}
    \sum_{j\ge1}
    j e^{-\kappa j/V_L}q_{L,j}^{b,\mathrm{bg}} .
\]
Summing the geometric series gives
\[
    m_{L,\mathrm{bg}}^{b,(\kappa)}
    =
    \frac{1}{V_L}
    \sum_{i\ge1}
    \frac{1}{
        e^{\beta\varepsilon_{i,L}^b+\kappa/V_L}-1
    },
\]
and
\[
    v_{L,\mathrm{bg}}^{b,(\kappa)}
    =
    \frac{1}{V_L^2}
    \sum_{i\ge1}
    \frac{
        e^{\beta\varepsilon_{i,L}^b+\kappa/V_L}
    }{
        \left(
            e^{\beta\varepsilon_{i,L}^b+\kappa/V_L}-1
        \right)^2
    } .
\]

We use the standard Weyl estimates for boxes with periodic, Dirichlet and
Neumann boundary conditions.  Since the spectra have been shifted by their
ground-state energies and these shifts are \(O(L^{-2})\), the same Weyl
asymptotics hold for the shifted spectra.  Thus, for every
\(\varphi\in C_c([0,\infty))\),
\[
    \frac{1}{V_L}
    \sum_i \varphi(\varepsilon_{i,L}^b)
    \longrightarrow
    \int_{\mathbb R^d}
    \frac{dp}{(2\pi)^d}\,
    \varphi(|p|^2).
\]
We shall also use the uniform counting bounds
\[
    \#\{i\ge1:0<\varepsilon_{i,L}^b\le r\}
    \le
    C V_L r^{d/2},
    \qquad 0<r\le1,
\]
and
\[
    \#\{i:\varepsilon_{i,L}^b\le r\}
    \le
    C V_L(1+r)^{d/2},
    \qquad r\ge0,
\]
with constants independent of \(L\) and
\(b\in\{\mathrm{per},D,N\}\).  Finally, for the non-ground-state spectrum,
\[
    \varepsilon_{1,L}^b\ge cL^{-2}.
\]
These are standard consequences of the Weyl estimates for boxes, see~\cite{ReedSimonIV} for example..

We first prove the convergence of the mean.  Define
\[
    B_L(x)
    :=
    \frac{1}{e^{\beta x+\kappa/V_L}-1},
    \qquad
    B(x)
    :=
    \frac{1}{e^{\beta x}-1}.
\]
The only singularity of \(B\) is at \(x=0\), so we use a truncation argument.

Fix \(0<\eta<R<\infty\).  Let
\[
    \mu_L^b
    :=
    \frac{1}{V_L}
    \sum_i \delta_{\varepsilon_{i,L}^b}.
\]
The Weyl asymptotics say that \(\mu_L^b\) converges weakly to the measure
\(\mu\) determined by
\[
    \int f(x)\,\mu(dx)
    =
    \int_{\mathbb R^d}
    \frac{dp}{(2\pi)^d} f(|p|^2).
\]
On \([\eta,R]\), \(B_L\to B\) uniformly.  Moreover,
\(\mu_L^b([\eta,R])\) is uniformly bounded by the counting estimate.  Hence
\[
    \int_{[\eta,R]} B_L(x)\,\mu_L^b(dx)
    -
    \int_{[\eta,R]} B(x)\,\mu_L^b(dx)
    \longrightarrow 0 .
\]
Since the limiting Weyl measure has no atoms, the function
\(B\mathbf 1_{[\eta,R]}\) is bounded and \(\mu\)-a.e. continuous.  Therefore
\[
    \frac{1}{V_L}
    \sum_{\eta\le \varepsilon_{i,L}^b\le R}
    B_L(\varepsilon_{i,L}^b)
    \longrightarrow
    \int_{\eta\le |p|^2\le R}
    \frac{dp}{(2\pi)^d}
    \frac{1}{e^{\beta |p|^2}-1}.
\]

It remains to control the tails.  For the high-energy part, since
\(\kappa\ge0\),
\[
    B_L(x)\le C e^{-\beta x/2}
\]
for large \(x\), uniformly in \(L\).  The global counting bound then gives
\[
    \lim_{R\to\infty}
    \limsup_{L\to\infty}
    \frac{1}{V_L}
    \sum_{\varepsilon_{i,L}^b>R}
    B_L(\varepsilon_{i,L}^b)
    =
    0 .
\]
For the low-energy part, using again \(\kappa\ge0\),
\[
    B_L(x)\le B(x)\le \frac{C}{x},
    \qquad x>0.
\]
By decomposing \((0,\eta)\) into dyadic shells and using the low-energy
counting bound,
\[
\begin{aligned}
    \frac{1}{V_L}
    \sum_{0<\varepsilon_{i,L}^b<\eta}
    B_L(\varepsilon_{i,L}^b)
    &\le
    \frac{C}{V_L}
    \sum_{0<\varepsilon_{i,L}^b<\eta}
    \frac{1}{\varepsilon_{i,L}^b}                                      \\
    &\le
    C
    \sum_{n\ge0}
    \frac{(2^{-n}\eta)^{d/2}}{2^{-n-1}\eta}                             \\
    &\le
    C\eta^{d/2-1}.
\end{aligned}
\]
Since \(d>2\), this tends to \(0\) as \(\eta\downarrow0\).  Combining the
convergence on \([\eta,R]\), the high-energy estimate and the low-energy
estimate yields
\[
    m_{L,\mathrm{bg}}^{b,(\kappa)}
    \longrightarrow
    \int_{\mathbb R^d}
    \frac{dp}{(2\pi)^d}
    \frac{1}{e^{\beta |p|^2}-1}
    =
    \rho_{\mathrm c}.
\]

It remains to prove the vanishing of the variance.  Set
\[
    G_L(x)
    :=
    \frac{
        e^{\beta x+\kappa/V_L}
    }{
        \left(e^{\beta x+\kappa/V_L}-1\right)^2
    } .
\]
Then
\[
    v_{L,\mathrm{bg}}^{b,(\kappa)}
    =
    \frac{1}{V_L^2}
    \sum_{i\ge1}
    G_L(\varepsilon_{i,L}^b).
\]

Fix \(\delta>0\).  For \(x\ge\delta\), \(G_L(x)\le C_\delta e^{-\beta x/2}\),
uniformly in \(L\).  Hence the global counting bound implies
\[
    \frac{1}{V_L^2}
    \sum_{\varepsilon_{i,L}^b\ge\delta}
    G_L(\varepsilon_{i,L}^b)
    =
    O(V_L^{-1}).
\]

For the low-energy part, use the elementary estimate
\[
    \frac{e^y}{(e^y-1)^2}\le C y^{-2},
    \qquad y>0.
\]
If \(\kappa>0\), then
\[
    \beta x+\kappa/V_L\ge c(x+V_L^{-1}),
\]
and hence
\[
    G_L(x)\le \frac{C}{(x+V_L^{-1})^2}.
\]
If \(\kappa=0\), the same bound holds on the non-ground-state spectrum.  Indeed,
\(\varepsilon_{i,L}^b\ge cL^{-2}\) for \(i\ge1\), whereas
\(V_L^{-1}=L^{-d}=o(L^{-2})\) since \(d>2\).  Thus
\[
    \varepsilon_{i,L}^b+V_L^{-1}\le C\varepsilon_{i,L}^b,
\]
and consequently
\[
    G_L(\varepsilon_{i,L}^b)
    \le
    \frac{C}{(\varepsilon_{i,L}^b+V_L^{-1})^2}.
\]
Therefore, for all \(\kappa\ge0\),
\[
\begin{aligned}
    \frac{1}{V_L^2}
    \sum_{0<\varepsilon_{i,L}^b<\delta}
    G_L(\varepsilon_{i,L}^b)
    &\le
    \frac{C}{V_L^2}
    \sum_{0<\varepsilon_{i,L}^b<\delta}
    \frac{1}{
        \left(\varepsilon_{i,L}^b+V_L^{-1}\right)^2
    }                                                                  \\
    &\le
    \frac{C}{V_L}
    \int_0^\delta
    \frac{x^{d/2-1}}{(x+V_L^{-1})^2}\,dx .
\end{aligned}
\]
The last inequality follows from the low-energy counting bound, for example
by summation by parts or by a dyadic shell decomposition.  The integral has
the standard estimate
\[
    \frac{1}{V_L}
    \int_0^\delta
    \frac{x^{d/2-1}}{(x+V_L^{-1})^2}\,dx
    =
    \begin{cases}
        O\!\left(V_L^{1-d/2}\right), & 2<d<4,\\[1mm]
        O\!\left((\log V_L)/V_L\right), & d=4,\\[1mm]
        O\!\left(V_L^{-1}\right), & d>4.
    \end{cases}
\]
In all cases this tends to \(0\).  Hence
\[
    v_{L,\mathrm{bg}}^{b,(\kappa)}
    \longrightarrow 0 .
\]
The proof is complete.
\end{proof}

\begin{remark}[Critical and effective densities]
We use the notation \(\rho_c=:\rho_{\mathrm{bg}}\) in this section to emphasize its standard
interpretation as the critical density of the ideal Bose gas, while the
endpoint of the limiting bridge is the excess density
\(a= \rho_{\mathrm{eff}}=\rho-\rho_c\).
\end{remark}

Having verified all assumptions under the trivial mark space
\(\mathsf{M}_0 = \{*\}\), we may apply~\cref{thm:main-canonical-bridge-limit} and~\cref{cor:ranked-effective-length-convergence} directly to obtain the common
unmarked limit.

 Define the unmarked
macroscopic length process
\[
    \Xi_{L,N_L}^{b,\mathrm{len}}
    :=
    \sum_c \delta_{\bigl(U(c),\, |c|/V_L\bigr)}
    \;\in\;
    \mathcal{N}_\ell\bigl([0,1] \times (0,\infty)\bigr),
\]
where the \(U(c)\) are independent uniform labels on \([0,1]\).

\begin{corollary}[Unmarked length bridge]
\label{cor:three-bc-unmarked-length-bridge}
Let \(b \in \{\mathrm{per}, D, N\}\), and assume
\(N_L / V_L \to \rho > \rho_{\mathrm{c}}\).  Then
\[
    \Xi_{L,N_L}^{b,\mathrm{len}}
    \Longrightarrow
    \overline{\Xi}^{\, \rho - \rho_{\mathrm{c}}}
    \qquad
    \text{in }
    \mathcal{N}_\ell\bigl([0,1] \times (0,\infty)\bigr),
\]
where \(\overline{\Xi}^{\, \rho - \rho_{\mathrm{c}}}\) is the Gamma bridge of endpoint \( \rho - \rho_{\mathrm{c}}\): the
Poisson point process $\overline{\Xi}$ on \([0,1] \times (0,\infty)\) with intensity
\(du\, dx/x\), conditioned in the density sense on
\[
    \int_{[0,1] \times (0,\infty)} x\, \overline{\Xi}(du,dx) =  \rho - \rho_{\mathrm{c}}.
\]
Moreover, if \(X_{L,1}^{\downarrow} \geq X_{L,2}^{\downarrow}
\geq \cdots\) are the ranked atoms of the length coordinate of
\(\Xi_{L,N_L}^{b,\mathrm{len}}\), then
\[
    \bigl(X_{L,1}^{\downarrow}, X_{L,2}^{\downarrow}, \ldots\bigr)
    \Longrightarrow
     (\rho - \rho_{\mathrm{c}})\, \bigl(P_1^{\downarrow}, P_2^{\downarrow}, \ldots\bigr),
\]
where \((P_i^{\downarrow})_{i \geq 1} \sim \mathrm{PD}(0,1)\).
\end{corollary}

Thus the unmarked macroscopic length distribution is identical for all
three boundary conditions.  The differences arise only in the marked
processes, which we address in the following subsections.

\subsection{Periodic boundary condition: marked Feynman cycles}
\label{subsec:periodic-marked}

This subsection specialises the finite-volume framework to the ideal Bose gas
with periodic boundary conditions.  We start from the canonical
Feynman--Kac representation on the torus, pass to the cycle decomposition,
disintegrate each cycle into a rooted Brownian loop, and attach to every loop
the geometric marks relevant at the macroscopic scale.  The resulting marked
cycle process is then shown to satisfy the hypotheses of~\cref{thm:main-canonical-bridge-limit}. 
\subsubsection{The finite-volume Feynman--Kac measure and its cycle
  marginal}
\label{sssec:periodic-fk-measure}

Let $\mathbb T_L^d=\mathbb R^d/L\mathbb Z^d$. We write
\[
  g_t(z)=(4\pi t)^{-d/2}
  \exp\!\left\{-\frac{|z|^2}{4t}\right\},
  \qquad z\in\mathbb R^d,
\]
for the free heat kernel.  The periodic heat kernel associated with
\(e^{t\Delta}\) on \(\mathbb T_L^d\) is
\[
  q_t^L(x,y)
  =
  \sum_{w\in\mathbb Z^d} g_t(y-x+Lw).
\]

For \(N\) bosons at inverse temperature \(\beta\), the canonical
Feynman--Kac measure on spatial configurations and permutations is
\begin{equation}\label{eq:periodic-spatial-permutation-measure}
  \mathbb Q_{L,N}(d\mathbf x,\pi)
  =
  \frac{1}{N!\,Z_{L,N}}
  \prod_{i=1}^N q_\beta^L(x_i,x_{\pi(i)})
  \,d\mathbf x,
  \qquad
  \mathbf x\in(\mathbb T_L^d)^N,\quad \pi\in\mathcal S_N ,
\end{equation}
where
\[
  Z_{L,N}
  =
  \frac{1}{N!}
  \sum_{\pi\in\mathcal S_N}
  \int_{(\mathbb T_L^d)^N}
  \prod_{i=1}^N q_\beta^L(x_i,x_{\pi(i)})
  \,dx_1\cdots dx_N
\]
is the canonical partition function. This is the ideal-gas instance of the spatial permutation measure studied in~\cite{BetzUeltschi2009}: the
weight of a permutation is determined by the heat-kernel weights of its
spatial jumps.

Let \(n_j=n_j(\pi)\) be the number of cycles of length \(j\) in \(\pi\).
For a cycle \(c=(i_1\,i_2\,\cdots\,i_j)\), the semigroup property gives
\[
  \int_{(\mathbb T_L^d)^j}
  q_\beta^L(x_{i_1},x_{i_2})
  \cdots
  q_\beta^L(x_{i_j},x_{i_1})
  \,dx_{i_1}\cdots dx_{i_j}
  =
  \int_{\mathbb T_L^d}q_{\beta j}^L(A,A)\,dA .
\]
The corresponding cycle weight is
\begin{equation}\label{defa}
  a_{L,j}
  =
  \frac{1}{j}
  \int_{\mathbb T_L^d} q_{\beta j}^L(A,A)\,dA
  =
  \frac{V_L}{j}\,q_{\beta j}^L(0,0),
\end{equation}
where the second identity follows from translation invariance on the torus.
Grouping permutations according to their cycle counts yields the canonical
cycle-count law
\begin{equation}\label{eq:periodic-canonical-cycle-law}
  \mathbb P_{L,N}(n_1,n_2,\ldots)
  =
  \frac{1}{Z_{L,N}}
  \prod_{j\ge1}\frac{a_{L,j}^{\,n_j}}{n_j!}\,
  \mathbf 1_{\{\sum_{j\ge1}j\,n_j=N\}} .
\end{equation}
 It is the form used in the
cycle-percolation description of the ideal Bose gas in S\"ut\H{o}'s works~\cite{SutoPercolation1993,SutoPercolation2002}. We will introduce the additional
spatial marks below. We claim that it is a disintegration of the same
finite-volume measure, not a change of ensemble.

\subsubsection{Rooted Brownian loops and geometric marks}
\label{sssec:periodic-rooted-loops-marks}

We next disintegrate the finite-volume cycle measure into rooted Brownian
loops and define the geometric marks that will be used in the macroscopic
limit.  For a cycle of length \(j\), let \(W_{A,A}^{L,\beta j}\) denote the
unnormalised Wiener measure on continuous paths
\[
  \omega:[0,\beta j]\to\mathbb T_L^d,
  \qquad
  \omega(0)=\omega(\beta j)=A \in \mathbb T_L^d.
\]
Its total mass is
\[
  W_{A,A}^{L,\beta j}(\Omega)
  =
  q_{\beta j}^L(A,A)
  =
  \sum_{w\in\mathbb Z^d}g_{\beta j}(Lw),
\]
which is independent of \(A\).  
We define the normalised rooted-loop kernel by
\[
  \kappa_{L,j}(dA,d\omega)
  =
  \frac{1}{j\,a_{L,j}}\,
  dA\,W_{A,A}^{L,\beta j}(d\omega).
\]
By the definition of $a_{L,j}$~\eqref{defa}, \(\kappa_{L,j}\) is a probability measure on rooted loops of duration
\(\beta j\).
Given the cycle counts \((n_j)_{j\ge1}\), we attach independently to each
cycle of length \(j\) a rooted loop with law \(\kappa_{L,j}\).  Equivalently,
the rooted loop-gas measure is
\[
  \frac{1}{Z_{L,N}}
  \prod_{j\ge1}
  \frac{1}{n_j!}
  \prod_{r=1}^{n_j}
  \left[
    \frac{1}{j}\,
    dA_{j,r}\,
    W_{A_{j,r},A_{j,r}}^{L,\beta j}(d\omega_{j,r})
  \right]
  \mathbf 1_{\{\sum_{j\ge1}j\,n_j=N\}} .
\]
This is the cycle-by-cycle disintegration of the Feynman--Kac measure
\eqref{eq:periodic-spatial-permutation-measure}.

In order to study the scaling limits of the loops,
we now associate three marks to a rooted loop.  First, if the root is
\(A\in\mathbb T_L^d\), define the rescaled root
\[
  R=A/L\in\mathbb T^d .
\]
Second, lift the periodic loop to a continuous path
\[
  \widetilde\omega:[0,\beta j]\to\mathbb R^d,
  \qquad
  \widetilde\omega(0)=A .
\]
Since the projected path closes on \(\mathbb T_L^d\), there is a unique
winding vector \(W(\omega)\in\mathbb Z^d\) such that
\[
  \widetilde\omega(\beta j)-\widetilde\omega(0)
  =
  L\,W(\omega).
\]
For macroscopic cycles, \(j\) is of order \(V_L\), so we define the winding
endpoint on the scale \(\sqrt{V_L}\):
\[
  Y_{L,j}(\omega)
  =
  \frac{L\,W(\omega)}{\sqrt{V_L}}
  \in\mathbb R^d .
\]
Third, after removing the linear winding part and rescaling time to
\([0,1]\), define the winding-corrected bridge fluctuation
\[
  \zeta_{L,j}(\omega)(s)
  =
  \frac{
    \widetilde\omega(\beta j s)
    -\widetilde\omega(0)
    -sL\,W(\omega)
  }{\sqrt{V_L}},
  \qquad 0\le s\le1 .
\]
Then \(\zeta_{L,j}\in C_0([0,1];\mathbb R^d)\), where
\[
  C_0([0,1];\mathbb R^d)
  =
  \left\{
    f\in C([0,1];\mathbb R^d): f(0)=f(1)=0
  \right\}
\]
is equipped with the supremum norm.  The mark space is
\[
  \mathsf M_{\mathrm{per}}
  =
  \mathbb T^d
  \times
  \mathbb R^d
  \times
  C_0([0,1];\mathbb R^d).
\]

\subsubsection{The marked Feynman cycle point process}
\label{sssec:periodic-marked-point-process}

We now collect the marked cycles into a point process.  For each cycle of
length \(j\), indexed by \(1\le r\le n_j\), let
\[
  U_{j,r}\sim {\rm Unif}[0,1]
\]
be independent of all other variables.  This auxiliary coordinate has no
physical meaning; it only labels atoms of the point process.  Let
\[
  M_{L,j,r}
  =
  (R_{j,r},Y_{L,j,r},\zeta_{L,j,r})
  \in \mathsf M_{\mathrm{per}}
\]
be the mark extracted from the rooted loop attached to the \(r\)-th cycle of
length \(j\).  The state space is
\[
  E_{\mathrm{per}}
  =
  [0,1]\times(0,\infty)\times\mathsf M_{\mathrm{per}} .
\]
The finite-volume marked Feynman cycle process is
\[
  \Xi_{L,N}^{\mathrm{per}}
  =
  \sum_{j\ge1}\sum_{r=1}^{n_j}
  \delta_{\left(
    U_{j,r},
    j/V_L,
    R_{j,r},
    Y_{L,j,r},
    \zeta_{L,j,r}
  \right)} .
\]
The second coordinate denotes the macroscopic cycle length, while the last
three coordinates denote the rescaled root, the scaled winding endpoint, and
the winding-corrected bridge fluctuation.

Define the finite-volume winding-endpoint law by
\[
  \mathsf G_{L,j}
  \left(
    \left\{\frac{Lw}{\sqrt{V_L}}\right\}
  \right)
  =
  \frac{g_{\beta j}(Lw)}
  {\sum_{m\in\mathbb Z^d}g_{\beta j}(Lm)},
  \qquad w\in\mathbb Z^d .
  \]
  This is well defined, since the denominator is strictly positive and finite by
the Gaussian decay of \(g_{\beta j}(Lm)\) over \(m\in\mathbb Z^d\).
For \(x>0\), let \(\mathsf B_x\) denote the law on
\(C_0([0,1];\mathbb R^d)\) of
\[
  \sqrt{2\beta x}\,B^{\mathrm{br}},
\]
where \(B^{\mathrm{br}}\) is a standard \(d\)-dimensional Brownian bridge on
\([0,1]\).  The corresponding single-loop mark kernel on $\mathsf M_{\mathrm{per}}$ is the following product measure
\[
  \mathsf J_{L,j}^{\mathrm{per}}(dR,dY,d\zeta)
  =
  \mathsf H_{\mathbb T^d}(dR)\,
  \mathsf G_{L,j}(dY)\,
  \mathsf B_{j/V_L}(d\zeta),
\]
where \(\mathsf H_{\mathbb T^d}\) denotes normalised Haar measure on
\(\mathbb T^d\).

\begin{proposition}[Finite-volume compatibility and mark factorisation]
\label{prop:periodic-finite-volume-compatibility}
The process \(\Xi_{L,N}^{\mathrm{per}}\) is obtained from the finite-volume
periodic ideal Bose Feynman--Kac measure
\eqref{eq:periodic-spatial-permutation-measure} by decomposing the
permutation into cycles, disintegrating each cycle into a rooted Brownian
loop, describing the marks defined in
\Cref{sssec:periodic-rooted-loops-marks}, and adding independent uniform
labels.  Consequently:
\begin{enumerate}
  \item the unmarked cycle-count marginal is
  \eqref{eq:periodic-canonical-cycle-law};
  \item conditionally on the cycle counts, the marks attached to distinct
  cycles are independent;
  \item for a cycle of length \(j\), the single-loop mark has law
  \[
    (R,Y_{L,j},\zeta_{L,j})
    \sim
    \mathsf J_{L,j}^{\mathrm{per}}
    =
    \mathsf H_{\mathbb T^d}\otimes
    \mathsf G_{L,j}\otimes
    \mathsf B_{j/V_L};
  \]
  in particular, the rescaled root, the scaled winding endpoint, and the
  winding-corrected bridge fluctuation are independent;
  \item cutting each rooted loop of duration \(\beta j\) into its \(j\)
  consecutive time-\(\beta\) legs and forgetting the marks recovers the
  spatial-permutation measure
  \eqref{eq:periodic-spatial-permutation-measure}.
\end{enumerate}
\end{proposition}

\begin{proof}

We only need to identify the single-loop mark law in \((3)\).  The remaining
statements follow directly from the  construction of the
marked point process.

  Since
\(q_{\beta j}^L(A,A)\) is translation invariant, the root \(A\) is uniform on
\(\mathbb T_L^d\), and therefore \(R=A/L\) has law
\(\mathsf H_{\mathbb T^d}\) and independent of the shape of cycles.  Decomposing the periodic bridge according to its
winding part, the part \(w\in\mathbb Z^d\) has mass
\(g_{\beta j}(Lw)\).  Thus
\[
  \mathbb P\left(Y_{L,j}=\frac{Lw}{\sqrt{V_L}}\right)
  =
  \frac{g_{\beta j}(Lw)}
       {\sum_{m\in\mathbb Z^d}g_{\beta j}(Lm)}
  =
  \mathsf G_{L,j}
  \left(
    \left\{\frac{Lw}{\sqrt{V_L}}\right\}
  \right).
\]
Finally, conditional on the root \(A\) and on the winding part \(w\), the
lifted bridge is the linear path from \(A\) to \(A+Lw\) plus a centred
Brownian bridge of duration \(\beta j\) for the generator \(\Delta\).  The law
of this centred bridge is independent of both \(A\) and \(w\).  After the time
change \(t=\beta j s\) and the spatial scaling by \(\sqrt{V_L}\), the centred
part has law
$ \mathsf B_{j/V_L}$,
which is
the law of \(\sqrt{2\beta\,j/V_L}\,B^{\mathrm{br}}\).  Hence
\[
  (R,Y_{L,j},\zeta_{L,j})
  \sim
  \mathsf H_{\mathbb T^d}\otimes
  \mathsf G_{L,j}\otimes
  \mathsf B_{j/V_L}
  =
  \mathsf J_{L,j}^{\mathrm{per}} .
\]
The product structure also gives the asserted independence of the three
marks, and independence across distinct cycles follows from the conditional
product construction. 
\end{proof}

\subsubsection{Periodic marked winding--bridge limit}
\label{sssec:periodic-marked-limit}
We first recall that in this model, we have
\[
 \mu_{L,j}^{\mathrm{per,eff}}
  =
   q_{L,j}^{\mathrm{per,eff}}\mathsf J_{L,j}^{\mathrm{per}}=\mathsf J_{L,j}^{\mathrm{per}}.
\]
We now identify the limiting effective one-cycle mark measure.  For \(x>0\),
define
\begin{equation}\label{eq:periodic-limiting-eta}
  \eta_x
  =
  \mathsf H_{\mathbb T^d}
  \otimes
  N(0,2\beta x I_d)
  \otimes
  \mathsf B_x
\end{equation}
on
$ \mathsf M_{\mathrm{per}}.$
The measure \(\eta_x\) has total mass one since
\(\mathsf H_{\mathbb T^d}\), \(N(0,2\beta xI_d)\), and \(\mathsf B_x\) are all
probability measures.
Thus the limiting scalar profile is
$ \phi_{\mathrm{per}}(x)
  =
  \eta_x(\mathsf M_{\mathrm{per}})
  \equiv 1$.
  It is obvious that \(x\mapsto\eta_x\) is weakly continuous on \((0,\infty)\) hence the condition $(1)$ in~\cref{ass:limiting-effective-kernel} is satisfied.

\begin{proposition}[Verification of the effective marked trace convergence]
\label{prop:periodic-verification}
Assume \(d>2\).   For the periodic ideal Bose gas with mark space
\(\mathsf M_{\mathrm{per}}\), the convergence part of~\cref{ass:effective-marked-trace-convergence} holds.  More precisely, for
every \(0<\delta<R<\infty\) and every
\(F\in C_b([0,1]\times[\delta,R]\times\mathsf M_{\mathrm{per}})\),
\[
    \sup_{x\in[\delta,R]}
    \left|
    \int_0^1\!\int_{\mathsf M_{\mathrm{per}}}
    F(u,x,m)\,
    \mu_{L,\lfloor xV_L\rfloor}^{\mathrm{per,eff}}(dm)\,du
    -
    \int_0^1\!\int_{\mathsf M_{\mathrm{per}}}
    F(u,x,m)\,\eta_x(dm)\,du
    \right|
    \longrightarrow 0 .
\]
\end{proposition}

\begin{proof}
Set
$   j_L(x):=\lfloor xV_L\rfloor$ and
$    x_L(x):=\frac{j_L(x)}{V_L}$.
Then
$   \sup_{x\in[\delta,R]} |x_L(x)-x|
    \le \frac1{V_L}
    \longrightarrow 0 $.
For \(L\) large enough, \(x_L(x)\in[\delta/2,2R]\) uniformly in
\(x\in[\delta,R]\).
The root components in  $\mu_{L,j_L(x)}^{\mathrm{per,eff}}$ and in $ \eta_x$ are identical, and it remains to compare uniformly the
winding endpoint law and the bridge law.

Let $h_L:={L}/{\sqrt{V_L}}=L^{1-d/2}$.
Since \(d>2\), \(h_L\to0\).  For \(j=j_L(x)\), the winding endpoint law is
supported on \(h_L\mathbb Z^d\) and satisfies
\[
    \mathsf G_{L,j_L(x)}(\{h_Lw\})
    =
    \frac{
    \exp\!\left\{-|h_Lw|^2/(4\beta x_L(x))\right\}
    }{
    \sum_{m\in\mathbb Z^d}
    \exp\!\left\{-|h_Lm|^2/(4\beta x_L(x))\right\}
    },
    \qquad w\in\mathbb Z^d .
\]
The common heat-kernel prefactor cancels in the ratio.  Hence
\(\mathsf G_{L,j_L(x)}\) is the Riemann-sum discretisation of the Gaussian
density
\[
    y\mapsto
    (4\pi\beta x_L(x))^{-d/2}
    \exp\!\left\{-|y|^2/(4\beta x_L(x))\right\}.
\]
The Riemann-sum convergence is uniform for
\(x_L(x)\in[\delta/2,2R]\).  Indeed, the Gaussian tails are uniformly
controlled on this compact parameter interval, and on every compact subset of
\(\mathbb R^d\) the Gaussian densities are uniformly continuous in both
\(y\) and the parameter.  Therefore, for every
\(f\in C_b(\mathbb R^d)\),
\[
    \sup_{x\in[\delta,R]}
    \left|
    \int_{\mathbb R^d} f(y)\,\mathsf G_{L,j_L(x)}(dy)
    -
    \int_{\mathbb R^d} f(y)\,
    N(0,2\beta x_L(x)I_d)(dy)
    \right|
    \longrightarrow 0 .
\]
Since \(x_L(x)\to x\) uniformly and the map
\(x\mapsto N(0,2\beta x I_d)\) is weakly continuous uniformly on compact
subsets of \((0,\infty)\), we also have
\[
    \sup_{x\in[\delta,R]}
    \left|
    \int_{\mathbb R^d} f(y)\,
    N(0,2\beta x_L(x)I_d)(dy)
    -
    \int_{\mathbb R^d} f(y)\,
    N(0,2\beta x I_d)(dy)
    \right|
    \longrightarrow 0 .
\]
Consequently,
$   \mathsf G_{L,j_L(x)}
    \Rightarrow
    N(0,2\beta xI_d)$
uniformly for \(x\in[\delta,R]\).
Similarly, the bridge laws satisfy
$   \mathsf B_{x_L(x)}
    \Rightarrow
    \mathsf B_x$
uniformly for \(x\in[\delta,R]\) since the scaling factors
\(\sqrt{2\beta x_L(x)}\) converge uniformly to \(\sqrt{2\beta x}\).
The uniform tightness of the winding and bridge laws, together with the
compactness of \(\mathbb T^d\), implies uniform tightness of the corresponding
product measures on \(\mathsf M_{\mathrm{per}}\).  Thus, for every
\(\varepsilon>0\), there exists a compact set
\(K_\varepsilon\subset\mathsf M_{\mathrm{per}}\), independent of \(x\) and of
all sufficiently large \(L\), such that
\[
    \sup_{x\in[\delta,R]}
    \left[
    \mu_{L,j_L(x)}^{\mathrm{per,eff}}(K_\varepsilon^c)
    +
    \eta_x(K_\varepsilon^c)
    \right]
    \le
    \frac{\varepsilon}{4(\|F\|_\infty\vee1)} .
\]
Hence the contribution of \(K_\varepsilon^c\) to the difference of the two
integrals is at most \(\varepsilon/2\), uniformly in
\(u\in[0,1]\), \(x\in[\delta,R]\), and all sufficiently large \(L\).
On
$   [0,1]\times[\delta,R]\times K_\varepsilon$,
the function \(F\) is uniformly continuous.  By this uniform continuity, the
preceding uniform weak convergence of the winding and bridge coordinates, and
the standard tensorisation argument for product measures, the contribution from
\(K_\varepsilon\) converges to zero uniformly in
\((u,x)\in[0,1]\times[\delta,R]\).  Therefore
\[
    \sup_{u\in[0,1]}\sup_{x\in[\delta,R]}
    \left|
    \int_{\mathsf M_{\mathrm{per}}}
    F(u,x,m)\,
    \mu_{L,j_L(x)}^{\mathrm{per,eff}}(dm)
    -
    \int_{\mathsf M_{\mathrm{per}}}
    F(u,x,m)\,\eta_x(dm)
    \right|
    \longrightarrow 0 .
\]
Integrating over \(u\in[0,1]\) gives
\[
    \sup_{x\in[\delta,R]}
    \left|
    \int_0^1\!\int_{\mathsf M_{\mathrm{per}}}
    F(u,x,m)\,
    \mu_{L,\lfloor xV_L\rfloor}^{\mathrm{per,eff}}(dm)\,du
    -
    \int_0^1\!\int_{\mathsf M_{\mathrm{per}}}
    F(u,x,m)\,\eta_x(dm)\,du
    \right|
    \longrightarrow 0 .
\]
The proof is complete.
\end{proof}

Now we have verified all assumptions in~\cref{sec:limit-objects-main-results}, we can conclude the following marked
bridge limit for the periodic ideal Bose gas model.

\begin{corollary}[Periodic marked winding--bridge limit]
\label{cor:periodic-marked-winding-bridge-limit}
Assume \(d>2\) and let \(N_L/V_L\to\rho>\rho_{\mathrm c}\). 
Then the finite-volume periodic marked Feynman cycle process
\(\Xi_{L,N_L}^{\mathrm{per}}\)  converges to the marked Gamma bridge of total
mass $\rho-\rho_{\mathrm c}$ with length-dependent mark measure \(\eta_x\) given by
\eqref{eq:periodic-limiting-eta}.  More precisely, before conditioning on the
total macroscopic mass, the limiting tilted Poisson intensity is
$ du\,e^{-s x}dx/x\,\eta_x(dm)$,
and conditioning the total mass to be $\rho-\rho_{\mathrm c}$ gives the canonical marked bridge.
The resulting bridge law is independent of the auxiliary tilt parameter \(s\).

Let
$ \bigl(X_{L,i},U_{L,i},R_{L,i},Y_{L,i},\zeta_{L,i}\bigr)_{i\ge1}$
be the atoms of \(\Xi_{L,N_L}^{\mathrm{per}}\) ranked by decreasing
macroscopic length.  Then, for every fixed \(m\ge1\),
\[
  \bigl(X_{L,i},U_{L,i},R_{L,i},Y_{L,i},\zeta_{L,i}\bigr)_{1\le i\le m}
  \Longrightarrow
  \bigl(X_i,U_i,R_i,Y_i,\zeta_i\bigr)_{1\le i\le m}.
\]
The limiting ranked lengths satisfy
$ (X_i)_{i\ge1}
  \sim
  (\rho-\rho_c)\,{\rm PD}(0,1)$.
Conditionally on \((X_i)_{i\ge1}\), the marks are independent and, for each
\(i\),
\[
  R_i\sim\mathsf H_{\mathbb T^d},
  \qquad
  Y_i\sim N(0,2\beta X_iI_d),
  \qquad
  \zeta_i\sim \sqrt{2\beta X_i}\,B^{\mathrm{br}}.
\]
\end{corollary}
\begin{remark}[Discrete periodic random-walk analogue]
\label{rem:rw-variant-after-ibg}
There is a completely discrete periodic analogue of the periodic ideal
Bose gas model.  Let
$   \Lambda_L^{\mathrm{lat}}=(\mathbb Z/L\mathbb Z)^d$,
    $V_L=L^d$,
and let \(K_L^{\mathrm{rw}}\) be the positive nearest-neighbour lattice
Laplacian,
\[
    (K_L^{\mathrm{rw}}f)(x)
    =
    \sum_{\ell=1}^d
    \bigl(2f(x)-f(x+e_\ell)-f(x-e_\ell)\bigr).
\]
By Fourier diagonalisation, its eigenvalues are
\[
    \varepsilon_{L,k}^{\mathrm{rw}}
    =
    \varepsilon\!\left(\frac{2\pi k}{L}\right),
    \qquad
    k\in\{0,\ldots,L-1\}^d,
\]
where the lattice dispersion relation is
\[
    \varepsilon(\theta)
    =
    2\sum_{\ell=1}^d(1-\cos\theta_\ell),
    \qquad
    \theta\in[-\pi,\pi]^d .
\]
Thus the one-cycle trace is
\[
    q_{L,j}^{\mathrm{rw}}
    =
    \operatorname{Tr} e^{-\beta jK_L^{\mathrm{rw}}}
    =
    \sum_{k\in\{0,\ldots,L-1\}^d}
    \exp\left\{
      -\beta j\,\varepsilon_{L,k}^{\mathrm{rw}}
    \right\}.
\]
The corresponding critical density is
\[
    \rho_{\mathrm c}^{\mathrm{rw}}(\beta)
    =
    \int_{[-\pi,\pi]^d}
    \frac{1}{e^{\beta\varepsilon(\theta)}-1}\,
    \frac{d\theta}{(2\pi)^d},
\]
which is finite for \(d>2\). 
On the macroscopic cycle scale \(j\asymp V_L=L^d\), only the zero Fourier mode
contributes.  Indeed,
 for \(j\asymp L^d\),
\[
    \beta j\min_{k\ne0}\varepsilon_{L,k}^{\mathrm{rw}}
    \asymp L^{d-2}
    \longrightarrow\infty .
\]
Thus the effective macroscopic trace is again
$   \phi_{\mathrm{rw}}(x)\equiv1$,
as in the periodic continuum model.

Consequently, if
$  N_L/V_L\longrightarrow
    \rho>\rho_{\mathrm c}^{\mathrm{rw}}(\beta)$,
then the excess macroscopic mass is
$   \rho-\rho_{\mathrm c}^{\mathrm{rw}}(\beta)$,
and the ranked macroscopic cycle lengths, after normalization by this excess
mass, converge to \(\mathrm{PD}(0,1)\).
Moreover, if the same diffusive path marks as in
\cref{cor:periodic-marked-winding-bridge-limit} are retained, Donsker's
invariance principle identifies the limiting mark kernel with the periodic
Brownian winding--bridge kernel:
$   \eta_x^{\mathrm{rw}}=\eta_x$,
    $x>0 $.
Equivalently, conditionally on a limiting macroscopic length \(X_i\), the
root is uniform on \(\mathbb T^d\), the winding displacement is Gaussian with
covariance \(2\beta X_i I_d\), and the fluctuation is
\(\sqrt{2\beta X_i}\,B^{\mathrm{br}}\).  
\end{remark}

\subsection{Dirichlet and Neumann boundary conditions: empirical local-process marks}
\label{subsec:DN-empirical-local-process}

We now treat the Dirichlet and Neumann boundary conditions in a unified way.
Denote by
$ q_t^{\Lambda_L,b}$
the heat kernel in
$ \Lambda_L=(0,L)^d$
with boundary condition \(b\).  
For \(N\) bosons at inverse temperature \(\beta\), the canonical
Feynman--Kac measure on spatial configurations and permutations is
\[
  \mathbb Q_{L,N}^{b}(d\mathbf x,\pi)
  =
  \frac{1}{N!\,Z_{L,N}^{b}}
  \prod_{i=1}^N q_\beta^{\Lambda_L,b}
      (x_i,x_{\pi(i)})
  \,d\mathbf x,
  \qquad
  \mathbf x\in\Lambda_L^N,\quad \pi\in\mathcal S_N ,
\]
where
\[
  Z_{L,N}^{b}
  =
  \frac{1}{N!}
  \sum_{\pi\in\mathcal S_N}
  \int_{\Lambda_L^N}
  \prod_{i=1}^N q_\beta^{\Lambda_L,b}
      (x_i,x_{\pi(i)})
  \,dx_1\cdots dx_N .
\]
As in the periodic case, a permutation decomposes into cycles.  Conditional on
\((\mathbf x,\pi)\), each cycle of length \(j\) is represented by the
concatenation of \(j\) independent \(b\)-Brownian bridges of time length
\(\beta\), from \(x_i\) to \(x_{\pi(i)}\).  Thus a cycle of length \(j\)
naturally gives an unrooted Brownian loop in \(\Lambda_L\) of total time
\(\beta j\).
For \(b=D\), this is a killed Brownian loop: the Dirichlet bridge is killed
upon hitting \(\partial\Lambda_L\), and is conditioned to survive up to its
terminal time and to arrive at the prescribed endpoint.  For \(b=N\), this is
a reflected Brownian loop in \(\overline{\Lambda_L}\), obtained from reflected
Brownian bridges with Neumann transition density.

The same diffusive scale is relevant in both cases.  Indeed, for a
macroscopic cycle satisfying
$ \frac{j}{V_L}\to x>0$,
the rescaled time length is
\[
  S_{L,j}
  :=
  \frac{\beta j}{L^2}
  \sim
  \beta x L^{d-2}
  \longrightarrow\infty
  \qquad (d>2).
\]
Thus both the killed loop and the reflected loop become, after diffusive
scaling, long loops in the unit cube
$  Q=(0,1)^d$
with diverging time length.  We therefore use the same type of empirical
local-process mark for the two boundary conditions: it describes the empirical
distribution of compact diffusive time windows seen from a uniformly chosen
time along the rescaled unrooted cycle.  The distinction between the
Dirichlet and Neumann cases enters through the limiting local process, which
will be identified separately below.

\subsubsection{The empirical local-process mark and the marked point process}
\label{sssec:DN-empirical-mark}

Let
\[
    \omega:[0,\beta j]\to\overline{\Lambda_L}
\]
be the \(b\)-Brownian loop associated with a cycle of length \(j\).   Its diffusively rescaled path is
\[
    Y_{L,j}(s)
    =
    L^{-1}\omega(L^2s),
    \qquad
    0\le s\le S_{L,j}.
\]
Thus, after diffusive scaling, a macroscopic cycle with length $j=O(V_L)$ becomes a \(b\)-Brownian
bridge in the unit cube \(Q=(0,1)^d\) whose time length diverges.  In the
Dirichlet case this is a long killed bridge conditioned on survival, while in
the Neumann case it is a long reflected bridge.  In both cases we encode the
local geometry of the long bridge by averaging over all shifted diffusive time
windows along the cycle.

Since the loop is closed, we extend \(Y_{L,j}\) periodically to all
\(s\in\mathbb R\) by
$   Y_{L,j}(s+S_{L,j})=Y_{L,j}(s)$.
For \(u\in\mathbb R\), let
$   (\theta_u\gamma)(t)=\gamma(u+t)$,
$\forall t\in\mathbb R$,
be the time-shift operator.  We define the empirical local-process mark of the
cycle by
\[
    \mathcal M_{L,j}^{b}(\omega)
    :=
    \frac1{S_{L,j}}
    \int_0^{S_{L,j}}
    \delta_{\theta_uY_{L,j}}\,du .
\]
This is a probability measure on
\[
    \mathcal X_b
    :=
    C_{\mathrm{loc}}(\mathbb R,\overline Q),
\]
the space of continuous two-sided paths in \(\overline Q\), equipped with the
topology of uniform convergence on compact time intervals.  Intuitively,
\(\mathcal M_{L,j}^{b}\) describes what a typical local time window of the long
cycle looks like when the root of the loop is chosen uniformly along its
diffusive time length.  Thus the mark describes the empirical distribution of
local shapes seen along the whole unrooted cycle, rather than the behaviour
near one prescribed point.  We take the mark space to be
\[
    \mathsf M_b:=\mathcal P(\mathcal X_b),
\]
with the topology of weak convergence.  Since \(\mathcal X_b\) is Polish,
\(\mathsf M_b\) is Polish as well.

We now attach these marks to the cycles in the canonical ensemble with
boundary condition \(b\).  Let \((n_j)_{j\ge1}\) denote the cycle counts under
the canonical Feynman--Kac measure \(\mathbb Q_{L,N_L}^{b}\).  Conditionally
on the cycle counts, the cycles are independent.  More precisely, for each
\(j\ge1\) and \(1\le r\le n_j\), let
$   \omega_{j,r}$
be a \(b\)-Brownian loop of duration \(\beta j\), sampled from the
corresponding normalised one-loop measure.  We attach to this loop the
empirical local-process mark
$ M_{L,j,r}^{b}
  :=
  \mathcal M_{L,j}^{b}(\omega_{j,r})
  \in \mathsf M_b $.
As in the general marked-cycle construction, we also assign to each cycle an
independent auxiliary label
$ U_{j,r}\sim {\rm Unif}[0,1]$,
independently of the cycle counts and of all loops.  The finite-volume marked
cycle point process is then defined by
\[
  \Xi_{L,N_L}^{b}
  =
  \sum_{j\ge1}\sum_{r=1}^{n_j}
  \delta_{\left(
    U_{j,r},
    j/V_L,
    M_{L,j,r}^{b}
  \right)}
  \in
  \mathcal N_\ell\bigl([0,1]\times(0,\infty)\times\mathsf M_b\bigr).
\]

Finally, let $  \mathsf J_{L,j}^{b}$
denote the law of the mark \(\mathcal M_{L,j}^{b}\) under the normalised
\(b\)-one-loop measure of duration \(\beta j\).  Then, conditionally on the
cycle counts, $  M_{L,j,r}^{b}\sim \mathsf J_{L,j}^{b}$,
independently over all pairs \((j,r)\).

\begin{remark}[Why the marks do not describe the global loop]
\label{rem:local-mark-not-loop}
A natural question is why we consider only local empirical marks and do not
define a mark that captures the whole macroscopic loop.  The reason is that a
macroscopic cycle of length \(j\asymp V_L\) has diffusively rescaled duration
\[
    S_L=\frac{\beta j}{L^2}\asymp L^{d-2}\to\infty .
\]
Thus a global-loop mark would have to encode a closed path with a diverging
time horizon, and there is no canonical non-degenerate limit in a fixed
finite-time loop space.

The local empirical mark uses a different observable: it observes the
periodically extended loop from a typical time and only on compact time
windows.  This always gives an element of
\(C_{\mathrm{loc}}(\mathbb R,\overline Q)\).  In this local viewpoint the
closing constraint is pushed to infinite time and disappears in the limit.
Consequently the limiting mark is a stationary two-sided process, not a loop
law.
\end{remark}
\subsubsection{Limiting empirical local-process marked point process: the Dirichlet case}
\label{sssec:dirichlet-empirical-one-loop-limit}

We first identify the limiting local process seen from a uniformly chosen time
on a long Dirichlet loop.  The limit is the two-sided stationary Dirichlet
taboo process in \(Q\).

Let \(p_t^{Q,D}\) denote the Dirichlet heat kernel in \(Q=(0,1)^d\).  Let
\[
    h_D(r)
    =
    2^{d/2}\prod_{\ell=1}^d\sin(\pi r_\ell),
    \qquad
    \varepsilon_D=\pi^2d,
\]
be the \(L^2(Q)\)-normalised positive ground state of \(-\Delta_Q^D\) and its
ground-state eigenvalue.
The Dirichlet taboo transition density is the Doob \(h\)-transform
\[
    p_t^{\mathrm{tab}}(r,s)
    =
    e^{\varepsilon_D t}
    \frac{h_D(s)}{h_D(r)}
    p_t^{Q,D}(r,s),
    \qquad r,s\in Q,\quad t>0 .
\]
Its invariant probability measure is
$    h_D(r)^2\,dr$ .

Let
$   \mathbb Q_{D,\mathrm{two}}^{\mathrm{tab}}
    \in\mathcal P(\mathcal X_D)$
denote the law of the two-sided stationary Dirichlet taboo process.  Equivalently,
if \(X=(X_t)_{t\in\mathbb R}\) is the canonical process, then for
\(t_1<\cdots<t_k\),
\[
\begin{aligned}
&\mathbb Q_{D,\mathrm{two}}^{\mathrm{tab}}
  \bigl(
    X_{t_1}\in dr_1,\ldots,X_{t_k}\in dr_k
  \bigr)
\\
&\quad =
  h_D(r_1)^2\,dr_1\,
  p_{t_2-t_1}^{\mathrm{tab}}(r_1,r_2)
  \cdots
  p_{t_k-t_{k-1}}^{\mathrm{tab}}(r_{k-1},r_k)
  dr_2\cdots dr_k .
\end{aligned}
\]
Accordingly, for \(x>0\), define
\[
  \eta_x^D
  :=
  \delta_{\mathbb Q_{D,\mathrm{two}}^{\mathrm{tab}}}
  \in\mathcal P(\mathsf M_D).
\]
The notation allows for length-dependent mark laws in the abstract marked
bridge theorem, although in the present Dirichlet case the limiting mark law is
independent of \(x\).

The following estimates are standard consequences of the spectral gap of the
Dirichlet Laplacian and the Markov bridge decomposition.

\begin{lemma}[Ground-state asymptotics and loop mixing]
\label{lem:dirichlet-ground-state-mixing}
Let \(\varepsilon_2^D>\varepsilon_D\) be the second Dirichlet eigenvalue of
\(-\Delta_Q^D\).  Then the following hold.

\begin{enumerate}
\item For every \(t_0>0\), there exists \(C<\infty\) such that, for all
\(t\ge t_0\) and \(r,s\in Q\),
\[
  \left|
    p_t^{Q,D}(r,s)
    -
    e^{-\varepsilon_Dt}h_D(r)h_D(s)
  \right|
  \le
  C e^{-\varepsilon_2^Dt}.
\]

\item Let \(F,G\) be bounded measurable functionals depending only on time
windows of length at most \(2T\).  Then there exist constants
\(C_{F,G},c_{F,G}>0\) such that, for all sufficiently large \(S\),
\[
\left|
  \operatorname{Cov}_{S}^{D,\mathrm{loop}}
  \bigl(
    F(\theta_uY),
    G(\theta_vY)
  \bigr)
\right|
\le
  C_{F,G}
  \exp\{-c_{F,G}(d_S(u,v)-4T)_+\},
\]
where
\[
  d_S(u,v)=\min\{|u-v|,S-|u-v|\}
\]
is the cyclic distance on the time circle of length \(S\).
\end{enumerate}
\end{lemma}

\begin{proof}
The first estimate follows from the spectral expansion of the Dirichlet heat
kernel on the cube and the spectral gap above the ground state; see, for
instance, Davies \cite[Ch.~4]{Davies1989}.

For the covariance estimate, use the cyclic invariance of the loop and order
the two time windows on the time circle.  If their cyclic distance is at most
\(4T\), the trivial bound
\[
  |\operatorname{Cov}(F,G)|
  \le 4\|F\|_\infty\|G\|_\infty
\]
is sufficient.  If the two windows are separated by a distance \(a>4T\), then
the Markov bridge decomposition expresses
\[
  \mathbb E_{S}^{D,\mathrm{loop}}
  \bigl[
    F(\theta_uY)G(\theta_vY)
  \bigr]
\]
as an integral containing two Dirichlet heat-kernel factors whose time lengths
are bounded below by \(a-4T\), up to deterministic constants depending only on
the window size.  Applying the ground-state asymptotics to the long connecting
pieces gives the product of the corresponding one-window expectations, with an
error bounded by
\[
  C_{F,G}e^{-c_{F,G}(a-4T)}
\]
for some \(c_{F,G}>0\).  Since \(a=d_S(u,v)\), this yields the stated bound.
\end{proof}

We now identify the limiting mark law and verify a mark-kernel condition
needed for~\cref{ass:effective-marked-trace-convergence}.

\begin{proposition}[Dirichlet empirical local-process mark]
\label{prop:dirichlet-empirical-local-process-limit}
For every \(0<\delta<R<\infty\) and every bounded continuous function
\(\Phi:\mathsf M_D\to\mathbb R\),
\[
  \sup_{\delta\le j/V_L\le R}
  \left|
    \int_{\mathsf M_D}\Phi(m)\,\mathsf J_{L,j}^{D}(dm)
    -
    \Phi\bigl(\mathbb Q_{D,\mathrm{two}}^{\mathrm{tab}}\bigr)
  \right|
  \longrightarrow0 .
\]
\end{proposition}

\begin{proof}
It is enough to prove the asserted convergence along every sequence
\(j_L\) such that
$ j_L/V_L\to x\in(0,\infty)$.
The uniform statement on compact intervals then follows by the usual
subsequence argument.

Set
$ S_L:=S_{L,j_L}\sim \beta x L^{d-2}\to\infty $.
Let \(F:\mathcal X_D\to\mathbb R\) be a bounded continuous local functional,
depending only on the restriction of the path to \([-T,T]\).  We first show
that
\[
  \left\langle
    \mathcal M_{L,j_L}^{D},F
  \right\rangle
  =
  \frac1{S_L}\int_0^{S_L}F(\theta_uY_{L,j_L})\,du
  \longrightarrow
  \int F\,d\mathbb Q_{D,\mathrm{two}}^{\mathrm{tab}}
\]
in probability.

By cyclic invariance of the normalised loop measure,
\[
  \mathbb E
  \left[
    \frac1{S_L}\int_0^{S_L}F(\theta_uY_{L,j_L})\,du
  \right]
  =
  \mathbb E\bigl[F(Y_{L,j_L})\bigr].
\]
Thus we first identify the local weak limit of the loop around a fixed time.
By a deterministic time shift, we may assume that \(F\) depends on the path on
an interval \([0,A]\), with \(A\le 2T\).

Consider a cylinder function depending on times
$  0\le t_1<\cdots<t_k\le A $.
Under the normalised Dirichlet loop in \(Q\) of duration \(S_L>A\), the joint
density of
$ (Y_{L,j_L}(t_1),\ldots,Y_{L,j_L}(t_k))$
is
\[
\begin{aligned}
&\frac{
  p_{t_2-t_1}^{Q,D}(r_1,r_2)
  \cdots
  p_{t_k-t_{k-1}}^{Q,D}(r_{k-1},r_k)
  p_{S_L-(t_k-t_1)}^{Q,D}(r_k,r_1)
}{
  \int_Q p_{S_L}^{Q,D}(z,z)\,dz
}
\,dr_1\cdots dr_k .
\end{aligned}
\]
By the ground-state asymptotics in
\(\cref{lem:dirichlet-ground-state-mixing}\),
\[
  p_{S_L-(t_k-t_1)}^{Q,D}(r_k,r_1)
  =
  e^{-\varepsilon_D(S_L-(t_k-t_1))}
  h_D(r_k)h_D(r_1)
  +o(e^{-\varepsilon_D S_L}),
\]
uniformly for \(r_1,r_k\in Q\).  Moreover,
\[
  \int_Q p_{S_L}^{Q,D}(z,z)\,dz
  =
  e^{-\varepsilon_D S_L}(1+o(1)),
\]
because \(\int_Q h_D(z)^2\,dz=1\).  Hence
\[
  \frac{
    p_{S_L-(t_k-t_1)}^{Q,D}(r_k,r_1)
  }{
    \int_Q p_{S_L}^{Q,D}(z,z)\,dz
  }
  \longrightarrow
  e^{\varepsilon_D(t_k-t_1)}h_D(r_k)h_D(r_1),
\]
uniformly on \(Q\times Q\).  Therefore the finite-dimensional density
converges to
\[
\begin{aligned}
&
e^{\varepsilon_D(t_k-t_1)}
h_D(r_1)
p_{t_2-t_1}^{Q,D}(r_1,r_2)
\cdots
p_{t_k-t_{k-1}}^{Q,D}(r_{k-1},r_k)
h_D(r_k)
\,dr_1\cdots dr_k ,
\end{aligned}
\]
which is exactly the finite-dimensional distribution of the two-sided
stationary Dirichlet taboo process.

Together with the standard tightness estimates for Brownian bridges on compact
time intervals, this finite-dimensional convergence implies weak convergence
on \(C([0,A],\overline Q)\).  Consequently,
\[
  \mathbb E
  \left[
    \left\langle
      \mathcal M_{L,j_L}^{D},F
    \right\rangle
  \right]= \mathbb E\bigl[F(Y_{L,j_L})\bigr]
  \longrightarrow
  \int F\,d\mathbb Q_{D,\mathrm{two}}^{\mathrm{tab}} .
\]

It remains to prove concentration.  By
\(\cref{lem:dirichlet-ground-state-mixing}\), there exist constants
\(C_F,c_F>0\) such that, for all sufficiently large \(L\),
\[
\left|
  \operatorname{Cov}
  \bigl(
    F(\theta_uY_{L,j_L}),
    F(\theta_vY_{L,j_L})
  \bigr)
\right|
\le
  C_F e^{-c_F(d_{S_L}(u,v)-4T)_+}.
\]
Therefore
\[
\begin{aligned}\operatorname{Var}
 \left(
  \frac1{S_L}\int_0^{S_L}F(\theta_uY_{L,j_L})\,du
 \right)
& =
  \frac1{S_L^2}
  \int_0^{S_L}\int_0^{S_L}
  \operatorname{Cov}
  \bigl(
    F(\theta_uY_{L,j_L}),
    F(\theta_vY_{L,j_L})
  \bigr)
  \,du\,dv
\\
&\quad\le
  \frac{C_F}{S_L^2}
  \int_0^{S_L}\int_0^{S_L}
  e^{-c_F(d_{S_L}(u,v)-4T)_+}
  \,du\,dv .
\end{aligned}
\]
By translation invariance on the time circle, for each fixed \(u\),
\[
\begin{aligned}
\int_0^{S_L}
  e^{-c_F(d_{S_L}(u,v)-4T)_+}\,dv
&=
\int_0^{S_L}
  e^{-c_F(\min\{w,S_L-w\}-4T)_+}\,dw
\\
&\le
2\int_0^{S_L/2}
  e^{-c_F(r-4T)_+}\,dr
\\
&\le
2\left(4T+\frac1{c_F}\right).
\end{aligned}
\]
Thus the double integral is \(O(S_L)\), and hence
\[
  \operatorname{Var}
 \left(
  \frac1{S_L}\int_0^{S_L}F(\theta_uY_{L,j_L})\,du
 \right)
 \le
 \frac{C_{F,T}}{S_L}
 \longrightarrow0 .
\]
Combining the convergence of the expectation with this variance bound gives
\[
  \left\langle
    \mathcal M_{L,j_L}^{D},F
  \right\rangle
  \longrightarrow
  \int F\,d\mathbb Q_{D,\mathrm{two}}^{\mathrm{tab}}
\]
in probability for every bounded continuous local functional \(F\).

We now upgrade this to convergence of random probability measures in
\(\mathsf M_D=\mathcal P(\mathcal X_D)\).  Choose a countable
convergence-determining family
$ (F_n)_{n\ge1}\subset C_b(\mathcal X_D)$
consisting of bounded continuous local functions.  Since
\(\mathcal X_D=C_{\mathrm{loc}}(\mathbb R,\overline Q)\) is Polish, such a
family exists.  Define
\[
  d_{\mathsf M}(\mu,\nu)
  =
  \sum_{n=1}^{\infty}2^{-n}
  \left(
    \left|
      \int F_n\,d\mu-\int F_n\,d\nu
    \right|
    \wedge1
  \right).
\]
This metric generates the topology of weak convergence on \(\mathsf M_D\).
The preceding convergence, applied to each \(F_n\), implies
\[
  d_{\mathsf M}
  \left(
    \mathcal M_{L,j_L}^{D},
    \mathbb Q_{D,\mathrm{two}}^{\mathrm{tab}}
  \right)
  \longrightarrow0
\]
in probability.  Therefore
$ \mathcal M_{L,j_L}^{D}
  \longrightarrow
  \mathbb Q_{D,\mathrm{two}}^{\mathrm{tab}}$
in probability in \(\mathsf M_D\).  Since the limit is deterministic, the law
\(\mathsf J_{L,j_L}^{D}\) of \(\mathcal M_{L,j_L}^{D}\) converges weakly to
$ \delta_{\mathbb Q_{D,\mathrm{two}}^{\mathrm{tab}}}$.
That is, for every bounded continuous
\(\Phi:\mathsf M_D\to\mathbb R\),
\[
  \int_{\mathsf M_D}\Phi(m)\,\mathsf J_{L,j_L}^{D}(dm)
  \longrightarrow
  \Phi\bigl(\mathbb Q_{D,\mathrm{two}}^{\mathrm{tab}}\bigr).
\]

Finally, if the uniform convergence on \([\delta,R]\) failed, then there would
exist \(\eta>0\), a subsequence \(L_n\), and integers \(j_{L_n}\) with
$ \delta\le j_{L_n}/V_{L_n}\le R$
such that
\[
\left|
    \int_{\mathsf M_D}\Phi(m)\,\mathsf J_{L_n,j_{L_n}}^{D}(dm)
    -
    \Phi\bigl(\mathbb Q_{D,\mathrm{two}}^{\mathrm{tab}}\bigr)
\right|
\ge \eta .
\]
Passing to a further subsequence, we may assume
$ j_{L_n}/V_{L_n}\to x\in[\delta,R]$.
This contradicts the sequential convergence proved above.  Hence the
convergence is uniform on compact subsets of \((0,\infty)\).
\end{proof}
Similar to the proof in~\cref{prop:periodic-verification}, we can verify the~\cref{ass:effective-marked-trace-convergence} holds. Thus, we can conclude the following marked bridge limit under Dirichlet condition.
\begin{corollary}[Dirichlet marked empirical-process bridge limit]
\label{cor:dirichlet-empirical-marked-bridge}
Assume \(d>2\) and let \(N_L/V_L\to\rho>\rho_{\mathrm c}\).  Then the
finite-volume Dirichlet marked Feynman cycle process
$ \Xi_{L,N_L}^{D}$
converges to the marked Gamma bridge of total mass
$ \rho-\rho_{\mathrm c}$
with length-dependent mark measure
$ \eta_x^D
  =
  \delta_{\mathbb Q_{D,\mathrm{two}}^{\mathrm{tab}}}$, $\forall x>0$.
\end{corollary}

\subsubsection{Limiting empirical local-process marked point process: the Neumann case}
\label{sssec:neumann-empirical-one-loop-limit}

We next consider Neumann boundary conditions.  The argument is parallel to the
Dirichlet case, with the ground state now given by the constant function.

Let \(p_t^{Q,N}\) denote the Neumann heat kernel in \(Q=(0,1)^d\).  The
\(L^2(Q)\)-normalised ground state and ground-state eigenvalue of
\(-\Delta_Q^N\) are
$    h_N(r)\equiv 1$,
and
  $  \varepsilon_N=0 $.
Thus the corresponding Doob ground-state transform is trivial, and the limiting
local process is the stationary reflected Brownian motion in \(\overline Q\),
with transition density \(p_t^{Q,N}\) and invariant probability measure \(dr\).

Let
$   \mathbb Q_{N,\mathrm{two}}^{\mathrm{ref}}
    \in \mathcal P(\mathcal X_N)$
denote the law of the two-sided stationary reflected Brownian motion in
\(\overline Q\).  Equivalently, if \(X=(X_t)_{t\in\mathbb R}\) is the canonical
process, then for \(t_1<\cdots<t_k\),
\[
\begin{aligned}
&\mathbb Q_{N,\mathrm{two}}^{\mathrm{ref}}
  \bigl(
    X_{t_1}\in dr_1,\ldots,X_{t_k}\in dr_k
  \bigr)
\\
&\quad =
  dr_1\,
  p_{t_2-t_1}^{Q,N}(r_1,r_2)
  \cdots
  p_{t_k-t_{k-1}}^{Q,N}(r_{k-1},r_k)
  dr_2\cdots dr_k .
\end{aligned}
\]
For \(x>0\), set
$    \eta_x^N
    :=
    \delta_{\mathbb Q_{N,\mathrm{two}}^{\mathrm{ref}}}
    \in\mathcal P(\mathsf M_N)$.

The following lemma is the Neumann analogue of
\(\cref{lem:dirichlet-ground-state-mixing}\).

\begin{lemma}[Neumann ground-state asymptotics and loop mixing]
\label{lem:neumann-ground-state-mixing}
Let \(\varepsilon_2^N>0\) be the first positive Neumann eigenvalue of
\(-\Delta_Q^N\).  Then the following hold.

\begin{enumerate}
\item For every \(t_0>0\), there exists \(C<\infty\) such that, for all
\(t\ge t_0\) and \(r,s\in\overline Q\),
\[
  \left|
    p_t^{Q,N}(r,s)-1
  \right|
  \le
  C e^{-\varepsilon_2^N t}.
\]

\item Let \(F,G\) be bounded measurable functionals depending only on time
windows of length at most \(2T\).  Then there exist constants
\(C_{F,G},c_{F,G}>0\) such that, for all sufficiently large \(S\),
\[
\left|
  \operatorname{Cov}_{S}^{N,\mathrm{loop}}
  \bigl(
    F(\theta_uY),
    G(\theta_vY)
  \bigr)
\right|
\le
  C_{F,G}
  \exp\{-c_{F,G}(d_S(u,v)-4T)_+\},
\]
where
\[
  d_S(u,v)=\min\{|u-v|,S-|u-v|\}.
\]
\end{enumerate}
\end{lemma}

\begin{proof}
The first estimate follows from the spectral expansion of the Neumann heat
kernel.  Since the ground state is \(h_N\equiv1\) and the next eigenvalue is
\(\varepsilon_2^N>0\), one has, uniformly for \(t\ge t_0\),
\[
    p_t^{Q,N}(r,s)
    =
    1+O(e^{-\varepsilon_2^N t}).
\]
The covariance estimate is obtained exactly as in
\(\cref{lem:dirichlet-ground-state-mixing}\).  Using the Markov bridge
decomposition, two local windows separated by cyclic distance \(a>4T\) are
connected by heat-kernel pieces of length at least \(a-4T\).  Applying the
above ground-state asymptotics to these connecting pieces factorises the
two-window expectation up to an error bounded by
\[
    C_{F,G}e^{-c_{F,G}(a-4T)}.
\]
The trivial bound on the covariance covers the case \(a\le4T\).
\end{proof}
\begin{proposition}[Neumann empirical local-process mark]
\label{prop:neumann-empirical-local-process-limit}
For every \(0<\delta<R<\infty\) and every bounded continuous function
\(\Phi:\mathsf M_N\to\mathbb R\),
\[
  \sup_{\delta\le j/V_L\le R}
  \left|
    \int_{\mathsf M_N}\Phi(m)\,\mathsf J_{L,j}^{N}(dm)
    -
    \Phi\bigl(\mathbb Q_{N,\mathrm{two}}^{\mathrm{ref}}\bigr)
  \right|
  \longrightarrow0 .
\]
\end{proposition}

\begin{proof}
The proof is the same as that of
\(\cref{prop:dirichlet-empirical-local-process-limit}\), with the following
replacements:
$    p_t^{Q,D}$ by $\ p_t^{Q,N}$;
$    h_D$ by $ h_N\equiv1$,
$    \varepsilon_D$ by $\varepsilon_N=0$
and
$   \mathbb Q_{D,\mathrm{two}}^{\mathrm{tab}}$
by $    \mathbb Q_{N,\mathrm{two}}^{\mathrm{ref}} $.
We indicate the only point where the limiting finite-dimensional distribution
changes.

Let \(j_L/V_L\to x\in(0,\infty)\) and set
$   S_L:=S_{L,j_L}\sim \beta x L^{d-2}\to\infty $.
For a cylinder function depending on times
\(0\le t_1<\cdots<t_k\le A\), the joint density under the normalised Neumann
loop of duration \(S_L>A\) is
\[
\begin{aligned}
&\frac{
  p_{t_2-t_1}^{Q,N}(r_1,r_2)
  \cdots
  p_{t_k-t_{k-1}}^{Q,N}(r_{k-1},r_k)
  p_{S_L-(t_k-t_1)}^{Q,N}(r_k,r_1)
}{
  \int_Q p_{S_L}^{Q,N}(z,z)\,dz
}
\,dr_1\cdots dr_k .
\end{aligned}
\]
By \(\cref{lem:neumann-ground-state-mixing}\),
\[
    p_{S_L-(t_k-t_1)}^{Q,N}(r_k,r_1)\longrightarrow 1
\]
uniformly in \(r_1,r_k\in\overline Q\), and
\[
    \int_Q p_{S_L}^{Q,N}(z,z)\,dz\longrightarrow 1 .
\]
Hence the above density converges to
\[
  dr_1\,
  p_{t_2-t_1}^{Q,N}(r_1,r_2)
  \cdots
  p_{t_k-t_{k-1}}^{Q,N}(r_{k-1},r_k)
  dr_2\cdots dr_k ,
\]
which is precisely the finite-dimensional distribution of
\(\mathbb Q_{N,\mathrm{two}}^{\mathrm{ref}}\).
The tightness on compact time intervals is the standard tightness of reflected
Brownian bridges.  Therefore the local process seen from a fixed time
converges weakly to
\(\mathbb Q_{N,\mathrm{two}}^{\mathrm{ref}}\).

The concentration of the empirical local process follows verbatim from the
covariance estimate in \(\cref{lem:neumann-ground-state-mixing}\).  Namely, for
every bounded continuous local functional \(F\),
\[
  \operatorname{Var}
  \left(
    \frac1{S_L}\int_0^{S_L}F(\theta_uY_{L,j_L})\,du
  \right)
  \le
  \frac{C_{F,T}}{S_L}
  \longrightarrow0 .
\]
Thus
$ \left\langle
    \mathcal M_{L,j_L}^{N},F
  \right\rangle
  \longrightarrow
  \int F\,d\mathbb Q_{N,\mathrm{two}}^{\mathrm{ref}}$
in probability for every bounded continuous local \(F\).

As in the proof of
\(\cref{prop:dirichlet-empirical-local-process-limit}\), a countable
convergence-determining family of local functions upgrades this to convergence
in probability in \(\mathsf M_N\):
$   \mathcal M_{L,j_L}^{N}
    \longrightarrow
    \mathbb Q_{N,\mathrm{two}}^{\mathrm{ref}} $.
Since the limit is deterministic, the laws
\(\mathsf J_{L,j_L}^{N}\) converge weakly to
$   \delta_{\mathbb Q_{N,\mathrm{two}}^{\mathrm{ref}}}$.
Finally, the uniform convergence for \(\delta\le j/V_L\le R\) follows by the
same subsequence argument used in the Dirichlet case.
\end{proof}

Similar to the periodic and Dirichlet cases, the preceding proposition verifies
the mark-kernel condition in
\(\cref{ass:effective-marked-trace-convergence}\).  We therefore obtain the
following Neumann marked bridge limit.

\begin{corollary}[Neumann marked empirical-process bridge limit]
\label{cor:neumann-empirical-marked-bridge}
Assume \(d>2\) and let \(N_L/V_L\to\rho>\rho_{\mathrm c}\).  Then the
finite-volume Neumann marked Feynman cycle process
\(\Xi_{L,N_L}^{N}\) converges to the marked Gamma bridge of total mass
\(\rho-\rho_{\mathrm c}\) with length-dependent mark measure
$   \eta_x^N   =
    \delta_{\mathbb Q_{N,\mathrm{two}}^{\mathrm{ref}}}$,
    $ x>0 $.
\end{corollary}

%% file: 05-brownian-loop-shape-winding.tex
\section{Double-well loop marks and finite-type extensions}
\label{sec:double-well-loop-marks}

The preceding sections developed the marked-cycle framework in abstract
generality.  We now derive the effective trace and the one-cycle mark law
from a concrete model: a Schrödinger loop gas with a tunnelling doublet
at the bottom of its spectrum.

The double-well is the simplest setting in which macroscopic cycles carry
non-trivial internal structure.  When the tunnelling splitting satisfies
$V_L\Delta_L\to\gamma\in[0,\infty)$, the two lowest eigenvalues both
remain visible on the macroscopic cycle scale $j\asymp V_L$, while all
higher modes are invisible.  When $\gamma\neq 0$, the resulting scalar profile
$\phi_\gamma(x)=1+e^{-\beta\gamma x}$ is non-constant, so the
macroscopic length law is a Poisson--Kingman bridge rather than a Gamma
bridge.  We set up the finite-volume double-well loop gas and isolate the
effective two-level trace in
\cref{subsec:finite-volume-double-well-loop-gas}.  Two natural mark
choices then extract different information from the same doublet.
In \cref{subsec:double-well-two-state-well-loop-marks} we construct a
\emph{well-loop mark}, encoding the effective two-state tunnelling
history in the localized basis, and apply the abstract marked
Poisson--Kingman limit.
In \cref{subsec:spectral-labels-finite-type-band-extensions} we treat the
\emph{spectral-label mark}, which retains only the eigenmode index in the
diagonal basis, and extend it to a finite-type band with $Q$ visible
components.  Both marks share the scalar profile~$\phi_\gamma$ but live
on different mark spaces with different conditional distributions.

\subsection{Finite-volume double-well loop gas}
\label{subsec:finite-volume-double-well-loop-gas}

Let \(\Lambda_L\subset\mathbb R^d\) be a finite box and let
\(V_L:=|\Lambda_L|\).  We impose one of the standard boundary conditions
$    b\in\{\mathrm{per},D,N\}$.
In this subsection \(b\) is fixed and suppressed from the notation.  We consider
the one-particle Schrödinger operator
\[
  H_L^{\mathrm{dw}}
  =
  -\Delta_{\Lambda_L}^{\,b}+U_L
\]
on \(L^2(\Lambda_L)\).  The superscript ``\(\mathrm{dw}\)'' indicates that
the bottom of the spectrum is generated by a double-well geometry.  Let
$ E_{0,L}<E_{1,L}\le E_{2,L}\le\cdots$
be the eigenvalues, counted with multiplicity, and let
$  \psi_{0,L},\psi_{1,L},\psi_{2,L},\ldots$
be an associated orthonormal eigenbasis.

We shift the operator by its ground-state energy and write
\[
  K_L^{\mathrm{dw}}
  :=
  H_L^{\mathrm{dw}}-E_{0,L}.
\]
Thus the eigenvalues of \(K_L^{\mathrm{dw}}\) are
$  \varepsilon_{k,L}=E_{k,L}-E_{0,L},\quad k\ge0$.
We set
\[
  \Delta_L:=\varepsilon_{1,L}.
\]
The low-energy assumption is that the two lowest shifted eigenvalues form a
macroscopic doublet, while the rest of the spectrum is invisible on the
macroscopic cycle scale.

\begin{assumption}[Critical double-well scaling]
\label{ass:critical-double-well-scaling}
There exists \(\gamma\in[0,\infty)\) such that
\[
  \lambda_{1,L}:=V_L\Delta_L\longrightarrow\gamma .
\]
Moreover, the next shifted eigenvalue escapes on the volume scale:
\[
  \lambda_{2,L}:=V_L\varepsilon_{2,L}\longrightarrow\infty .
\]
\end{assumption}

The condition \(V_L\Delta_L\to\gamma\) says that the tunnelling splitting is
exactly visible to cycles of length \(j\asymp V_L\).  Indeed, if
\(j/V_L\to x>0\), then
\[
  e^{-\beta j\Delta_L}
  =
  e^{-\beta (j/V_L)(V_L\Delta_L)}
  \longrightarrow
  e^{-\beta\gamma x}.
\]
By contrast, the condition \(V_L\varepsilon_{2,L}\to\infty\) implies that, on
the same scale,
\[
  e^{-\beta j\varepsilon_{2,L}}
  =
  e^{-\beta (j/V_L)(V_L\varepsilon_{2,L})}
  \longrightarrow0.
\]
Since the spectrum is ordered, every fixed mode \(k\ge2\) is exponentially
suppressed on the macroscopic cycle scale.  The total contribution of all
background modes will be controlled by the background concentration assumption
below.

Let
\[
  k_{L,t}^{\mathrm{dw}}(x,y)
  :=
  e^{-tK_L^{\mathrm{dw}}}(x,y)
  =
  e^{tE_{0,L}}e^{-tH_L^{\mathrm{dw}}}(x,y)
\]
be the heat kernel of the shifted semigroup.  The unshifted kernel has the
usual Feynman--Kac representation in terms of Brownian bridges in
\(\Lambda_L\): periodic bridges in the periodic case, killed bridges in the
Dirichlet case, and reflected bridges in the Neumann case.  Since the shift by
\(E_{0,L}\) only multiplies each time-\(t\) kernel by \(e^{tE_{0,L}}\), the
shifted kernel gives the same canonical loop gas after normalization.

For \(N\ge1\), the canonical \(N\)-particle Feynman--Kac measure may be written
as
\[
  \mathbb P_{L,N}^{\mathrm{dw}}(d\mathbf x,d\pi)
  =
  \frac{1}{N!\,Z_{L,N}^{\mathrm{dw}}}
  \prod_{i=1}^N
  k_{L,\beta}^{\mathrm{dw}}
  \bigl(x_i,x_{\pi(i)}\bigr)
  \,d x_1\cdots d x_N,
\]
where \(\pi\in\mathfrak S_N\), \(\mathbf x=(x_1,\ldots,x_N)\), and
\(Z_{L,N}^{\mathrm{dw}}\) is the normalizing constant.  Decomposing
\(\pi\) into cycles gives the usual cycle representation.  If \(n_j\) denotes
the number of cycles of length \(j\), then the cycle weights are
\[
  q_{L,j}^{\mathrm{dw}}
  :=
  \operatorname{Tr} e^{-\beta jK_L^{\mathrm{dw}}}
  =
  \sum_{k\ge0}e^{-\beta j\varepsilon_{k,L}}.
\]
Consequently, under the canonical measure,
\[
  \mathbb P_{L,N}^{\mathrm{dw}}
  \bigl(n_j=m_j,\ j\ge1\bigr)
  =
  \frac{1}{Z_{L,N}^{\mathrm{dw}}}
  \prod_{j\ge1}
  \frac{1}{m_j!}
  \left(
    \frac{q_{L,j}^{\mathrm{dw}}}{j}
  \right)^{m_j},
\]
for all sequences \((m_j)_{j\ge1}\) satisfying
\[
    \sum_{j\ge1}j\,m_j=N.
\]

The effective part of the trace is the contribution of the doublet
\[
    \varepsilon_{0,L}=0,
    \qquad
    \varepsilon_{1,L}=\Delta_L.
\]
Thus we define
\[
  q_{L,j}^{\mathrm{eff}}
  :=
  1+e^{-\beta j\Delta_L}.
\]
The remaining part is the background trace
\[
  q_{L,j}^{\mathrm{bg}}
  :=
  q_{L,j}^{\mathrm{dw}}-q_{L,j}^{\mathrm{eff}}
  =
  \sum_{k\ge2}e^{-\beta j\varepsilon_{k,L}}.
\]

Indeed, by introducing the finite-volume effective spectral measure
\[
  \Sigma_L^{\mathrm{dw}}
  :=
  \delta_0+\delta_{V_L\Delta_L},
\]
we may write
\[
  q_{L,j}^{\mathrm{eff}}
  =
  \int_{[0,\infty)}
  e^{-\beta (j/V_L)\lambda}
  \,\Sigma_L^{\mathrm{dw}}(d\lambda).
\]
By \(\Cref{ass:critical-double-well-scaling}\),
\[
  \Sigma_L^{\mathrm{dw}}
  \Longrightarrow
  \Sigma^{\mathrm{dw}}
  :=
  \delta_0+\delta_\gamma
\]
vaguely on \([0,\infty)\).  Hence, uniformly for \(j/V_L\) in compact subsets
of \((0,\infty)\),
\[
  q_{L,j}^{\mathrm{eff}}
  -
  \phi_{\mathrm{dw}}(j/V_L)
  \longrightarrow0,
\]
where
\[
  \phi_{\mathrm{dw}}(x)
  :=
  \int_{[0,\infty)}
  e^{-\beta x\lambda}
  \,\Sigma^{\mathrm{dw}}(d\lambda)
  =
  1+e^{-\beta\gamma x},
  \qquad x>0.
\]
This scalar profile is the macroscopic trace factor that will appear in the
limiting marked Poisson--Kingman intensity.

\begin{assumption}[Double-well background density concentration]
\label{ass:double-well-background-density-concentration}
The background traces
\[
  q_{L,j}^{\mathrm{bg}}
  =
  \sum_{k\ge2}e^{-\beta j\varepsilon_{k,L}},
  \qquad j\ge1,
\]
satisfy \(\Cref{ass:background-density-concentration}\) with some $\kappa>0$ and limiting
background density \(\rho_{\mathrm{bg}}\).  Consequently, the modes
\(k\ge2\) do not contribute to the effective macroscopic mark; their only
macroscopic effect is the deterministic density \(\rho_{\mathrm{bg}}\).
\end{assumption}

\begin{remark}[Example of the critical double-well scaling]
The scaling assumption \(V_L\Delta_L\to\gamma\) is natural in semiclassical
double-well theory.  For a smooth confining double-well potential with two
non-degenerate minima, the corresponding semiclassical Schrödinger operator
\[
  H_h^{\mathrm{sc}}
  =
  -h^2\Delta+V_{\mathrm{dw}}
\]
has a ground-state doublet whose splitting
\[
  \Delta(h):=E_1(h)-E_0(h)
\]
is exponentially small in \(h\); see, for example,
\cite{DimassiSjostrand1999} and \cite{HelfferSjostrand1984}.  Choosing \(h=h_L\) so that
\[
  V_L\Delta(h_L)\longrightarrow\gamma
\]
produces precisely the spectral structure in
\(\Cref{ass:critical-double-well-scaling}\): two levels visible on the
macroscopic cycle scale and higher modes invisible to the effective trace.
The arguments below use only this spectral structure, not the detailed WKB
asymptotics.

A concrete three-dimensional realization is obtained as follows.  Let
\[
  \Lambda_L=(0,L)^3,
  \qquad
  V_L=L^3,
\]
and impose, for definiteness, Dirichlet boundary conditions.  Choose a compactly
supported attractive one-well potential \(w\) on \(\mathbb R^3\) such that
\[
  h_{\rm w}:=-\Delta+w
\]
has exactly one simple negative eigenvalue \(e_*<0\).  For example, one may
take a sufficiently shallow spherical square well such as
\[
  w(x)=-4\,\mathbf 1_{\{|x|<1\}}.
\]
This is a standard three-dimensional square-well example; see, for instance,
\cite{ReedSimonIV} and \cite{TeschlQM}.  Let
\[
  \kappa:=\sqrt{-e_*}.
\]
For two identical wells separated by a distance \(R\), define
\[
  W_R(x)
  :=
  w\left(x+\frac R2 e_1\right)
  +
  w\left(x-\frac R2 e_1\right),
  \qquad e_1=(1,0,0).
\]
The two-well operator on \(L^2(\mathbb R^3)\) has two eigenvalues
\(E_0(R)<E_1(R)<0\) near \(e_*\), and the standard tunnelling estimate gives
\[
  E_1(R)-E_0(R)
  =
  C_{\rm tun}\frac{e^{-\kappa R}}{R}(1+o(1)),
  \qquad R\to\infty,
\]
for some \(C_{\rm tun}>0\); see
\cite{Agmon,HarrellDoubleWells,HelfferSjostrand1984}.

Fix \(\gamma>0\) and choose \(R_L\) by
\[
  C_{\rm tun}L^3\frac{e^{-\kappa R_L}}{R_L}=\gamma .
\]
Then
\[
  R_L
  =
  \frac{3}{\kappa}\log L
  -
  \frac1\kappa\log\log L
  +
  O(1),
\]
so in particular \(R_L\to\infty\) and \(R_L=o(L)\).  Place the two wells in the
bulk of \(\Lambda_L\), for instance at
\[
  a_L=\left(\frac L2,\frac L2,\frac L2\right)-\frac{R_L}{2}e_1,
  \qquad
  b_L=\left(\frac L2,\frac L2,\frac L2\right)+\frac{R_L}{2}e_1,
\]
and set
\[
  U_L(x):=w(x-a_L)+w(x-b_L).
\]
Since the wells remain far from the boundary, standard exponential
localization estimates imply that the finite-volume splitting has the same
leading asymptotics as the infinite-volume two-well splitting:
\[
  E_{1,L}-E_{0,L}
  =
  C_{\rm tun}\frac{e^{-\kappa R_L}}{R_L}(1+o(1)).
\]
See, for example, \cite{Agmon,ReedSimonIV} for exponential localization and
finite-volume comparison arguments.  Therefore
\[
  V_L\varepsilon_{1,L}
  =
  L^3(E_{1,L}-E_{0,L})
  \longrightarrow
  \gamma.
\]

Moreover, because the one-well operator has only one negative eigenvalue, the
two-well operator has only two low-lying eigenvalues near \(e_*\).  The third
eigenvalue belongs to the background part of the spectrum and stays separated
from the doublet by an order-one gap.  Hence
\[
  \liminf_{L\to\infty}\varepsilon_{2,L}>0,
  \qquad
  V_L\varepsilon_{2,L}\longrightarrow\infty.
\]
Finally, the remaining modes satisfy the usual Weyl-type background limit; see,
for instance, \cite{DaviesSpectral} and \cite{ReedSimonIV}.
In particular, for suitable test functions \(F\),
\[
  \frac1{L^3}\sum_{r\ge2}F(E_{r,L}-E_{0,L})
  \longrightarrow
  \int_{\mathbb R^3}
  F(|p|^2-e_*)\,\frac{dp}{(2\pi)^3}.
\]
This gives a concrete model satisfying both
\(\Cref{ass:critical-double-well-scaling}\) and
\(\Cref{ass:double-well-background-density-concentration}\).
\end{remark}

\subsection{Two-state well-loop marks and the marked Poisson--Kingman limit}
\label{subsec:double-well-two-state-well-loop-marks}

We now enrich the macroscopic cycles by an internal two-state mark.  The
motivation is that, in the critical double-well regime, the two lowest
eigenstates remain visible on the scale \(j\sim V_L\), while the higher modes
are invisible to the macroscopic effective trace.  In the spectral basis the
doublet only contributes two scalar weights, but in the localized well basis it
also describes tunnelling between the two wells during the imaginary-time
evolution of a long cycle.  The mark introduced below describes this effective
two-state tunnelling history.  It is an effective low-energy mark, not a
functional of the full spatial Brownian bridge.

Let
\[
  \mathcal H_L^{\mathrm{eff}}
  :=
  \operatorname{span}\{\psi_{0,L},\psi_{1,L}\}.
\]
On this subspace the shifted Hamiltonian is diagonal in the spectral basis:
\[
  K_L^{\mathrm{dw}}\big|_{\mathcal H_L^{\mathrm{eff}}}
  =
  0\cdot |\psi_{0,L}\rangle\langle\psi_{0,L}|
  +
  \Delta_L\,|\psi_{1,L}\rangle\langle\psi_{1,L}|.
\]
For the well-loop description we use the localized basis
\[
  h_{L,-}
  :=
  \frac{\psi_{0,L}+\psi_{1,L}}{\sqrt2},
  \qquad
  h_{L,+}
  :=
  \frac{\psi_{0,L}-\psi_{1,L}}{\sqrt2}.
\]
We identify the two localized states with
\[
  \mathsf S:=\{-1,+1\}.
\]

In the basis \((h_{L,-},h_{L,+})\), the effective semigroup over a cycle of
length \(j\) is
\[
  P_{L,j}^{\mathrm{well}}
  =
  \frac12
  \begin{pmatrix}
    1+e^{-\beta j\Delta_L} & 1-e^{-\beta j\Delta_L} \\
    1-e^{-\beta j\Delta_L} & 1+e^{-\beta j\Delta_L}
  \end{pmatrix}.
\]
Thus
\[
  \operatorname{Tr} P_{L,j}^{\mathrm{well}}
  =
  1+e^{-\beta j\Delta_L}
  =
  q_{L,j}^{\mathrm{eff}}.
\]
Equivalently, \(P_{L,j}^{\mathrm{well}}\) is the time-one transition matrix of
the continuous-time Markov chain on \(\mathsf S\) with jump rate
\[
  r_{L,j}:=\frac{\beta j\Delta_L}{2}.
\]

We take the well-loop mark space to be
\[
  \mathsf M_{\mathrm{well}}
  :=
  D([0,1],\mathsf S),
\]
equipped with the Skorokhod topology.  For each \(L\) and \(j\), let
\((X_t)_{0\le t\le1}\) denote the above two-state chain with jump rate
\(r_{L,j}\).  Define the finite-volume unnormalised well-loop measure by
\[
  \mu_{L,j}^{\mathrm{well}}(F)
  :=
  \sum_{s\in\mathsf S}
  \mathbb E_s^{(L,j)}
  \left[
    F(X)\,\mathbf 1_{\{X_1=s\}}
  \right],
  \qquad
  F\in C_b(\mathsf M_{\mathrm{well}}),
\]
where \(\mathbb E_s^{(L,j)}\) denotes expectation for the chain started at
\(s\).  Its total mass is
\[
  \mu_{L,j}^{\mathrm{well}}(\mathsf M_{\mathrm{well}})
  =
  q_{L,j}^{\mathrm{eff}}
  =
  1+e^{-\beta j\Delta_L}.
\]
We write
\[
  \widehat\mu_{L,j}^{\mathrm{well}}
  :=
  \frac{\mu_{L,j}^{\mathrm{well}}}{q_{L,j}^{\mathrm{eff}}}
\]
for the corresponding normalized mark law.

For \(x>0\), let \((X_t^{(x)})_{0\le t\le1}\) be the two-state Markov chain on
\(\mathsf S\) with jump rate
\[
  r_x:=\frac{\beta\gamma x}{2}.
\]
Define the limiting unnormalised well-loop kernel by
\[
  \eta_x^{\mathrm{well}}(F)
  :=
  \sum_{s\in\mathsf S}
  \mathbb E_s^{(x)}
  \left[
    F(X^{(x)})\,\mathbf 1_{\{X_1^{(x)}=s\}}
  \right],
  \qquad
  F\in C_b(\mathsf M_{\mathrm{well}}).
\]
Thus \(\eta_x^{\mathrm{well}}\) is a finite measure, not a probability measure.
Its total mass is
\[
  \eta_x^{\mathrm{well}}(\mathsf M_{\mathrm{well}})
  =
  1+e^{-\beta\gamma x}
  =
  \phi_{\mathrm{dw}}(x).
\]
We also set
\[
  \widehat\eta_x^{\mathrm{well}}
  :=
  \frac{\eta_x^{\mathrm{well}}}{\phi_{\mathrm{dw}}(x)}.
\]

\begin{proposition}[Well-loop one-cycle convergence]
\label{prop:double-well-well-loop-one-cycle-convergence}
Assume~\cref{ass:critical-double-well-scaling}.  Then, for every
\(0<\delta<R<\infty\) and every
\(F\in C_b(\mathsf M_{\mathrm{well}})\),
\[
  \sup_{\delta\le j/V_L\le R}
  \left|
  \mu_{L,j}^{\mathrm{well}}(F)
    -
 \eta_{j/V_L}^{\mathrm{well}}(F)
  \right|
  \longrightarrow0.
\]
Moreover,
\[
  \sup_{\delta\le j/V_L\le R}
  \left|
 \widehat\mu_{L,j}^{\mathrm{well}}(F)
    -
\widehat\eta_{j/V_L}^{\mathrm{well}}(F)
  \right|
  \longrightarrow0.
\]
\end{proposition}

\begin{proof}
For \(j/V_L\in[\delta,R]\),
\[
  r_{L,j}
  =
  \frac{\beta}{2}\frac{j}{V_L}\,V_L\Delta_L.
\]
Since \(V_L\Delta_L\to\gamma\), we have
\[
  \sup_{\delta\le j/V_L\le R}
  \left|
    r_{L,j}-r_{j/V_L}
  \right|
  \le
  \frac{\beta R}{2}
  \left|
    V_L\Delta_L-\gamma
  \right|
  \longrightarrow0.
\]
Because the state space \(\mathsf S\) is finite, convergence of the jump rates
implies uniform convergence, in total variation on \(D([0,1],\mathsf S)\), of
the corresponding path laws on every compact range \(j/V_L\in[\delta,R]\).
Applying this to the bounded measurable functional
\[
  X\mapsto F(X)\mathbf 1_{\{X_1=X_0\}},
\]
and summing over \(s\in\mathsf S\), gives the unnormalised convergence.

The total masses satisfy, uniformly for \(j/V_L\in[\delta,R]\),
\[
  q_{L,j}^{\mathrm{eff}}
  =
  1+e^{-\beta j\Delta_L}
  \longrightarrow
  1+e^{-\beta\gamma j/V_L}
  =
  \phi_{\mathrm{dw}}(j/V_L).
\]
Both \(q_{L,j}^{\mathrm{eff}}\) and \(\phi_{\mathrm{dw}}(j/V_L)\) are bounded
below by \(1\).  Dividing the unnormalised convergence by the total masses
therefore gives the normalized convergence.
\end{proof}

The effective spectral measure associated with the doublet is
\[
  \Sigma^{\mathrm{dw}}
  =
  \delta_0+\delta_\gamma.
\]
Accordingly,
\[
  \phi_{\mathrm{dw}}(x)
  =
  \int_{[0,\infty)}
  e^{-\beta x\lambda}\,
  \Sigma^{\mathrm{dw}}(d\lambda)
  =
  1+e^{-\beta\gamma x}.
\]

\begin{proposition}[Verification of the abstract assumptions]
\label{prop:double-well-well-loop-verification}
Assume~\cref{ass:critical-double-well-scaling} and
\(\cref{ass:double-well-background-density-concentration}\).  Then all the assumptions of~\cref{thm:main-canonical-bridge-limit} holds with
\[
  \mathsf M=\mathsf M_{\mathrm{well}},
  \qquad
  \eta_x=\eta_x^{\mathrm{well}},
  \qquad
  \Sigma=\Sigma^{\mathrm{dw}},
  \qquad
  \phi=\phi_{\mathrm{dw}}.
\]
\end{proposition}

\begin{proof}
First, the map \(x\mapsto\eta_x^{\mathrm{well}}\) is weakly continuous because
the two-state path law depends continuously on the jump rate
\(r_x=\beta\gamma x/2\).  Moreover,
\[
  \eta_x^{\mathrm{well}}(\mathsf M_{\mathrm{well}})
  =
  1+e^{-\beta\gamma x}
  =
  \int_{[0,\infty)}
  e^{-\beta x\lambda}\,
  (\delta_0+\delta_\gamma)(d\lambda).
\]
Since \(0<\phi_{\mathrm{dw}}(x)\le2\), for every \(\kappa>0\),
\[
  \int_0^\infty
  (1\wedge x)e^{-\kappa x}
  \frac{\phi_{\mathrm{dw}}(x)}{x}\,dx
  \le
  2\int_0^\infty
  (1\wedge x)e^{-\kappa x}\frac{dx}{x}
  <\infty.
\]
Thus the~\cref{ass:limiting-effective-kernel} holds.

The tilted total effective mass has the same law as the sum of two independent
Gamma variables with rates \(\kappa\) and \(\kappa+\beta\gamma\).  Indeed, the
scalar Lévy density is
\[
  e^{-\kappa x}\frac{\phi_{\mathrm{dw}}(x)}{x}\,dx
  =
  e^{-\kappa x}\frac{dx}{x}
  +
  e^{-(\kappa+\beta\gamma)x}\frac{dx}{x}.
\]
Consequently the tilted total mass has a continuous density on
\((0,\infty)\), and the~\cref{ass:limiting-effective-density} is satisfied.

Second, \cref{prop:double-well-well-loop-one-cycle-convergence} gives the
effective one-cycle convergence.  The additional \(u\)-coordinate in the
abstract marked trace assumption is harmless, since it is integrated against
Lebesgue measure on \([0,1]\).  The uniform trace bound follows from
\[
  q_{L,j}^{\mathrm{eff}}
  =
  1+e^{-\beta j\Delta_L}
  \le2.
\]
Hence the~\cref{ass:effective-marked-trace-convergence} holds.

Third, define
\[
  \Sigma_L^{\mathrm{dw}}
  :=
  \delta_0+\delta_{V_L\Delta_L}.
\]
Then, for every \(L\) and \(j\ge1\),
\[
  q_{L,j}^{\mathrm{eff}}
  =
  1+e^{-\beta j\Delta_L}
  =
  \int_{[0,\infty)}
  e^{-\beta(j/V_L)\lambda}\,
  \Sigma_L^{\mathrm{dw}}(d\lambda).
\]
The total mass is
\[
  \Theta_L
  =
  \Sigma_L^{\mathrm{dw}}([0,\infty))
  =
  2,
\]
so the lower bound condition in the absolute case holds with any
\(1<\Theta_*<2\).  The compact convergence of the Laplace transforms follows
from \(V_L\Delta_L\to\gamma\) and
\(\lfloor xV_L\rfloor/V_L\to x\) uniformly on compact subsets of
\((0,\infty)\).  Finally,
\[
  \sup_L
  \int_{[0,\infty)}
  \log(1+\kappa+\beta\lambda)\,
  \Sigma_L^{\mathrm{dw}}(d\lambda)
  <\infty
\]
because \(V_L\Delta_L\to\gamma<\infty\).  Hence the~\cref{ass:effective-spectral-LLT-criterion} through the absolute condition (A).

Finally, the~\cref{ass:background-density-concentration} is exactly
\cref{ass:double-well-background-density-concentration}.
\end{proof}

For an effective cycle of
length \(j\), we attach an independent mark $M_{j,r}^{\mathrm{well}}$ with law $\widehat\mu_{L,j}^{\mathrm{well}}$. Define
\[
  \Xi_{L,N_L}^{\mathrm{well}}
  :=
  \sum_{j\ge1}\sum_{r=1}^{n_j}
  \delta_{\left(U_{j,r},\,j/V_L,\,M_{j,r}^{\mathrm{well}}\right)}
\]
as the finite volume canonical marked point process on
\[
  E_{\mathrm{well}}
  :=
  [0,1]\times(0,\infty)\times\mathsf M_{\mathrm{well}}.
  \]

For any $\kappa>0,$ $a>0$, let
\[
  \Pi_{\mathrm{well}}^{(\kappa)}
  \sim
  \operatorname{PPP}\bigl(\nu_{\mathrm{well}}^{(\kappa)}\bigr)
\]
with intensity
\[
  \nu_{\mathrm{well}}^{(\kappa)}(du,dx,dm)
  =
  du\,e^{-\kappa x}\frac{dx}{x}\,
  \eta_x^{\mathrm{well}}(dm),
  \]
  and $\Pi_{a}^{\mathrm{well}\text{-}\mathrm{br}}$ be the bridge  of
\(\Pi_{\mathrm{well}}^{(\kappa)}\) conditioned, in the density sense, on
\[
  \int_{E_{\mathrm{well}}}
  x\,
  \Pi_{\mathrm{well}}^{(\kappa)}(du,dx,dm)
  =
  a.
\]
By applying~\cref{thm:main-canonical-bridge-limit} and~\cref{cor:ranked-effective-length-convergence}, we derive the main results in this double-well model. We remark that the limit  of the length point process is not a Gamma bridge when $\gamma\neq 0$ since the two rates $0$ and $\beta \gamma$ are distinct. In particular,  the ranked lengths are not governed by
a Poisson-Dirichlet law. However, when $\gamma=0$, its ranked proportions instead have law $PD(0,2)$. 

\begin{corollary}[Double-well well-loop marked Poisson--Kingman bridge]
\label{cor:double-well-well-loop-pk-limit}
Assume~\cref{ass:critical-double-well-scaling} and
\(\cref{ass:double-well-background-density-concentration}\).  If
\(N_L/V_L\to\rho>\rho_{\mathrm{bg}}\), then
\[
  \Xi_{L,N_L}^{\mathrm{well}}
  \Longrightarrow
  \Pi_{\rho-\rho_{\mathrm{bg}}}^{\mathrm{well}\text{-}\mathrm{br}}
  \qquad
  \text{in }\mathcal N_\ell(E_{\mathrm{well}}).
\]

The limiting ranked lengths \((X_i)_{i\ge1}\) are the ranked jumps of the
Poisson--Kingman bridge with scalar profile
\[
  \phi_{\mathrm{dw}}(x)=1+e^{-\beta\gamma x}
\]
conditioned to have total mass \(\rho-\rho_{\mathrm{bg}}\).    Conditionally on the ranked lengths, the auxiliary
variables \(U_i\) are independent uniform variables on \([0,1]\), and the
well-loop marks are conditionally independent with
\[
  M_i^{\mathrm{well}}
  \sim
  \widehat\eta_{X_i}^{\mathrm{well}}
  =
  \frac{\eta_{X_i}^{\mathrm{well}}}
       {1+e^{-\beta\gamma X_i}}.
\]

\end{corollary}

\subsection{Spectral-label marks and finite-type band extensions}
\label{subsec:spectral-labels-finite-type-band-extensions}

\subsubsection{Double-well spectral-label marks and finite-type band extensions}
\label{subsubsec:spectral-labels-model-description}

The well-loop mark constructed above describes the effective tunnelling history
inside the two-dimensional doublet.  There is also a simpler, purely spectral
mark: it describes only the eigenvalue label of the effective mode used by the
cycle.  This mark loses the pathwise well interpretation, but it makes the
finite-type low-energy structure transparent and extends naturally from the
doublet to finitely many effective bands.

We formulate this directly in finite-type form.  Fix \(Q<\infty\), let
\[
  \mathsf M_{\mathrm{ft}}:=\{1,\ldots,Q\}
\]
with the discrete topology, and let
$
  \theta_1,\ldots,\theta_Q>0
$
be fixed band weights.  Assume that the effective finite-volume band
parameters satisfy
\begin{equation}
  \lambda_{L,r}\longrightarrow\lambda_r\in[0,\infty),
  \qquad r=1,\ldots,Q.
  \label{eq:finite-type-band-parameter-convergence}
\end{equation}
For a cycle of length \(j\), define the finite-volume unnormalised
spectral-label kernel by
\[
  \mu_{L,j}^{\mathrm{ft}}
  :=
  \sum_{r=1}^Q
  \theta_r
  e^{-\beta\lambda_{L,r}j/V_L}
  \delta_r,
  \qquad j\ge1.
\]
Its total effective trace is
\[
  q_{L,j}^{\mathrm{ft}}
  :=
  \mu_{L,j}^{\mathrm{ft}}(\mathsf M_{\mathrm{ft}})
  =
  \sum_{r=1}^Q
  \theta_r
  e^{-\beta\lambda_{L,r}j/V_L}.
\]
We also write
$
  \widehat\mu_{L,j}^{\mathrm{ft}}
  :=
{\mu_{L,j}^{\mathrm{ft}}}/{q_{L,j}^{\mathrm{ft}}}
$
for the normalized finite-volume spectral-label law.

For \(x>0\), the limiting unnormalised spectral-label kernel is
\[
  \eta_x^{\mathrm{ft}}
  :=
  \sum_{r=1}^Q
  \theta_r
  e^{-\beta\lambda_r x}
  \delta_r.
\]
Its total mass is the scalar profile
\[
  \phi_{\mathrm{ft}}(x)
  :=
  \eta_x^{\mathrm{ft}}(\mathsf M_{\mathrm{ft}})
  =
  \sum_{r=1}^Q
  \theta_r
  e^{-\beta\lambda_r x}.
\]
We define the normalized limiting spectral-label law by
$
  \widehat\eta_x^{\mathrm{ft}}
  :=
{\eta_x^{\mathrm{ft}}}/{\phi_{\mathrm{ft}}(x)}$.

Let
\[
  \Sigma^{\mathrm{ft}}
  :=
  \sum_{r=1}^Q
  \theta_r\delta_{\lambda_r}.
\]
Then
\[
  \phi_{\mathrm{ft}}(x)
  =
  \int_{[0,\infty)}
  e^{-\beta x\lambda}\,
  \Sigma^{\mathrm{ft}}(d\lambda).
\]

The double-well spectral-label model is the case \(Q=2\), with
$ \theta_1=\theta_2=1$,
$  \lambda_{L,1}=0$,
$  \lambda_{L,2}=V_L\Delta_L$.
Under \(V_L\Delta_L\to\gamma\), its limiting profile is
\[
  \phi_{\mathrm{dw}}(x)=1+e^{-\beta\gamma x}.
\]

\begin{remark}[Path marks versus spectral labels]
\label{rem:path-marks-versus-spectral-labels}
The well-loop mark and the spectral-label mark encode different information
but share the same scalar trace.  The well-loop mark remembers the effective
two-state trajectory in the localized well basis, whereas the spectral label
remembers only the eigenmode in the diagonal spectral basis.  Therefore the two
constructions have the same effective total mass distribution and the same
local-limit normalization, but their limiting point processes live on different
mark spaces and have different conditional mark distributions.
\end{remark}

\subsubsection{Verification of the assumptions and the canonical limit}
\label{subsubsec:spectral-labels-verification-and-limit}

We now verify the abstract assumptions for the finite-type kernel.  The
double-well spectral-label construction follows by the \(Q=2\) specialization
above.

\begin{lemma}[Finite-type one-cycle convergence]
\label{lem:finite-type-one-cycle-convergence}
Assume \eqref{eq:finite-type-band-parameter-convergence}.  Then, for every
\(0<\delta<R<\infty\) and every
\[
  F\in
  C_b\bigl([0,1]\times[\delta,R]\times\mathsf M_{\mathrm{ft}}\bigr),
\]
one has
\[
  \sup_{x\in[\delta,R]}
  \Bigg|
  \int_0^1
  \int_{\mathsf M_{\mathrm{ft}}}
  F(u,x,m)\,
  \mu_{L,\lfloor xV_L\rfloor}^{\mathrm{ft}}(dm)\,du
  -
  \int_0^1
  \int_{\mathsf M_{\mathrm{ft}}}
  F(u,x,m)\,
  \eta_x^{\mathrm{ft}}(dm)\,du
  \Bigg|
  \longrightarrow0 .
\]
Moreover,
\[
  \sup_{L\ge1}\sup_{j\ge1}
  q_{L,j}^{\mathrm{ft}}
  \le
  \sum_{r=1}^Q\theta_r
  <\infty .
\]
\end{lemma}

\begin{proof}
Let \(j_L(x):=\lfloor xV_L\rfloor\).  For each \(r=1,\ldots,Q\),
\[
  \lambda_{L,r}\frac{j_L(x)}{V_L}
  \longrightarrow
  \lambda_r x
\]
uniformly for \(x\in[\delta,R]\).  Indeed,
\[
  \sup_{x\in[\delta,R]}
  \left|
    \lambda_{L,r}\frac{j_L(x)}{V_L}
    -
    \lambda_r x
  \right|
  \le
  R|\lambda_{L,r}-\lambda_r|
  +
  |\lambda_r|
  \sup_{x\in[\delta,R]}
  \left|
    \frac{j_L(x)}{V_L}-x
  \right|,
\]
and the right-hand side tends to zero.  Hence
\[
  \max_{1\le r\le Q}
  \sup_{x\in[\delta,R]}
  \left|
    e^{-\beta\lambda_{L,r}j_L(x)/V_L}
    -
    e^{-\beta\lambda_r x}
  \right|
  \longrightarrow0.
\]
Since the mark space is finite and \(F\) is bounded, summing over
\(r=1,\ldots,Q\) gives the stated convergence.  The uniform trace bound follows
from
\[
  q_{L,j}^{\mathrm{ft}}
  =
  \sum_{r=1}^Q
  \theta_r e^{-\beta\lambda_{L,r}j/V_L}
  \le
  \sum_{r=1}^Q\theta_r .
\]
\end{proof}

\begin{proposition}[Verification of the finite-type abstract assumptions]
\label{prop:finite-type-abstract-verification}
Assume \eqref{eq:finite-type-band-parameter-convergence}, and set
\[
  \Theta:=\sum_{r=1}^Q\theta_r .
\]
If \(\Theta\ge1\), then the finite-type spectral-label part satisfies
\[
  \mathsf M=\mathsf M_{\mathrm{ft}},
  \qquad
  \eta_x=\eta_x^{\mathrm{ft}},
  \qquad
  \Sigma=\Sigma^{\mathrm{ft}},
  \qquad
  \phi=\phi_{\mathrm{ft}},
\]
and the 
  \cref{ass:limiting-effective-kernel}, \cref{ass:limiting-effective-density}, \cref{ass:effective-marked-trace-convergence} and \cref{ass:effective-spectral-LLT-criterion} 
hold.  More precisely, if \(\Theta>1\), the spectral local-limit criterion
holds through the absolute case \(\mathrm{(A)}\), while if
\(\Theta=1\), it holds through the critical finite-type case
\(\mathrm{(B)}\).
\end{proposition}

\begin{proof}
The map \(x\mapsto\eta_x^{\mathrm{ft}}\) is weakly continuous because
\(\mathsf M_{\mathrm{ft}}\) is finite and each coefficient
\(x\mapsto e^{-\beta\lambda_r x}\) is continuous.  Moreover,
\[
  \eta_x^{\mathrm{ft}}(\mathsf M_{\mathrm{ft}})
  =
  \sum_{r=1}^Q
  \theta_r e^{-\beta\lambda_r x}
  =
  \int_{[0,\infty)}
  e^{-\beta x\lambda}\,
  \Sigma^{\mathrm{ft}}(d\lambda).
\]
Since \(0<\phi_{\mathrm{ft}}(x)\le\Theta\), for every \(\kappa>0\),
\[
  \int_0^\infty
  (1\wedge x)e^{-\kappa x}
  \frac{\phi_{\mathrm{ft}}(x)}{x}\,dx
  \le
  \Theta
  \int_0^\infty
  (1\wedge x)e^{-\kappa x}\frac{dx}{x}
  <\infty.
\]
Thus the~\cref{ass:limiting-effective-kernel} holds.

We next verify the density assumption for the limiting effective mass.  Under
the \(\kappa\)-tilted limiting Poisson process, the length intensity is
\[
  e^{-\kappa x}\frac{\phi_{\mathrm{ft}}(x)}{x}\,dx
  =
  \sum_{r=1}^Q
  \theta_r
  e^{-(\kappa+\beta\lambda_r)x}
  \frac{dx}{x}.
\]
The marked Poisson process decomposes into \(Q\) independent components,
one for each label \(r\).  Let \(T_r^{(\kappa)}\) be the total mass of the
\(r\)-th component.  Then
\[
  T^{(\kappa)}
  :=
  \int_{[0,1]\times(0,\infty)\times\mathsf M_{\mathrm{ft}}}
  x\,\Pi_{\mathrm{ft}}^{(\kappa)}(du,dx,dm)
  =
  \sum_{r=1}^Q T_r^{(\kappa)}
\]
with independent summands.  For \(s\ge0\),
\[
\begin{aligned}
  \mathbb E e^{-sT_r^{(\kappa)}}
  &=
  \exp\left\{
    -\theta_r
    \int_0^\infty
    (1-e^{-sx})
    e^{-(\kappa+\beta\lambda_r)x}
    \frac{dx}{x}
  \right\}  \\
  &=
  \left(
    \frac{\kappa+\beta\lambda_r}
         {\kappa+\beta\lambda_r+s}
  \right)^{\theta_r}.
\end{aligned}
\]
Hence
\[
  T_r^{(\kappa)}
  \sim
  \mathrm{Gamma}\bigl(\theta_r,\kappa+\beta\lambda_r\bigr),
\]
where the second parameter is the rate.  Therefore
\[
  T^{(\kappa)}
  \stackrel{d}{=}
  \sum_{r=1}^Q
  \mathrm{Gamma}\bigl(\theta_r,\kappa+\beta\lambda_r\bigr)
\]
as a sum of independent Gamma random variables.  Each summand has a continuous
density on \((0,\infty)\), and the finite convolution of these densities is
again continuous on \((0,\infty)\).  In fact it is strictly positive on
\((0,\infty)\).  Thus the limiting effective mass has a continuous density
\(f_0^{(\kappa)}\) on \((0,\infty)\), and the~\cref{ass:limiting-effective-density} holds.

The~\cref{ass:effective-marked-trace-convergence} follows from
\cref{lem:finite-type-one-cycle-convergence}.

For the spectral local-limit criterion, define
\[
  \Sigma_L^{\mathrm{ft}}
  :=
  \sum_{r=1}^Q
  \theta_r\delta_{\lambda_{L,r}}.
\]
Then, for every \(j\ge1\),
\[
  q_{L,j}^{\mathrm{ft}}
  =
  \int_{[0,\infty)}
  e^{-\beta(j/V_L)\lambda}\,
  \Sigma_L^{\mathrm{ft}}(d\lambda).
\]
If \(\Theta>1\), then
\[
  \Sigma_L^{\mathrm{ft}}([0,\infty))=\Theta,
\]
so the lower-mass condition in the absolute case holds with any
\(1<\Theta_*<\Theta\).  The compact convergence of the Laplace transforms
follows from \(\lambda_{L,r}\to\lambda_r\), and the logarithmic moment condition
is automatic because \(Q<\infty\) and the sequences \(\lambda_{L,r}\) are
bounded.

If \(\Theta=1\), then
\[
  q_{L,j}^{\mathrm{ft}}
  =
  \sum_{r=1}^Q
  \theta_r e^{-\beta\lambda_{L,r}j/V_L}
\]
is exactly the critical finite-type representation in the critical
finite-type case.  This proves the proposition.
\end{proof}

For an effective cycle of length \(j\), attach an independent spectral-label
mark \(M_{j,\ell}^{\mathrm{ft}}\) with law
$
  \widehat\mu_{L,j}^{\mathrm{ft}}$.
Define the finite-volume canonical marked point process by
\[
  \Xi_{L,N_L}^{\mathrm{ft}}
  :=
  \sum_{j\ge1}\sum_{\ell=1}^{n_j}
  \delta_{\left(U_{j,\ell},\,j/V_L,\,M_{j,\ell}^{\mathrm{ft}}\right)}
\]
on
$
  E_{\mathrm{ft}}
  :=
  [0,1]\times(0,\infty)\times\mathsf M_{\mathrm{ft}}.
  $
  
For any \(\kappa>0\), let
\[
  \Pi_{\mathrm{ft}}^{(\kappa)}
  \sim
  \operatorname{PPP}\bigl(\nu_{\mathrm{ft}}^{(\kappa)}\bigr)
\]
with intensity
\[
  \nu_{\mathrm{ft}}^{(\kappa)}(du,dx,dm)
  =
  du\,e^{-\kappa x}\frac{dx}{x}\,
  \eta_x^{\mathrm{ft}}(dm).
\]
For \(a>0\), let \(\Pi_a^{\mathrm{ft}\text{-}\mathrm{br}}\) be the bridge of
\(\Pi_{\mathrm{ft}}^{(\kappa)}\) conditioned, in the density sense, on
\[
  \int_{E_{\mathrm{ft}}}
  x\,\Pi_{\mathrm{ft}}^{(\kappa)}(du,dx,dm)
  =
  a.
\]
By the preceding proposition, the required density exists and is positive for
all \(a>0\).  The resulting bridge law is independent of the auxiliary tilt
parameter \(\kappa>0\).

By applying \cref{thm:main-canonical-bridge-limit} and
\cref{cor:ranked-effective-length-convergence}, we obtain the following
finite-type spectral-label limit.  When \(\phi_{\mathrm{ft}}\) is non-constant,
the limiting length bridge is not a Gamma bridge.  In particular, in the
double-well spectral-label case with \(V_L\Delta_L\to\gamma>0\), the ranked
lengths are not governed by \(\mathrm{PD}(0,1)\).

\begin{corollary}[Finite-type spectral-label marked Poisson--Kingman bridge]
\label{cor:finite-type-spectral-label-pk-limit}
Assume \eqref{eq:finite-type-band-parameter-convergence},
\[
  \Theta:=\sum_{r=1}^Q\theta_r\ge1,
\]
and \(\cref{ass:background-density-concentration}\).  If
\(N_L/V_L\to\rho>\rho_{\mathrm{bg}}\), then
\[
  \Xi_{L,N_L}^{\mathrm{ft}}
  \Longrightarrow
  \Pi_{\rho-\rho_{\mathrm{bg}}}^{\mathrm{ft}\text{-}\mathrm{br}}
  \qquad
  \text{in }\mathcal N_\ell(E_{\mathrm{ft}}).
\]
The limiting ranked lengths \((X_i)_{i\ge1}\) are the ranked jumps of the
Poisson--Kingman bridge with scalar profile
\[
  \phi_{\mathrm{ft}}(x)
  =
  \sum_{r=1}^Q
  \theta_r e^{-\beta\lambda_r x}
\]
conditioned to have total mass \(\rho-\rho_{\mathrm{bg}}\).  Conditionally on
the ranked lengths, the auxiliary variables \(U_i\) are independent uniform
variables on \([0,1]\), and the spectral labels are conditionally independent
with
\[
  \mathbb P\bigl(M_i^{\mathrm{ft}}=r\,\big|\,X_i\bigr)
  =
  \frac{
    \theta_r e^{-\beta\lambda_r X_i}
  }{
    \sum_{s=1}^Q \theta_s e^{-\beta\lambda_s X_i}
  },
  \qquad r=1,\ldots,Q.
  \]
  If all positive weight $\lambda_r$ equal a common value $\lambda_*$, then $\phi_{\mathrm ft}(x)=\Theta e^{-\beta \lambda_* x}$ and the ranked normalized length have law $PD(0,\Theta)$. If at least two positive weight are distinct, the limit law is not a Poisson Dirichlet type.
\end{corollary}

%% file: 06-proofs.tex
\section{Proofs of the main results}
\label{sec:proofs-canonical-framework}

We prove the main canonical bridge limit theorem
(\cref{thm:main-canonical-bridge-limit})
and~\cref{cor:ranked-effective-length-convergence}.
The argument proceeds through the following steps.
After introducing a tilted grand-canonical Poisson representation
(\cref{subsec:proof-poisson-decomposition}), we decompose the
intensity into effective and background parts
(\cref{subsec:proof-effective-background-splitting}).
Unconditioned Poisson convergence of the effective process is
proved in~\cref{subsec:proof-mark-unconditioned-convergence}.
An effective local limit theorem is established separately under
the absolute and the critical finite-type criteria
(\cref{subsec:proof-effective-local-limit}).
The background part is shown to be deterministically concentrated
and invisible (\cref{subsec:proof-background-concentration}).
Combining these results with a bridge identity yields the
canonical convergence
(\cref{subsec:proof-bridge-convergence}).

The strategy of representing a canonical law via a conditional
Poisson process and then applying a local limit theorem to pass
from the grand-canonical to the canonical ensemble has been used
to analyse condensation phase transitions in other probabilistic
models by the author: in~\cite{SunSPA2026} for reversible
coagulation--fragmentation processes, and
in~\cite{SunAIHP2025} for sparse Erd\H{o}s--R\'enyi random
graphs.  In both of those settings, the condensed mass
concentrates on a \emph{single} macroscopic particle (or
component).  By contrast, in the ideal Bose gas the condensate
is dispersed among infinitely many macroscopic cycles, and the
conditioning produces a Poisson--Kingman \emph{bridge} rather
than a single large component.  This structural difference is
the main source of the new technical difficulties addressed in
the subsections below.
\subsection{A conditional Poisson Point process representation for the canonical law}\label{subsec:proof-poisson-decomposition}
We introduce a conditional Poisson point process representation for the law of our canonical marked cycle point process $\Xi_{L,N}$ defined in~\cref{subsec:marked-canonical-process}. This representation is the cornerstone of our proof. We refer to Kallenberg~\cite{Kallenberg2017} and Daley--Vere-Jones~\cite{DaleyVereJones2003} for the basic properties on the Poisson point processes.

Fix \(\kappa>0\) and set
$ z_L^{(\kappa)}:=e^{-\kappa/V_L}$.
Define the tilted grand-canonical intensity on
$       E=[0,1]\times(0,\infty)\times\mathsf M$
by
\[
        \nu_L^{(\kappa)}(du,dx,dm)
        :=
        \sum_{j\ge1}
        \frac{e^{-\kappa j/V_L}}{j}\,
        du\,\delta_{j/V_L}(dx)\,\mu_{L,j}(dm).
\]
Let $ \Pi_L^{(\kappa)}
        \sim
        \operatorname{PPP}\bigl(\nu_L^{(\kappa)}\bigr)$
be the Poisson point process with this intensity.  We write
\(\mathbf P_L^{(\kappa)}\) and \(\mathbf E_L^{(\kappa)}\) for its law and
expectation.

In terms of cycle lengths, the number of cycles of length \(j\) under
\(\mathbf P_L^{(\kappa)}\) is a Poisson random variable with mean
$       \frac{e^{-\kappa j/V_L}}{j}\,q_{L,j}$, denoted by $N_{L,j}^{(\kappa)}$.
These random variables are independent over \(j\).  Conditionally on the
number of cycles of length \(j\), their marks $m_{j,\cdot}$ are independent with law
\(J_{L,j}\), and their auxiliary time coordinates $U_{j,\cdot}$ are independent uniform
variables on \([0,1]\). Hence,
\[
        \Pi_L^{(\kappa)}
        =
        \sum_{j\ge1}
        \sum_{\ell=1}^{N_{L,j}^{(\kappa)}}
        \delta_{\left(U_{j,\ell},\,j/V_L,\,m_{j,\ell}\right)}.
\]

Define the total particle number under this tilted grand-canonical law by
\[
        S_L^{(\kappa)}
        :=
        V_L\int_E x\,\Pi_L^{(\kappa)}(du,dx,dm).
\]
Since every atom has length coordinate \(j/V_L\), the random variable
\(S_L^{(\kappa)}\) is integer-valued and 
\[
        S_L^{(\kappa)}
        =
        \sum_{j\ge1} j\,N_{L,j}^{(\kappa)}.
\]

The Laplace functional of \(\Pi_L^{(\kappa)}\) is
\[
        \mathbf E_L^{(\kappa)}
        \left[
            e^{-\langle F,\Pi_L^{(\kappa)}\rangle}
        \right]
        =
        \exp\left\{
            -
            \int_E
            \left(1-e^{-F(u,x,m)}\right)
            \nu_L^{(\kappa)}(du,dx,dm)
        \right\}
\]
for every non-negative measurable \(F:E\to[0,\infty)\).

\begin{lemma}[Finiteness of the finite-volume tilted intensity]
For every \(\kappa>0\),
\[
        \Lambda_L^\kappa
        :=
        \sum_{j\ge 1}
        e^{-\kappa j/V_L}\frac{q_{L,j}}{j}
        <\infty .
\]
Consequently, the finite-volume tilted Poisson point process
\(\Pi_L^\kappa\) is well defined.
\end{lemma}

\begin{proof}
By definition,
\[
        q_{L,j}
        =
        \operatorname{Tr}_{\mathcal H_L}(e^{-\beta jK_L}).
\]
Since \(K_L\ge 0\), the spectral theorem gives
\[
        e^{-\beta jK_L}\le e^{-\beta K_L},
        \qquad j\ge 1,
\]
in the sense of positive operators.  Hence
\[
        q_{L,j}
        =
        \operatorname{Tr}(e^{-\beta jK_L})
        \le
        \operatorname{Tr}(e^{-\beta K_L})
        =
        q_{L,1}
        <\infty .
\]
Therefore
\[
        \Lambda_L^\kappa
        \le
        q_{L,1}
        \sum_{j\ge 1}\frac{e^{-\kappa j/V_L}}{j}.
\]
The last series is finite because \(\kappa/V_L>0\).  This proves the
claim.
\end{proof}

Now we show that the canonical law of $\Xi_{L,N}$ can be obtained by conditioning this tilted Poisson process on
the total particle number.

\begin{lemma}[The conditional PPP representation]
\label{lem:canonical-conditioning}
For every \(N\in\mathbb N\) such that \(Z_{L,N}>0\),
\[
        \Xi_{L,N}
        \stackrel{d}{=}
        \Pi_L^{(\kappa)}
        \,|\,
        \{S_L^{(\kappa)}=N\}.
        \]
The conditional law does not depend on the chosen value of \(\kappa>0\). In particular,
\[
        \mathbf P_L^{(\kappa)}(S_L^{(\kappa)}=N)
        =
        \exp\{-\Lambda_L^{(\kappa)}\}
        e^{-\kappa N/V_L}
        Z_{L,N}.
\]
\end{lemma}

\begin{proof}
At the cycle count level, under \(\mathbf P_L^{(\kappa)}\), the probability of
a configuration \((n_j)_{j\ge1}\) is
\[
        \exp\{-\Lambda_L^{(\kappa)}\}
        \prod_{j\ge1}
        \frac{1}{n_j!}
        \left(
            \frac{e^{-\kappa j/V_L}q_{L,j}}{j}
        \right)^{n_j},
\]
where
\[
        \Lambda_L^{(\kappa)}
        :=
        \sum_{j\ge1}
        \frac{e^{-\kappa j/V_L}}{j}q_{L,j}.
\]
On the event
$ \sum_{j\ge1}j n_j=N$,
the tilt contributes the common factor
\[
        \prod_{j\ge1}
        e^{-\kappa j n_j/V_L}
        =
        e^{-\kappa N/V_L}.
\]
Hence, after conditioning on \(S_L^{(\kappa)}=N\), the cycle-count weight is
proportional to
\[
        \prod_{j\ge1}
        \frac{1}{n_j!}
        \left(\frac{q_{L,j}}{j}\right)^{n_j},
\]
which is exactly the canonical cycle-count law.
The marks and auxiliary time coordinates are conditionally independent with the
same kernels \(J_{L,j}\) and uniform laws as in the definition of
\(\Xi_{L,N}\).  Therefore the full marked point process has the canonical law.
The same computation also shows independence of \(\kappa\): on the event
\(S_L^{(\kappa)}=N\), the \(\kappa\)-dependent factor is the constant
\(e^{-\kappa N/V_L}\), which cancels in the conditional normalization.
\end{proof}

\subsection{Effective and background Poisson  decomposition}
\label{subsec:proof-effective-background-splitting}

We now apply the finite-volume decomposition in~\cref{subsec:finite-volume-effective-background}
\[
        \mu_{L,j}
        =
        \mu_{L,j}^{\mathrm{eff}}
        +
        \mu_{L,j}^{\mathrm{bg}}
\]
to the tilted grand-canonical intensity.  Define
\[
        \nu_{L,\mathrm{eff}}^{(\kappa)}(du,dx,dm)
        :=
        \sum_{j\ge1}
        \frac{e^{-\kappa j/V_L}}{j}\,
        du\,\delta_{j/V_L}(dx)\,
        \mu_{L,j}^{\mathrm{eff}}(dm),
\]
and
\[
        \nu_{L,\mathrm{bg}}^{(\kappa)}(du,dx,dm)
        :=
        \sum_{j\ge1}
        \frac{e^{-\kappa j/V_L}}{j}\,
        du\,\delta_{j/V_L}(dx)\,
        \mu_{L,j}^{\mathrm{bg}}(dm).
\]
Then
$   \nu_L^{(\kappa)}
        =
        \nu_{L,\mathrm{eff}}^{(\kappa)}
        +
        \nu_{L,\mathrm{bg}}^{(\kappa)}$.

Let
\[
        \Pi_{L,\mathrm{eff}}^{(\kappa)}
        \sim
        \operatorname{PPP}\bigl(\nu_{L,\mathrm{eff}}^{(\kappa)}\bigr),
        \qquad
        \Pi_{L,\mathrm{bg}}^{(\kappa)}
        \sim
        \operatorname{PPP}\bigl(\nu_{L,\mathrm{bg}}^{(\kappa)}\bigr)
\]
be independent.  Clearly, the Poisson point process $ \Pi_L^{(\kappa)}$ has the   decomposition 
\[
        \Pi_L^{(\kappa)}
        \stackrel{d}{=}
        \Pi_{L,\mathrm{eff}}^{(\kappa)}
        +
        \Pi_{L,\mathrm{bg}}^{(\kappa)}.
\]

Define the effective and background particle numbers under the tilted
grand-canonical law by
\[
        G_L^{(\kappa)}
        :=
        V_L\int_E x\,
        \Pi_{L,\mathrm{eff}}^{(\kappa)}(du,dx,dm),
\]
and
\[
        B_L^{(\kappa)}
        :=
        V_L\int_E x\,
        \Pi_{L,\mathrm{bg}}^{(\kappa)}(du,dx,dm).
\]
Then
\[
        S_L^{(\kappa)}
        =
        G_L^{(\kappa)}+B_L^{(\kappa)}.
\]
Moreover, \(G_L^{(\kappa)}\) and \(B_L^{(\kappa)}\) are independent
integer-valued random variables.

At the level of cycle counts,
\[
        G_L^{(\kappa)}
        =
        \sum_{j\ge1}
        j\,N_{L,j}^{\mathrm{eff},(\kappa)},
        \qquad
        B_L^{(\kappa)}
        =
        \sum_{j\ge1}
        j\,N_{L,j}^{\mathrm{bg},(\kappa)},
\]
where
\[
        N_{L,j}^{\mathrm{eff},(\kappa)}
        \sim
        \operatorname{Poisson}
        \left(
            \frac{e^{-\kappa j/V_L}}{j}
            q_{L,j}^{\mathrm{eff}}
        \right),
\]
and
\[
        N_{L,j}^{\mathrm{bg},(\kappa)}
        \sim
        \operatorname{Poisson}
        \left(
            \frac{e^{-\kappa j/V_L}}{j}
            q_{L,j}^{\mathrm{bg}}
        \right),
\]
independently over \(j\) and over the two parts.

The  mean and variance values of the total number of particles in the background  are
\[
        \mathbf E_L^{(\kappa)}
        \left[
            \frac{B_L^{(\kappa)}}{V_L}
        \right]
        =
        \frac1{V_L}
        \sum_{j\ge1}
        e^{-\kappa j/V_L}
        q_{L,j}^{\mathrm{bg}}
        =
        m_{L,\mathrm{bg}}^{(\kappa)},
\]
and
\[
        \operatorname{Var}_L^{(\kappa)}
        \left(
            \frac{B_L^{(\kappa)}}{V_L}
        \right)
        =
        \frac1{V_L^2}
        \sum_{j\ge1}
        j\,e^{-\kappa j/V_L}
        q_{L,j}^{\mathrm{bg}}
        =
        v_{L,\mathrm{bg}}^{(\kappa)}.
\]
These are exactly the quantities appearing in~\cref{ass:background-density-concentration}.

Combining the decomposition with~\cref{lem:canonical-conditioning}, we obtain
the part-resolved canonical representation
\begin{equation}\label{eq:can}
        \left(
            \Xi_{L,\mathrm{eff}},
            \Xi_{L,\mathrm{bg}}
        \right)
        \stackrel{d}{=}
        \left(
            \Pi_{L,\mathrm{eff}}^{(\kappa)},
            \Pi_{L,\mathrm{bg}}^{(\kappa)}
        \right)
        \,|\,
        \{G_L^{(\kappa)}+B_L^{(\kappa)}=N_L\}.
\end{equation}

\subsection{Convergence of the unconditioned effective Poisson process}
\label{subsec:proof-mark-unconditioned-convergence}

We next prove the unconditioned convergence of the effective Poisson process $\Pi_{L,\mathrm{eff}}^{(\kappa)}$ under the assumptions in~\cref{sec:limit-objects-main-results}.

For \(h\in\mathcal H\), define the finite-volume effective Laplace exponent
\[
        A_L^{(\kappa)}(h)
        :=
        \int_E
        \left(1-e^{-h(u,x,m)}\right)
        \nu_{L,\mathrm{eff}}^{(\kappa)}(du,dx,dm).
\]
In expanded form,
\[
\begin{aligned}
        A_L^{(\kappa)}(h)
        &=
        \sum_{j\ge1}
        \frac{e^{-\kappa j/V_L}}{j}
        \int_0^1
        \int_{\mathsf M}
        \left(
            1-e^{-h(u,j/V_L,m)}
        \right)
        \mu_{L,j}^{\mathrm{eff}}(dm)\,du .
\end{aligned}
\]
We will show that the limiting exponent is
\[
        A^{(\kappa)}(h)
        :=
        \int_E
        \left(1-e^{-h(u,x,m)}\right)
        \nu^{(\kappa)}(du,dx,dm),
\]
where $\nu^{(\kappa)}$ is defined in~\eqref{defnu}. In expanded form,
\[
        A^{(\kappa)}(h)
        =
        \int_0^\infty
        e^{-\kappa x}\frac{dx}{x}
        \int_0^1
        \int_{\mathsf M}
        \left(1-e^{-h(u,x,m)}\right)
        \eta_x(dm)\,du .
\]
We remark that $\nu^{(\kappa)}$ and $\eta_x$ are well-defined under~\cref{ass:limiting-effective-kernel}.
\begin{proposition}[Unconditioned effective Poisson convergence]
\label{prop:unconditioned-effective-poisson-convergence}
Assume~\cref{ass:limiting-effective-kernel} and
\cref{ass:effective-marked-trace-convergence}.  Then, for every \(\kappa>0\)
and every \(h\in\mathcal H\),
\[
        A_L^{(\kappa)}(h)
        \longrightarrow
        A^{(\kappa)}(h).
        \]
Therefore, 
\[
        \Pi_{L,\mathrm{eff}}^{(\kappa)}
        \Longrightarrow
        \Pi^{(\kappa)}
        \qquad
        \text{in }\mathcal N_\ell(E),
\]
where $  \Pi^{(\kappa)}
        \sim
        \operatorname{PPP}\bigl(\nu^{(\kappa)}\bigr)$.        
\end{proposition}

\begin{proof}
Fix \(h\in\mathcal H\).  There exists \(0<\delta<M<\infty\) such that
\[
        h(u,x,m)=0
        \qquad
        \text{for }x\notin[\delta,M].
\]
Then only indices with \(j/V_L\in[\delta,M]\) contribute to
\(A_L^{(\kappa)}(h)\).

Set
\[
        F_h(u,x,m)
        :=
        1-e^{-h(u,x,m)}.
\]
Then \(F_h\) is bounded and continuous on
\([0,1]\times[\delta,M]\times\mathsf M\).  By~\cref{ass:effective-marked-trace-convergence},
\[
        \sup_{x\in[\delta,M]}
        \Bigg|
        \int_0^1
        \int_{\mathsf M}
        F_h(u,x,m)\,
        \mu_{L,\lfloor xV_L\rfloor}^{\mathrm{eff}}(dm)\,du
        -
        \int_0^1
        \int_{\mathsf M}
        F_h(u,x,m)\,
        \eta_x(dm)\,du
        \Bigg|
        \longrightarrow0 .
\]

For \(x_{L,j}=j/V_L\), we have
\(\lfloor x_{L,j}V_L\rfloor=j\).  Hence
\[
\begin{aligned}
        A_L^{(\kappa)}(h)
        &=
        \sum_{\delta V_L\le j\le M V_L}
        \frac{e^{-\kappa x_{L,j}}}{j}
        \int_0^1
        \int_{\mathsf M}
        F_h(u,x_{L,j},m)
        \mu_{L,j}^{\mathrm{eff}}(dm)\,du
        \\
        &=
        \frac1{V_L}
        \sum_{\delta V_L\le j\le M V_L}
        \frac{e^{-\kappa x_{L,j}}}{x_{L,j}}
        \int_0^1
        \int_{\mathsf M}
        F_h(u,x_{L,j},m)
        \mu_{L,j}^{\mathrm{eff}}(dm)\,du .
\end{aligned}
\]
The preceding uniform convergence and the ordinary Riemann-sum convergence on
\([\delta,M]\) imply
\[
        A_L^{(\kappa)}(h)
        \longrightarrow
        \int_\delta^M
        e^{-\kappa x}\frac{dx}{x}
        \int_0^1
        \int_{\mathsf M}
        F_h(u,x,m)\,\eta_x(dm)\,du .
\]
Since \(h\) vanishes outside \([\delta,M]\), the right-hand side is precisely
\(A^{(\kappa)}(h)\).

Therefore
\[
        \mathbf E_L^{(\kappa)}
        \exp\left\{
            -\langle h,\Pi_{L,\mathrm{eff}}^{(\kappa)}\rangle
        \right\} =
        \exp\{-A_L^{(\kappa)}(h)\}
        \longrightarrow
        \exp\{-A^{(\kappa)}(h)\}.
\]
The right-hand side is the Laplace functional of the Poisson point process
\(\Pi^{(\kappa)}\).
Standard convergence of Poisson point processes on finite-measure windows,
together with the definition of the length-bounded topology, yields
\[
        \Pi_{L,\mathrm{eff}}^{(\kappa)}
        \Longrightarrow
        \Pi^{(\kappa)}
        \qquad
        \text{in }\mathcal N_\ell(E).
\]
\end{proof}

\subsection{Effective local limit theorem}
\label{subsec:proof-effective-local-limit}

 Fix \(\kappa>0\) and \(h\in\mathcal H\).     Define the \(h\)-tilted
finite-volume effective intensity by
\[
        \nu_{L,\mathrm{eff},h}^{(\kappa)}(du,dx,dm)
        :=
        e^{-h(u,x,m)}
        \nu_{L,\mathrm{eff}}^{(\kappa)}(du,dx,dm).
\]
Let $    \Pi_{L,\mathrm{eff},h}^{(\kappa)}
        \sim
        \operatorname{PPP}
        \bigl(
            \nu_{L,\mathrm{eff},h}^{(\kappa)}
        \bigr)$,
and define its particle number by
\[
        G_{L,h}^{(\kappa)}
        :=
        V_L
        \int_E x\,
        \Pi_{L,\mathrm{eff},h}^{(\kappa)}(du,dx,dm).
        \]
        In this subsection, we will prove a local limit theorem for  $G_{L,h}^{(\kappa)}$ (\cref{thm:effective-h-tilted-LLT}). This result is crucial in the proof of the main result.
        
Set
\[
        q_{L,j}^{\mathrm{eff},h}
        :=
        \int_0^1
        \int_{\mathsf M}
        e^{-h(u,j/V_L,m)}
        \mu_{L,j}^{\mathrm{eff}}(dm)\,du .
\]
Then the characteristic function of \(G_{L,h}^{(\kappa)}/V_L\) is
\[
        \varphi_{L,h}^{(\kappa)}(t)
        :=
        \mathbf E
        \exp\left\{
            it\,\frac{G_{L,h}^{(\kappa)}}{V_L}
        \right\}
        =
        \exp\left\{
            \psi_{L,h}^{(\kappa)}(t)
        \right\},
        \qquad t\in\mathbb R,
\]
where
\[
        \psi_{L,h}^{(\kappa)}(t)
        =
        \int_E
        \left(e^{itx}-1\right)
        e^{-h(u,x,m)}
        \nu_{L,\mathrm{eff}}^{(\kappa)}(du,dx,dm).
        \]
By using the notation $q_{L,j}^{\mathrm{eff},h}$, we can write
\[
        \psi_{L,h}^{(\kappa)}(t)
        =
        \sum_{j\ge1}
        \frac{e^{-\kappa j/V_L}}{j}
        q_{L,j}^{\mathrm{eff},h}
        \left(
            e^{itj/V_L}-1
        \right).
\]

We also introduce the corresponding limiting \(h\)-tilted quantities.  Define
\[
        \nu_h^{(\kappa)}(du,dx,dm)
        :=
        e^{-h(u,x,m)}
        \nu^{(\kappa)}(du,dx,dm),
\]
and let
$        \Pi_h^{(\kappa)}
        \sim
        \operatorname{PPP}
        \bigl(
            \nu_h^{(\kappa)}
        \bigr)$.
Its total mass is
\[
        T_h^{(\kappa)}
        :=
        \int_E x\,
        \Pi_h^{(\kappa)}(du,dx,dm).
\]
The characteristic exponent of \(T_h^{(\kappa)}\) is
\[
        \psi_h^{(\kappa)}(t)
        :=
        \int_E
        \left(e^{itx}-1\right)
        e^{-h(u,x,m)}
        \nu^{(\kappa)}(du,dx,dm),
\]
and    $ \varphi_h^{(\kappa)}(t)
        :=
        \exp\left\{
            \psi_h^{(\kappa)}(t)
        \right\}$.
Whenever \(T_h^{(\kappa)}\) has a density, we denote it by $ f_h^{(\kappa)}$.
For \(h=0\), this agrees with the density
\(f_0^{(\kappa)}\) introduced in~\cref{ass:limiting-effective-density}.

\subsubsection{Technical Estimates in Absolute case}
We first establish some useful estimates in the Absolute case (A) under the~\cref{ass:effective-spectral-LLT-criterion}.
The proof is based on Fourier inversion for the
lattice \(V_L^{-1}\mathbb Z\).  The compact part of the Fourier integral is
controlled by the marked trace convergence~\cref{ass:effective-marked-trace-convergence}, while the large-frequency part is
controlled by the spectral criterion in~\cref{ass:effective-spectral-LLT-criterion}. See Petrov~\cite{Petrov1975} 
for the standard lattice local limit
estimates for sums of independent integer-valued random variables.

\begin{lemma}[Compact convergence of Fourier exponents]
\label{lem:compact-fourier-convergence}
Assume~\cref{ass:limiting-effective-kernel} and~\cref{ass:effective-marked-trace-convergence}.  Then, for every
\(A<\infty\),
\[
        \sup_{|t|\le A}
        \left|
            \psi_{L,h}^{(\kappa)}(t)
            -
            \psi_h^{(\kappa)}(t)
        \right|
        \longrightarrow0.
\]
Consequently,
\[
        \sup_{|t|\le A}
        \left|
            \varphi_{L,h}^{(\kappa)}(t)
            -
            \varphi_h^{(\kappa)}(t)
        \right|
        \longrightarrow0.
\]
\end{lemma}

\begin{proof}
Fix \(A<\infty\) and \(h\in\mathcal H\).  Choose \(0<\delta<M<\infty\) such that \(h(u,x,m)=0\) for \(x\notin[\delta,M]\).
Define \(g_t(u,x,m)=(e^{itx}-1)e^{-h(u,x,m)}\); then \(g_t\) is bounded and, on \([0,1]\times[\delta,M]\times\mathsf M\), the family \(\{g_t:|t|\le A\}\) is bounded and equicontinuous.
We split \(\psi_{L,h}^{(\kappa)}(t)=\sum_{j\ge1}\frac{e^{-\kappa j/V_L}}{j}q_{L,j}^{\mathrm{eff},h}(e^{itj/V_L}-1)\) into three parts:
\[
\psi_{L,h}^{(\kappa)}(t)=\underbrace{\sum_{j< \delta V_L}}_{(\mathrm{I})}
+\underbrace{\sum_{\delta V_L\le j\le M V_L}}_{(\mathrm{II})}
+\underbrace{\sum_{j> M V_L}}_{(\mathrm{III})}.
\]
For \(j<\delta V_L\), \(|e^{itj/V_L}-1|\le |t|j/V_L\le A j/V_L\) and \(q_{L,j}^{\mathrm{eff},h}\le q_{L,j}^{\mathrm{eff}}\le C_{\mathrm{eff}}\). Hence
\[
|(\mathrm{I})|\le A C_{\mathrm{eff}}\frac1{V_L}\sum_{j<\delta V_L}1 \le A C_{\mathrm{eff}}\delta .
\]
For \(j> M V_L\), \(|e^{itj/V_L}-1|\le2\), so
\[
|(\mathrm{III})|\le 2C_{\mathrm{eff}}\sum_{j> M V_L}\frac{e^{-\kappa j/V_L}}{j}
\le \mathrm{const}\cdot\int_{M}^{\infty} e^{-\kappa x}\frac{dx}{x}.
\]
The same bounds hold for the corresponding parts of \(\psi_h^{(\kappa)}(t)\) using~\cref{ass:limiting-effective-kernel}.
Thus both tails can be made arbitrarily small, uniformly in \(L\) and \(|t|\le A\), by choosing \(\delta\) small and \(M\) large.

Rewrite (II) as a Riemann sum:
\[
(\mathrm{II})=\frac1{V_L}\sum_{\delta V_L\le j\le M V_L}\frac{e^{-\kappa x_{L,j}}}{x_{L,j}}
\int_0^1\!\int_{\mathsf M} g_t(u,x_{L,j},m)\,\mu_{L,j}^{\mathrm{eff}}(dm)\,du ,
\quad x_{L,j}=j/V_L .
\]
By the equicontinuity of \(\{g_t\}_{|t|\le A}\) and the marked trace convergence
\cref{ass:effective-marked-trace-convergence}, we have, uniformly for \(|t|\le A\),
\[
\sup_{x\in[\delta,M]}\Bigl|\int_0^1\!\int_{\mathsf M} g_t(u,x,m)\mu_{L,\lfloor xV_L\rfloor}^{\mathrm{eff}}(dm)du
-\int_0^1\!\int_{\mathsf M} g_t(u,x,m)\eta_x(dm)du\Bigr|\to0 .
\]
Moreover, \(x\mapsto e^{-\kappa x}/x\) is continuous on \([\delta,M]\). Hence, by standard Riemann-sum convergence,
\[
(\mathrm{II})\longrightarrow \int_{\delta}^{M} e^{-\kappa x}\frac{dx}{x}\int_0^1\!\int_{\mathsf M} g_t(u,x,m)\eta_x(dm)du
\]
uniformly for \(|t|\le A\).
Combining the tail estimates with the convergence of (II) proves the uniform convergence of \(\psi_{L,h}^{(\kappa)}\) to \(\psi_h^{(\kappa)}\) on \([-A,A]\). Since the exponents are locally bounded, exponentiation yields the same for the characteristic functions.
\end{proof}


\begin{lemma}[Fourier tail control under the absolute condition]
\label{lem:fourier-tail-control-A}
Assume~\cref{ass:limiting-effective-kernel},
\cref{ass:effective-marked-trace-convergence}, and assume that
condition \(\mathrm{(A)}\) in
\cref{ass:effective-spectral-LLT-criterion} holds.  Then, for every
\(\kappa>0\) and every \(h\in\mathcal H\),
\[
        \lim_{A\to\infty}
        \limsup_{L\to\infty}
        \int_{A<|t|\le \pi V_L}
        \left|
            \varphi_{L,h}^{(\kappa)}(t)
        \right|\,dt
        =
        0.
\]
Moreover, the limiting characteristic function
\(\varphi_h^{(\kappa)}\) belongs to \(L^1(\mathbb R)\). Hence
\(T_h^{(\kappa)}\) has a bounded continuous density
\(f_h^{(\kappa)}\), given by
\[
        f_h^{(\kappa)}(a)
        =
        \frac1{2\pi}
        \int_{\mathbb R}
        e^{-ita}
        \varphi_h^{(\kappa)}(t)\,dt .
\]
\end{lemma}
\begin{proof}
Since
\[
        \Re\psi_{L,h}^{(\kappa)}(t)
        =
        -\sum_{j\ge1}
        \frac{e^{-\kappa j/V_L}}{j}
        q_{L,j}^{\mathrm{eff},h}
        \left(
            1-\cos(tj/V_L)
        \right),
\]
we have
\[
        \left|
        \varphi_{L,h}^{(\kappa)}(t)
        \right|
        =
        \exp\left\{
        -
        \sum_{j\ge1}
        \frac{e^{-\kappa j/V_L}}{j}
        q_{L,j}^{\mathrm{eff},h}
        \left(
            1-\cos(tj/V_L)
        \right)
        \right\}.
\]

We first compare the \(h\)-tilted characteristic function with the untilted
one.  Since \(h\in\mathcal H\), there exist \(0<\delta<M<\infty\) such that
\[
        h(u,x,m)=0
        \qquad
        \text{for }x\notin[\delta,M].
\]
Consequently,
\[
        q_{L,j}^{\mathrm{eff},h}
        =
        q_{L,j}^{\mathrm{eff}}
        \qquad
        \text{whenever }j/V_L\notin[\delta,M].
\]
Moreover \(h\) is bounded, and \(q_{L,j}^{\mathrm{eff}}\) is uniformly
bounded by $C_{\mathrm{eff}}$. Hence
\[
\begin{aligned}
        \log
        \frac{
            |\varphi_{L,h}^{(\kappa)}(t)|
        }{
            |\varphi_{L,0}^{(\kappa)}(t)|
        }
        &=
        -\sum_{j\ge1}
        \frac{e^{-\kappa j/V_L}}{j}
        \left(
            q_{L,j}^{\mathrm{eff},h}
            -
            q_{L,j}^{\mathrm{eff}}
        \right)
        \left(
            1-\cos(tj/V_L)
        \right)
        \\
        &\le
        2
        \sum_{\delta V_L\le j\le MV_L}
        \frac{e^{-\kappa j/V_L}}{j}
        \left|
            q_{L,j}^{\mathrm{eff},h}
            -
            q_{L,j}^{\mathrm{eff}}
        \right|
        \le 2C_{\mathrm{eff}}\log(M/\delta).
\end{aligned}
\]
Thus, uniformly for \(|t|\le \pi V_L\),
\[
        \left|
        \varphi_{L,h}^{(\kappa)}(t)
        \right|
        \le
        C_{\kappa,h}
        \left|
        \varphi_{L,0}^{(\kappa)}(t)
        \right|.
\]
It remains to prove an integrable polynomial bound for
\(\varphi_{L,0}^{(\kappa)}\).

Under condition \(\mathrm{(A)}\), we have
\[
        q_{L,j}^{\mathrm{eff}}
        =
        \int_{[0,\infty)}
        e^{-\beta(j/V_L)\lambda}
        \Sigma_L(d\lambda).
\]
Therefore
\[
\begin{aligned}
        -\log
        \left|
        \varphi_{L,0}^{(\kappa)}(t)
        \right|
        &=
        \sum_{j\ge1}
        \frac{e^{-\kappa j/V_L}}{j}
        q_{L,j}^{\mathrm{eff}}
        \left(
            1-\cos(tj/V_L)
        \right)
        \\
        &=
        \int_{[0,\infty)}
        \sum_{j\ge1}
        \frac{
            e^{-(\kappa+\beta\lambda)j/V_L}
        }{j}
        \left(
            1-\cos(tj/V_L)
        \right)
        \Sigma_L(d\lambda).
\end{aligned}
\]
For \(\alpha>0\), set
\[
        S_{L,\alpha}(t)
        :=
        \sum_{j\ge1}
        \frac{e^{-\alpha j/V_L}}{j}
        \left(
            1-\cos(tj/V_L)
        \right).
\]
We use the following elementary estimate: there exists a universal constant
\(C<\infty\) such that, for all \(L\), all \(\alpha>0\), and all
\(|t|\le \pi V_L\),
\[
        S_{L,\alpha}(t)
        \ge
        \log(1+|t|)
        -
        C\log(1+\alpha)
        -
        C .
\]
Indeed,  let
\[
        r=e^{-\alpha/V_L},
        \qquad
        \theta=t/V_L,
\]
then by Taylor expansion,
\[
\begin{aligned}
        S_{L,\alpha}(t)
        &=
        \sum_{j\ge1}
        \frac{r^j}{j}
        \left(
            1-\cos(j\theta)
        \right)
        \\
        &=
        -\log(1-r)
        +
        \frac12
        \log
        \left(
            1-2r\cos\theta+r^2
        \right)
        \\
        &=
        \frac12
        \log
        \left(
            1+
            \frac{
                2r(1-\cos\theta)
            }{
                (1-r)^2
            }
        \right).
\end{aligned}
\]
If \(\alpha\le V_L\), then \(1-r\le \alpha/V_L\), \(r\ge e^{-1}\), and
\(1-\cos\theta\ge 2(\theta/\pi)^2\) for \(|\theta|\le\pi\). Hence, there exists positive constants $c$ and $C$ independent of $\alpha$, $L$ and $t$, such that
\[
        S_{L,\alpha}(t)
        \ge
        \frac12
        \log
        \left(
            1+c\frac{t^2}{\alpha^2}
        \right)
        \ge
        \log(1+|t|)
        -
        C\log(1+\alpha)
        -
        C .
\]
If \(\alpha>V_L\), then \(|t|\le \pi V_L\) implies
\[
        1+|t|
        \le
        1+\pi V_L
        \le
        1+\pi\alpha
        \le
        (1+\pi)(1+\alpha).
\]
Hence
\[
        \log(1+|t|)
        \le
        \log(1+\alpha)+\log(1+\pi).
\]
Thus, we can choose $C>1$, then
\[
        \log(1+|t|)
        -
        C\log(1+\alpha)
        -
        C
        \le 0 .
\]
Since \(S_{L,\alpha}(t)\ge0\), we obtain
\[
        S_{L,\alpha}(t)
        \ge
        \log(1+|t|)
        -
        C\log(1+\alpha)
        -
        C .
\]
 This proves the claim.

Applying this estimate with
\[
        \alpha=\kappa+\beta\lambda
\]
gives
\[
\begin{aligned}
        -\log
        \left|
        \varphi_{L,0}^{(\kappa)}(t)
        \right|
        &\ge
        \int_{[0,\infty)}
        \left[
            \log(1+|t|)
            -
            C\log(1+\kappa+\beta\lambda)
            -
            C
        \right]
        \Sigma_L(d\lambda)
        \\
        &=
        \Theta_L\log(1+|t|)
        -
        C
        \int_{[0,\infty)}
        \log(1+\kappa+\beta\lambda)\,
        \Sigma_L(d\lambda)
        -
        C\Theta_L ,
\end{aligned}
\]
where
\[
        \Theta_L:=\Sigma_L([0,\infty)).
\]
By condition \(\mathrm{(A)}\), for all sufficiently large \(L\),
\[
        \Theta_L\ge \Theta_*>1,
\]
and the logarithmic moment condition gives
\[
        \sup_L
        \int_{[0,\infty)}
        \log(1+\kappa+\beta\lambda)\,
        \Sigma_L(d\lambda)
        <\infty .
\]
Since \(\kappa>0\), this logarithmic moment bound also controls
\(\Theta_L\), because
\[
        \log(1+\kappa+\beta\lambda)
        \ge
        \log(1+\kappa)>0.
\]
Thus there exists \(C_\kappa<\infty\) such that, for all sufficiently large
\(L\) and all \(|t|\le\pi V_L\),
\[
        -\log
        \left|
        \varphi_{L,0}^{(\kappa)}(t)
        \right|
        \ge
        \Theta_*
        \log(1+|t|)
        -
        C_\kappa .
\]
Combining this with the comparison estimate between the tilted and untilted
characteristic functions yields
\begin{equation}\label{tailbound}
        \left|
        \varphi_{L,h}^{(\kappa)}(t)
        \right|
        \le
        C_{\kappa,h}
        (1+|t|)^{-\Theta_*},
        \qquad |t|\le\pi V_L ,
\end{equation}
for all sufficiently large \(L\).  Since \(\Theta_*>1\), the function
\((1+|t|)^{-\Theta_*}\) is integrable on \(\mathbb R\). Therefore
\[
\begin{aligned}
        \lim_{A\to\infty}
        \limsup_{L\to\infty}
        \int_{A<|t|\le\pi V_L}
        \left|
            \varphi_{L,h}^{(\kappa)}(t)
        \right|\,dt
        &\le
        C_{\kappa,h}
        \lim_{A\to\infty}
        \int_{|t|>A}
        (1+|t|)^{-\Theta_*}\,dt
        \\
        &=0 .
\end{aligned}
\]
This proves the Fourier tail estimate.

It remains to prove that the limiting characteristic function belongs to
\(L^1(\mathbb R)\).  Fix \(t\in\mathbb R\).  For all sufficiently large \(L\),
we have \(|t|\le\pi V_L\), and therefore
\[
        \left|
        \varphi_{L,h}^{(\kappa)}(t)
        \right|
        \le
        C_{\kappa,h}
        (1+|t|)^{-\Theta_*}.
\]
By \cref{lem:compact-fourier-convergence},
\[
        \varphi_{L,h}^{(\kappa)}(t)
        \longrightarrow
        \varphi_h^{(\kappa)}(t)
\]
locally uniformly, and hence pointwise.  Passing to the limit gives
\[
        \left|
        \varphi_h^{(\kappa)}(t)
        \right|
        \le
        C_{\kappa,h}
        (1+|t|)^{-\Theta_*}.
\]
Since \(\Theta_*>1\), this bound is integrable. Thus
\[
        \varphi_h^{(\kappa)}\in L^1(\mathbb R).
\]
By the Fourier inversion theorem, \(T_h^{(\kappa)}\) has a bounded continuous
density given by
\[
        f_h^{(\kappa)}(a)
        =
        \frac1{2\pi}
        \int_{\mathbb R}
        e^{-ita}
        \varphi_h^{(\kappa)}(t)\,dt .
\]
This completes the proof.
\end{proof}

\subsubsection{Technical Estimates in Critical finite-type case}
Now we prove two useful results in the critical finite-type case.

\begin{lemma}[Untilted Finite-type estimate]
\label{lem:finite-type-core-estimate}
Assume condition \(\mathrm{(B)}\) in
\cref{ass:effective-spectral-LLT-criterion}.  Set
\[
        \alpha_{L,r}:=\kappa+\beta\lambda_{L,r},
        \qquad
        \alpha_r:=\kappa+\beta\lambda_r .
\]
Then, for \(h=0\),
\[
        G_{L,0}^{(\kappa)}
        \stackrel{d}{=}
        \sum_{r=1}^R Y_{L,r},
\]
where the \(Y_{L,r}\)'s are independent negative-binomial random variables. 
Moreover,
\[
        \sup_L\sup_{n\ge0}
        V_L
        \mathbf P
        \left(
            G_{L,0}^{(\kappa)}=n
        \right)
        <\infty ,
\]
and, uniformly for \(n/V_L\) in compact subsets of \((0,\infty)\),
\[
        V_L
        \mathbf P
        \left(
            G_{L,0}^{(\kappa)}=n
        \right)
  -
        f_0^{(\kappa)}
        \left(
            \frac n{V_L}
        \right)\to 0,
\]
where
$       f_0^{(\kappa)}
        =
        g_1*\cdots*g_R$ and $g_r$ is the Gamma density
\[
        g_r(x)
        =
        \frac{\alpha_r^{\theta_r}}{\Gamma(\theta_r)}
        x^{\theta_r-1}e^{-\alpha_r x}
        \mathbf 1_{\{x>0\}} .
\]
In particular, since
$       \sum_{r=1}^R\theta_r=1$,
the density \(f_0^{(\kappa)}\) is locally bounded on \([0,\infty)\) and
continuous on \((0,\infty)\).
\end{lemma}

\begin{proof}
For \(h=0\), the probability generating function is
\[
        \mathbf E z^{G_{L,0}^{(\kappa)}}
        =
        \exp\left\{
        \sum_{j\ge1}
        \frac{e^{-\kappa j/V_L}}{j}
        q_{L,j}^{\mathrm{eff}}
        (z^j-1)
        \right\}                                                 
        =
        \prod_{r=1}^R
        \exp\left\{
        \theta_r
        \sum_{j\ge1}
        \frac{e^{-\alpha_{L,r}j/V_L}}{j}
        (z^j-1)
        \right\}.
\]
Since
\[
        \sum_{j\ge1}\frac{a^jz^j}{j}
        =
        -\log(1-az),
\]
we obtain
\[
        \mathbf E z^{G_{L,0}^{(\kappa)}}
        =
        \prod_{r=1}^R
        \left(
            \frac{1-a_{L,r}}{1-a_{L,r}z}
        \right)^{\theta_r},
        \qquad
        a_{L,r}:=e^{-\alpha_{L,r}/V_L}.
\]
This proves the negative-binomial decomposition. Moreover,
\[
        \mathbf P(Y_{L,r}=n)
        =
        \frac{\Gamma(n+\theta_r)}
             {\Gamma(\theta_r)\Gamma(n+1)}
        (1-a_{L,r})^{\theta_r}
        a_{L,r}^{\,n}.
\]

By Stirling's formula, uniformly for \(n/V_L\) in compact subsets of
\((0,\infty)\),
\[
        V_L
        \mathbf P(Y_{L,r}=n)
    -
        g_r\left(\frac n{V_L}\right)\to 0.
\]
The same Stirling estimate, with \(n=0\) treated separately, gives
\[
        \sup_L\sup_{n\ge0}
        V_L^{\theta_r}
        \mathbf P(Y_{L,r}=n)
        (1+n)^{1-\theta_r}
        <\infty .
\]
Using repeatedly the elementary discrete beta-convolution estimate
\[
        \sum_{k=0}^{n}
        (1+k)^{\alpha-1}
        (1+n-k)^{\beta-1}
        \le
        C_{\alpha,\beta}
        (1+n)^{\alpha+\beta-1},
        \qquad \alpha,\beta>0,
\]
we obtain
\[
        \sup_L\sup_{n\ge0}
        V_L
        \mathbf P(G_{L,0}^{(\kappa)}=n)
        <\infty ,
\]
because \(\sum_r\theta_r=1\).
The same finite convolution argument, combined with the one-dimensional
local limit convergence of each \(Y_{L,r}\), yields
\[
        V_L
        \mathbf P(G_{L,0}^{(\kappa)}=n)
-
        f_0^{(\kappa)}(n/V_L)\to 0
\]
uniformly for \(n/V_L\) in compact subsets of \((0,\infty)\).
The regularity of \(f_0^{(\kappa)}\) follows from the fact that it is a finite
convolution of Gamma densities and that the total shape parameter equals
\(1\).
\end{proof}

\begin{lemma}[Compact thinning and deconvolution transfer]
\label{lem:compact-thinning-deconvolution-transfer}
Assume condition \(\mathrm{(B)}\) in
\cref{ass:effective-spectral-LLT-criterion}, and assume the marked trace
convergence in \cref{ass:effective-marked-trace-convergence}.  Let
\(h\in\mathcal H\), with \(h\ge0\).  Then there exists a finite signed
compound-exponential measure \(\rho^{(h)}\) on \([0,\infty)\) such that
\[
        f_h^{(\kappa)}(a)
        =
        \int_{[0,\infty)}
        f_0^{(\kappa)}(a-y)\,\rho^{(h)}(dy),
        \qquad a>0,
\]
where \(f_0^{(\kappa)}\) is extended by zero to \((-\infty,0)\).  Moreover,
for every compact interval \(K\subset(0,\infty)\),
\[
        \sup_{\substack{n\in\mathbb N:\\ n/V_L\in K}}
        \left|
            V_L\mathbf P(G_{L,h}^{(\kappa)}=n)
            -
            f_h^{(\kappa)}
            \left(\frac n{V_L}\right)
        \right|
        \longrightarrow0 .
\]
\end{lemma}

\begin{proof}
Choose \(0<\delta<M<\infty\) such that
$       h(u,x,m)=0$
for $x\notin[\delta,M]$.
Set
\[
        q_h(x)
        :=
        \int_0^1
        \int_{\mathsf M}
        e^{-h(u,x,m)}\,\eta_x(dm)\,du .
\]
Then \(q_h(x)=\phi(x)\) for \(x\notin[\delta,M]\), and, since
\(h\ge0\),
$       0\le q_h(x)\le \phi(x)$.
By the marked trace convergence,
\[
        q_{L,\lfloor xV_L\rfloor}^{\mathrm{eff}}
        \longrightarrow \phi(x),
        \qquad
        q_{L,\lfloor xV_L\rfloor}^{\mathrm{eff},h}
        \longrightarrow q_h(x),
\]
uniformly for \(x\in[\delta,M]\).

Since \(h\ge0\), the \(h\)-tilted effective intensity is a sub-intensity
of the untilted effective intensity. Hence the Poisson thinning
construction gives a coupling under which
\[
        G_{L,0}^{(\kappa)}
        =
        G_{L,h}^{(\kappa)}
        +
        \Delta_L^{(h)},
\]
where \(G_{L,h}^{(\kappa)}\) is independent of the deleted mass
\(\Delta_L^{(h)}\). Define
\[
        d_{L,j}^{(h)}
        :=
        \frac{e^{-\kappa j/V_L}}{j}
        \left(
            q_{L,j}^{\mathrm{eff}}
            -
            q_{L,j}^{\mathrm{eff},h}
        \right).
\]
Then \(d_{L,j}^{(h)}\ge0\), and \(d_{L,j}^{(h)}=0\) unless
\(j/V_L\in[\delta,M]\). The probability generating function of the
deleted mass is therefore
\[
        \mathbf E z^{\Delta_L^{(h)}}
        =
        \exp
        \left\{
            \sum_{j\ge1}
            d_{L,j}^{(h)}(z^j-1)
        \right\}.
\]

The preceding decomposition gives
\[
        F_{L,0}(z)
        =
        F_{L,h}(z)F_{L,\Delta}(z),
\]
and hence
\[
        F_{L,h}(z)
        =
        F_{L,0}(z)F_{L,\Delta}(z)^{-1}.
\]
Define coefficients \(\rho_L^{(h)}(k)\) by
\[
        \sum_{k\ge0}\rho_L^{(h)}(k)z^k
        =
        \exp
        \left\{
            -
            \sum_{j\ge1}
            d_{L,j}^{(h)}(z^j-1)
        \right\}.
\]
Comparing coefficients yields the exact identity
\[
        \mathbf P(G_{L,h}^{(\kappa)}=n)
        =
        \sum_{k\ge0}
        \mathbf P(G_{L,0}^{(\kappa)}=n-k)
        \rho_L^{(h)}(k),
\]
where the summand is understood to be zero if \(n-k<0\).

Next set
\[
        \lambda_L^{(h)}
        :=
        \sum_{j\ge1}
        d_{L,j}^{(h)}\delta_{j/V_L}.
\]
For every continuous function \(g\) on \([0,\infty)\),
\[
\begin{aligned}
        \int g\,d\lambda_L^{(h)}
        &=
        \sum_{\delta V_L\le j\le M V_L}
        g\left(\frac j{V_L}\right)
        \frac{e^{-\kappa j/V_L}}{j}
        \left(
            q_{L,j}^{\mathrm{eff}}
            -
            q_{L,j}^{\mathrm{eff},h}
        \right)                                                     \\
        &=
        \frac1{V_L}
        \sum_{\delta V_L\le j\le M V_L}
        g\left(\frac j{V_L}\right)
        e^{-\kappa j/V_L}
        \left(
            q_{L,j}^{\mathrm{eff}}
            -
            q_{L,j}^{\mathrm{eff},h}
        \right)
        \frac{1}{j/V_L}.
\end{aligned}
\]
By the uniform convergence on \([\delta,M]\), this Riemann sum converges to
\[
        \int_\delta^M
        g(x)e^{-\kappa x}
        \bigl(\phi(x)-q_h(x)\bigr)\frac{dx}{x}.
\]
Thus
\[
        \lambda_L^{(h)}
        \Rightarrow
        \lambda^{(h)},
        \qquad
        \lambda^{(h)}(dx)
        =
        e^{-\kappa x}
        \bigl(\phi(x)-q_h(x)\bigr)
        \frac{dx}{x}\mathbf 1_{[\delta,M]}(x).
\]
In particular, \(\lambda^{(h)}\) is finite, absolutely continuous and
compactly supported.

Let \(\rho^{(h)}\) be the finite signed measure on \([0,\infty)\) defined
by
\[
        \int_{[0,\infty)}e^{-sy}\rho^{(h)}(dy)
        =
        \exp
        \left\{
            -
            \int_{[0,\infty)}
            (e^{-sx}-1)\lambda^{(h)}(dx)
        \right\}.
\]
Equivalently,
\[
        \int_{[0,\infty)}e^{-sy}\rho^{(h)}(dy)
        =
        \exp
        \left\{
            \int_\delta^M
            e^{-\kappa x}
            \bigl(q_h(x)-\phi(x)\bigr)
            (e^{-sx}-1)
            \frac{dx}{x}
        \right\}.
\]
Since \(\lambda_L^{(h)}\Rightarrow\lambda^{(h)}\) and
\(\lambda_L^{(h)}([0,\infty))\to\lambda^{(h)}([0,\infty))\), the
exponential-series representation of compound-exponential measures gives
\[
        \sum_{k\ge0}\rho_L^{(h)}(k)\delta_{k/V_L}
        \Rightarrow
        \rho^{(h)}
\]
weakly as finite signed measures.

We prove the corresponding total-variation control. Since
\(d_{L,j}^{(h)}\) is supported on \(j/V_L\in[\delta,M]\), and the
quantities
\(q_{L,j}^{\mathrm{eff}}-q_{L,j}^{\mathrm{eff},h}\) are uniformly bounded
on this range for all large \(L\), one has
\[
        \sup_L\sum_{j\ge1}d_{L,j}^{(h)}<\infty.
\]
Moreover,
\[
        \sum_{k\ge0}\rho_L^{(h)}(k)\delta_{k/V_L}
        =
        \exp\{\lambda_L^{(h)}([0,\infty))\}
        \sum_{\ell=0}^{\infty}
        \frac{(-1)^\ell}{\ell!}
        \bigl(\lambda_L^{(h)}\bigr)^{*\ell}.
\]
Hence its total variation is dominated by the positive measure
\[
        \exp\{\lambda_L^{(h)}([0,\infty))\}
        \sum_{\ell=0}^{\infty}
        \frac{1}{\ell!}
        \bigl(\lambda_L^{(h)}\bigr)^{*\ell}.
\]
These dominating positive measures have uniformly bounded total mass and
converge weakly to
\[
        \exp\{\lambda^{(h)}([0,\infty))\}
        \sum_{\ell=0}^{\infty}
        \frac{1}{\ell!}
        \bigl(\lambda^{(h)}\bigr)^{*\ell}.
\]
The limiting measure has no atoms in \((0,\infty)\), because
\(\lambda^{(h)}\) is absolutely continuous. Therefore, for every compact
\(K\subset(0,\infty)\),
\[
        \lim_{\varepsilon\downarrow0}
        \limsup_{L\to\infty}
        \sup_{a\in K}
        \left|
        \sum_{\substack{k\ge0:\\ |a-k/V_L|\le\varepsilon}}
        \rho_L^{(h)}(k)
        \right|
        =
        0.
\]
Indeed, the absolute value is bounded by the corresponding mass of the
dominating positive measure, and the last assertion follows from weak
convergence and the non-atomicity of the limiting measure on \(K\).

Now fix \(a_L=n/V_L\in K\). By the coefficient identity,
\[
        V_L\mathbf P(G_{L,h}^{(\kappa)}=n)
        =
        \sum_{k\ge0}
        V_L\mathbf P(G_{L,0}^{(\kappa)}=n-k)
        \rho_L^{(h)}(k).
\]
By Lemma 6.6, under condition \(\mathrm{(B)}\),
\[
        \sup_L\sup_{m\ge0}
        V_L\mathbf P(G_{L,0}^{(\kappa)}=m)<\infty,
\]
and
\[
        V_L\mathbf P(G_{L,0}^{(\kappa)}=m)
        =
        f_0^{(\kappa)}(m/V_L)+o(1)
\]
uniformly when \(m/V_L\) ranges in compact subsets of \((0,\infty)\).
In particular, \(f_0^{(\kappa)}\) is bounded on \((0,\infty)\), after
extension by zero to \((-\infty,0)\).

Let \(\varepsilon>0\) be small enough that
\(\varepsilon<\inf K/2\). On the set
$       a_L-k/V_L\ge\varepsilon$,
the variable \((n-k)/V_L\) ranges in a compact subset of
\((0,\infty)\), so the preceding local limit theorem applies uniformly.
On the set $a_L-\frac{k}{V_L}\le-\varepsilon$,
we have \(k>n\) for all large \(L\), and hence both
\[
        \mathbf P(G_{L,0}^{(\kappa)}=n-k)
        \quad\text{and}\quad
        f_0^{(\kappa)}
        \left(a_L-\frac{k}{V_L}\right)
\]
are zero. The remaining boundary region
$       \left|a_L-k/V_L\right|<\varepsilon$
is negligible uniformly in \(a_L\in K\), by the total-variation control
above and the uniform boundedness just stated. Therefore,
\[
        V_L\mathbf P(G_{L,h}^{(\kappa)}=n)
        =
        \sum_{k\ge0}
        f_0^{(\kappa)}
        \left(a_L-\frac{k}{V_L}\right)
        \rho_L^{(h)}(k)
        +o(1),
\]
uniformly for \(a_L\in K\).

It remains to pass from the last discrete signed convolution to its
limit. Away from the boundary set \(y=a_L\), the functions
\(y\mapsto f_0^{(\kappa)}(a_L-y)\) are uniformly continuous for
\(a_L\in K\), because \(f_0^{(\kappa)}\) is continuous on compact
subsets of \((0,\infty)\). The same boundary estimate as above removes
the region \(|a_L-y|<\varepsilon\), and the weak convergence of the signed
measures gives
\[
        \sum_{k\ge0}
        f_0^{(\kappa)}
        \left(a_L-\frac{k}{V_L}\right)
        \rho_L^{(h)}(k)
        =
        \int_{[0,\infty)}
        f_0^{(\kappa)}(a_L-y)\rho^{(h)}(dy)
        +o(1),
\]
uniformly for \(a_L\in K\). Consequently,
\[
        V_L\mathbf P(G_{L,h}^{(\kappa)}=n)
        =
        \int_{[0,\infty)}
        f_0^{(\kappa)}
        \left(\frac n{V_L}-y\right)
        \rho^{(h)}(dy)
        +o(1),
\]
uniformly for \(n/V_L\in K\).

Finally we identify the limiting density. Define, for \(a>0\),
\[
        \widetilde f_h^{(\kappa)}(a)
        :=
        \int_{[0,\infty)}
        f_0^{(\kappa)}(a-y)\rho^{(h)}(dy),
\]
where \(f_0^{(\kappa)}\) is extended by zero to \((-\infty,0)\). Taking
Laplace transforms and using the definition of \(\rho^{(h)}\), we obtain
\[
\begin{aligned}
        \int_0^\infty e^{-sa}\widetilde f_h^{(\kappa)}(a)\,da
        &=
        \exp
        \left\{
            \int_0^\infty
            e^{-\kappa x}
            \phi(x)
            (e^{-sx}-1)
            \frac{dx}{x}
        \right\}                                                    \\
        &\quad\times
        \exp
        \left\{
            \int_\delta^M
            e^{-\kappa x}
            \bigl(q_h(x)-\phi(x)\bigr)
            (e^{-sx}-1)
            \frac{dx}{x}
        \right\}.
\end{aligned}
\]
Since \(q_h(x)=\phi(x)\) for \(x\notin[\delta,M]\), the right-hand side is
\[
        \exp
        \left\{
            \int_0^\infty
            e^{-\kappa x}
            q_h(x)
            (e^{-sx}-1)
            \frac{dx}{x}
        \right\}.
\]
This is precisely the Laplace transform of the limiting total mass
\(T_h^{(\kappa)}\). Hence
$       \widetilde f_h^{(\kappa)}=f_h^{(\kappa)}$.
Therefore
\[
        V_L\mathbf P(G_{L,h}^{(\kappa)}=n)
        =
        f_h^{(\kappa)}
        \left(\frac n{V_L}\right)
        +o(1),
\]
uniformly for \(n/V_L\in K\). The convolution representation of
\(f_h^{(\kappa)}\) follows from the definition of
\(\widetilde f_h^{(\kappa)}\). This proves the lemma.
\end{proof}

\subsubsection{Effective local limit theorems}
We now gather all the useful results concerning the effective part, which will be invoked in the proof of the main theorems.
\begin{lemma}[Effective lattice bound]
\label{lem:effective-lattice-bound}
Assume~\cref{ass:limiting-effective-kernel},
\cref{ass:effective-marked-trace-convergence}, and
\cref{ass:effective-spectral-LLT-criterion}. Then, for every
\(\kappa>0\) and every \(h\in\mathcal H\),
\[
        \limsup_{L\to\infty}
        V_L
        \sup_{n\ge0}
        \mathbf P
        \left(
            G_{L,h}^{(\kappa)}=n
        \right)
        <\infty .
\]
Consequently,
\[
        \sup_L\sup_{n\ge0}
        V_L
        \mathbf E_L^{(\kappa)}
        \left[
            e^{-\langle h,\Pi_{L,\mathrm{eff}}^{(\kappa)}\rangle}
            \mathbf 1_{\{G_L^{(\kappa)}=n\}}
        \right]
        <\infty .
\]
\end{lemma}

\begin{proof}
We distinguish the two conditions in
\cref{ass:effective-spectral-LLT-criterion}.

If condition \(\mathrm{(A)}\) holds, Fourier inversion on the lattice
\(V_L^{-1}\mathbb Z\) gives
\[
        V_L
        \mathbf P(G_{L,h}^{(\kappa)}=n)
        =
        \frac1{2\pi}
        \int_{-\pi V_L}^{\pi V_L}
        e^{-itn/V_L}
        \varphi_{L,h}^{(\kappa)}(t)\,dt .
\]
Hence, using the absolute Fourier-tail bound~\eqref{tailbound},
\[
\begin{aligned}
        V_L
        \mathbf P(G_{L,h}^{(\kappa)}=n)
        &\le
        \frac1{2\pi}
        \int_{-\pi V_L}^{\pi V_L}
        \left|\varphi_{L,h}^{(\kappa)}(t)\right|\,dt          \\
        &\le
        \frac{C_{\kappa,h}}{2\pi}
        \int_{\mathbb R}(1+|t|)^{-\Theta_*}\,dt
        \le C_{\kappa,h}.
\end{aligned}
\]
Thus
\[
        \limsup_{L\to\infty}
        V_L
        \sup_{n\ge0}
        \mathbf P(G_{L,h}^{(\kappa)}=n)
        <\infty .
\]

If condition \(\mathrm{(B)}\) holds, then by the Girsanov formula,
\[
        \mathbf P(G_{L,h}^{(\kappa)}=n)
        \le
        e^{A_L^{(\kappa)}(h)}
        \mathbf P(G_{L,0}^{(\kappa)}=n).
\]
By \cref{prop:unconditioned-effective-poisson-convergence},
\[
        \sup_L A_L^{(\kappa)}(h)<\infty ,
\]
and by \cref{lem:finite-type-core-estimate},
\[
        \sup_L\sup_{n\ge0}
        V_L
        \mathbf P(G_{L,0}^{(\kappa)}=n)
        <\infty .
\]
Therefore
\[
        \sup_L\sup_{n\ge0}
        V_L
        \mathbf P(G_{L,h}^{(\kappa)}=n)
        <\infty .
\]

Finally, again by the Girsanov formula,
\[
        \mathbf E_L^{(\kappa)}
        \left[
            e^{-\langle h,\Pi_{L,\mathrm{eff}}^{(\kappa)}\rangle}
            \mathbf 1_{\{G_L^{(\kappa)}=n\}}
        \right]                                                
 =
        e^{-A_L^{(\kappa)}(h)}
        \mathbf P(G_{L,h}^{(\kappa)}=n).
\]
Since \(A_L^{(\kappa)}(h)\) is uniformly bounded, the desired estimate follows.
\end{proof}

\begin{theorem}[Effective \(h\)-tilted lattice local limit theorem]
\label{thm:effective-h-tilted-LLT}
Assume~\cref{ass:limiting-effective-kernel},
\cref{ass:effective-marked-trace-convergence}, and
\cref{ass:effective-spectral-LLT-criterion}. Then, for every
\(\kappa>0\), every \(h\in\mathcal H\), and every compact interval
\(K\subset(0,\infty)\),
\[
        \sup_{\substack{n\in\mathbb N:\\ n/V_L\in K}}
        \left|
            V_L
            \mathbf P
            \left(
                G_{L,h}^{(\kappa)}=n
            \right)
            -
            f_h^{(\kappa)}
            \left(
                \frac n{V_L}
            \right)
        \right|
        \longrightarrow0 .
\]
\end{theorem}
\begin{proof}
  The condition (B) has been studied in~\cref{lem:compact-thinning-deconvolution-transfer}. Here we only talk about the condition (A). In this case, Fourier inversion gives
\[
        V_L
        \mathbf P(G_{L,h}^{(\kappa)}=n)
        =
        \frac1{2\pi}
        \int_{-\pi V_L}^{\pi V_L}
        e^{-itn/V_L}
        \varphi_{L,h}^{(\kappa)}(t)\,dt .
\]
Fix \(A<\infty\).  Then, uniformly in \(n\),
\[
\begin{aligned}
&\left|
\frac1{2\pi}
\int_{-A}^{A}
e^{-itn/V_L}
\left(
    \varphi_{L,h}^{(\kappa)}(t)
    -
    \varphi_h^{(\kappa)}(t)
\right)
\,dt
\right|                                                       \\
&\qquad\le
\frac{A}{\pi}
\sup_{|t|\le A}
\left|
    \varphi_{L,h}^{(\kappa)}(t)
    -
    \varphi_h^{(\kappa)}(t)
\right|
\longrightarrow0
\end{aligned}
\]
by \cref{lem:compact-fourier-convergence}.  Moreover,
\cref{lem:fourier-tail-control-A} gives
\[
        \lim_{A\to\infty}
        \limsup_{L\to\infty}
        \int_{A<|t|\le \pi V_L}
        \left|
            \varphi_{L,h}^{(\kappa)}(t)
        \right|dt
        =
        0,
\]
and also implies
\[
        \varphi_h^{(\kappa)}\in L^1(\mathbb R).
\]
Therefore
\[
        V_L
        \mathbf P(G_{L,h}^{(\kappa)}=n)
        =
        \frac1{2\pi}
        \int_{\mathbb R}
        e^{-itn/V_L}
        \varphi_h^{(\kappa)}(t)\,dt
        +o(1),
\]
uniformly in \(n\).  Since
\[
        f_h^{(\kappa)}(a)
        =
        \frac1{2\pi}
        \int_{\mathbb R}
        e^{-ita}\varphi_h^{(\kappa)}(t)\,dt,
\]
the desired local limit theorem follows under \(\mathrm{(A)}\).

\end{proof}

The following weighted version is the form used in the canonical conditioning
argument.
\begin{corollary}[Weighted effective local limit theorem]
\label{cor:weighted-effective-LLT}
Assume the hypotheses of~\cref{thm:effective-h-tilted-LLT}.  Then, for every \(\kappa>0\),
every \(h\in\mathcal H\), and every compact interval
\(K\subset(0,\infty)\),
\[
        \sup_{\substack{n\in\mathbb N:\\ n/V_L\in K}}
        \Bigg|
        V_L
        \mathbf E_L^{(\kappa)}
        \left[
            e^{-\langle h,\Pi_{L,\mathrm{eff}}^{(\kappa)}\rangle}
            \mathbf 1_{\{G_L^{(\kappa)}=n\}}
        \right]
        -
        e^{-A^{(\kappa)}(h)}
        f_h^{(\kappa)}
        \left(
            \frac n{V_L}
        \right)
        \Bigg|
        \longrightarrow0 .
\]
\end{corollary}

\begin{proof}
The Poisson change-of-intensity formula gives
\[
        \mathbf E_L^{(\kappa)}
        \left[
            e^{-\langle h,\Pi_{L,\mathrm{eff}}^{(\kappa)}\rangle}
            \mathbf 1_{\{G_L^{(\kappa)}=n\}}
        \right]
   =
        e^{-A_L^{(\kappa)}(h)}
        \mathbf P
        \left(
            G_{L,h}^{(\kappa)}=n
        \right).
\]
Combining the two limits in~\cref{prop:unconditioned-effective-poisson-convergence}
and~\cref{thm:effective-h-tilted-LLT} proves the result.
\end{proof}

\subsection{Background concentration and invisibility}
\label{subsec:proof-background-concentration}

We now prove that the background part contributes only a deterministic
density and has no visible atoms in the length-bounded topology.

\begin{lemma}[Grand-canonical background density concentration]
\label{lem:background-density-concentration-proof}
Let $\kappa$ satisfies~\cref{ass:background-density-concentration}.  Then,
\[
        \frac{B_L^{(\kappa)}}{V_L}
        \longrightarrow
        \rho_{\mathrm{bg}}
        \qquad
        \text{in probability under }\mathbf P_L^{(\kappa)} .
\]
\end{lemma}

\begin{proof}
By the computation in~\cref{subsec:proof-effective-background-splitting},
\[
        \mathbf E_L^{(\kappa)}
        \left[
            \frac{B_L^{(\kappa)}}{V_L}
        \right]
        =
        m_{L,\mathrm{bg}}^{(\kappa)}
        \longrightarrow
        \rho_{\mathrm{bg}},
\qquad
        \operatorname{Var}_L^{(\kappa)}
        \left(
            \frac{B_L^{(\kappa)}}{V_L}
        \right)
        =
        v_{L,\mathrm{bg}}^{(\kappa)}
        \longrightarrow0 .
\]
Hence Chebyshev's inequality gives the claim immediately.
\end{proof}

\begin{lemma}[Grand-canonical background invisibility]
\label{lem:grand-canonical-background-invisibility}
Let $\kappa$ satisfies~\cref{ass:background-density-concentration}.   Then, for every \(0<\delta<M<\infty\),
\[
        \mathbf P_L^{(\kappa)}
        \left(
            \Pi_{L,\mathrm{bg}}^{(\kappa)}
            \bigl(
                [0,1]\times[\delta,M]\times\mathsf M
            \bigr)>0
        \right)
        \longrightarrow0 .
\]
\end{lemma}

\begin{proof}

By Markov's inequality,
\[
        \mathbf P_L^{(\kappa)}
        \left(
            \Pi_{L,\mathrm{bg}}^{(\kappa)}( [0,1]\times[\delta,M]\times\mathsf M)>0
        \right)
\le
        \mathbf E_L^{(\kappa)}
        \left[
            \Pi_{L,\mathrm{bg}}^{(\kappa)}( [0,1]\times[\delta,M]\times\mathsf M)
        \right]
=
        \sum_{\delta V_L\le j\le M V_L}
        \frac{e^{-\kappa j/V_L}}{j}
        q_{L,j}^{\mathrm{bg}} .
\]
On the summation range \(j\ge \delta V_L\), hence
\[
        \frac1j
        \le
        \frac{j}{\delta^2 V_L^2}.
\]
Therefore
\[
        \sum_{\delta V_L\le j\le M V_L}
        \frac{e^{-\kappa j/V_L}}{j}
        q_{L,j}^{\mathrm{bg}}
        \le
        \frac1{\delta^2 V_L^2}
        \sum_{j\ge1}
        j\,e^{-\kappa j/V_L}
        q_{L,j}^{\mathrm{bg}}=
        \frac{v_{L,\mathrm{bg}}^{(\kappa)}}{\delta^2}.
\]
By~\cref{ass:background-density-concentration},
\(v_{L,\mathrm{bg}}^{(\kappa)}\to0\).  This proves the claim.
\end{proof}

\subsection{Bridge identity and canonical bridge convergence}
\label{subsec:proof-bridge-convergence}

We now combine the weighted effective local limit theorem~\cref{cor:weighted-effective-LLT} with the background
concentration estimate~\cref{lem:background-density-concentration-proof}.  This yields a full local limit theorem for the total
particle number
\[
        S_L^{(\kappa)}
        =
        G_L^{(\kappa)}+B_L^{(\kappa)}.
\]
After conditioning on \(S_L^{(\kappa)}=N_L\), the limiting effective process
is identified as the marked Poisson--Kingman bridge.

We first give the bridge identity for the limiting Poisson process.

\begin{lemma}[Limiting bridge identity]
\label{lem:limiting-bridge-identity}
Let \(h\in\mathcal H\).  Assume that \(T^{(\kappa)}\) and
\(T_h^{(\kappa)}\) have densities \(f_0^{(\kappa)}\) and
\(f_h^{(\kappa)}\), respectively.  Then
\begin{equation}\label{limitpp}
        \mathbf E
        \left[
            e^{-\langle h,\Pi^{(\kappa)}\rangle};
            T^{(\kappa)}\in da
        \right]
        =
        e^{-A^{(\kappa)}(h)}
        f_h^{(\kappa)}(a)\,da .
\end{equation}
Consequently, for every \(a>0\) such that \(f_0^{(\kappa)}(a)>0\),
\begin{equation}\label{limitbr}
        \mathbf E
        \exp\left\{
            -\langle h,\Pi_a^{\mathrm{br}}\rangle
        \right\}
        =
        e^{-A^{(\kappa)}(h)}
        \frac{
            f_h^{(\kappa)}(a)
        }{
            f_0^{(\kappa)}(a)
        } .
\end{equation}
\end{lemma}

\begin{proof}
By the Girsanov formula,
multiplication by
$   e^{-\langle h,\Pi^{(\kappa)}\rangle}$
changes the law of \(\Pi^{(\kappa)}\) into the law of
\(\Pi_h^{(\kappa)}\), up to the normalizing factor
$       \exp\{-A^{(\kappa)}(h)\}$.
Thus, for every bounded measurable function \(g:(0,\infty)\to\mathbb R\),
\[
\begin{aligned}
        \mathbf E
        \left[
            e^{-\langle h,\Pi^{(\kappa)}\rangle}
            g(T^{(\kappa)})
        \right]
        &=
        e^{-A^{(\kappa)}(h)}
        \mathbf E
        \left[
            g(T_h^{(\kappa)})
        \right]  \\
        &=
        e^{-A^{(\kappa)}(h)}
        \int_0^\infty
        g(a)f_h^{(\kappa)}(a)\,da .
\end{aligned}
\]
This proves the density identity~\eqref{limitpp}.  Dividing by the density
\(f_0^{(\kappa)}(a)\) of \(T^{(\kappa)}\) gives the bridge Laplace functional~\eqref{limitbr}.
\end{proof}

\begin{lemma}[Independence of the bridge from \(\kappa\)]
\label{lem:bridge-kappa-independence}
Let \(\kappa,\kappa'>0\).  Suppose that the corresponding bridges are defined
at the same mass \(a>0\).  Then
\[
        \mathcal L
        \left(
            \Pi^{(\kappa)}
            \,\middle|\,
            T^{(\kappa)}=a
        \right)
        =
        \mathcal L
        \left(
            \Pi^{(\kappa')}
            \,\middle|\,
            T^{(\kappa')}=a
        \right).
\]
\end{lemma}

\begin{proof}
The two limiting intensities are related by
\[
        \nu^{(\kappa')}(du,dx,dm)
        =
        e^{-(\kappa'-\kappa)x}
        \nu^{(\kappa)}(du,dx,dm).
\]
Changing \(\kappa\) to \(\kappa'\) is therefore an exponential tilt by the
total mass
\[
        T=\int_E x\,\Pi(du,dx,dm).
\]
After conditioning on \(T=a\), this tilt becomes the constant
\(e^{-(\kappa'-\kappa)a}\), which cancels in conditional expectations.  Hence
the bridge law does not depend on the auxiliary parameter.
\end{proof}

We next prove the full weighted local limit theorem.  It says that the
background only shifts the effective local limit by the deterministic density
\(\rho_{\mathrm{bg}}\).

\begin{theorem}[Full weighted local limit theorem]
\label{thm:full-weighted-LLT}
Assume~\cref{ass:limiting-effective-kernel},
\cref{ass:effective-marked-trace-convergence},
\cref{ass:effective-spectral-LLT-criterion}, and let \(\kappa>0\) satisfies
\cref{ass:background-density-concentration}.  Fix 
\(h\in\mathcal H\).  Then, for every compact interval
$       K\subset(\rho_{\mathrm{bg}},\infty)$,
one has
\[
        \sup_{\substack{N\in\mathbb N:\\ N/V_L\in K}}
        \Bigg|
        V_L
        \mathbf E_L^{(\kappa)}
        \left[
            e^{-\langle h,\Pi_{L,\mathrm{eff}}^{(\kappa)}\rangle}
            \mathbf 1_{\{S_L^{(\kappa)}=N\}}
        \right]
        -
        e^{-A^{(\kappa)}(h)}
        f_h^{(\kappa)}
        \left(
            \frac N{V_L}-\rho_{\mathrm{bg}}
        \right)
        \Bigg|
        \longrightarrow0 .
\]
In particular, taking \(h=0\),
\[
        \sup_{\substack{N\in\mathbb N:\\ N/V_L\in K}}
        \left|
        V_L
        \mathbf P_L^{(\kappa)}
        \left(
            S_L^{(\kappa)}=N
        \right)
        -
        f_0^{(\kappa)}
        \left(
            \frac N{V_L}-\rho_{\mathrm{bg}}
        \right)
        \right|
        \longrightarrow0 .
\]
\end{theorem}

\begin{proof}
By independence of the effective and background Poisson processes,
\[
        V_L
        \mathbf E_L^{(\kappa)}
        \left[
            e^{-\langle h,\Pi_{L,\mathrm{eff}}^{(\kappa)}\rangle}
            \mathbf 1_{\{S_L^{(\kappa)}=N\}}
        \right]
=
        \mathbf E_L^{(\kappa)}
        \left[
            R_L^{(h)}
            \left(
                N-B_L^{(\kappa)}
            \right)
        \right],
\]
where
\[
        R_L^{(h)}(n)
        :=
        V_L
        \mathbf E_L^{(\kappa)}
        \left[
            e^{-\langle h,\Pi_{L,\mathrm{eff}}^{(\kappa)}\rangle}
            \mathbf 1_{\{G_L^{(\kappa)}=n\}}
        \right],
        \qquad n\in\mathbb N,
\]
and we set \(R_L^{(h)}(n)=0\) for \(n<0\).

Let
\[
        a_L(N):=\frac N{V_L}-\rho_{\mathrm{bg}} .
\]
Since \(K\subset(\rho_{\mathrm{bg}},\infty)\) is compact, there exists
\(\alpha>0\) such that
\[
        a_L(N)\ge 2\alpha
\]
for all sufficiently large \(L\) and all \(N\) with \(N/V_L\in K\).

Fix \(0<\varepsilon<\alpha\).  On the event
\[
        \mathcal G_{L,\varepsilon}
        :=
        \left\{
            \left|
            \frac{B_L^{(\kappa)}}{V_L}
            -
            \rho_{\mathrm{bg}}
            \right|
            \le\varepsilon
        \right\},
\]
we have
\[
        \frac{N-B_L^{(\kappa)}}{V_L}
        \in K_\varepsilon
\]
for a compact interval \(K_\varepsilon\subset(0,\infty)\), uniformly in
\(N/V_L\in K\) and all sufficiently large \(L\).  Hence the weighted effective
local limit theorem,
~\cref{cor:weighted-effective-LLT}, gives
\[
        R_L^{(h)}
        \left(
            N-B_L^{(\kappa)}
        \right)
        =
        e^{-A^{(\kappa)}(h)}
        f_h^{(\kappa)}
        \left(
            \frac{N-B_L^{(\kappa)}}{V_L}
        \right)
        +
        o(1),
\]
uniformly on \(\mathcal G_{L,\varepsilon}\) and uniformly for \(N/V_L\in K\).

The density \(f_h^{(\kappa)}\) is continuous, hence uniformly continuous on
the relevant compact interval.  Therefore, on
\(\mathcal G_{L,\varepsilon}\),
\[
        f_h^{(\kappa)}
        \left(
            \frac{N-B_L^{(\kappa)}}{V_L}
        \right)
\]
is uniformly close to
\[
        f_h^{(\kappa)}
        \left(
            \frac N{V_L}-\rho_{\mathrm{bg}}
        \right)
\]
when \(\varepsilon\) is small.

It remains to control the complement of \(\mathcal G_{L,\varepsilon}\).  By~\cref{lem:effective-lattice-bound} and the Girsanov formula, there exists \(C_h<\infty\) such that
\[
        \sup_{L}
        \sup_{n\ge0}
        R_L^{(h)}(n)
        \le C_h .
\]
Hence
\[
        \mathbf E_L^{(\kappa)}
        \left[
            R_L^{(h)}
            \left(
                N-B_L^{(\kappa)}
            \right)
            \mathbf 1_{\mathcal G_{L,\varepsilon}^c}
        \right]
        \le
        C_h
        \mathbf P_L^{(\kappa)}
        \left(
            \mathcal G_{L,\varepsilon}^c
        \right).
\]
By~\cref{lem:background-density-concentration-proof}, this tends to zero.
Letting first \(L\to\infty\) and then \(\varepsilon\downarrow0\) proves the
weighted local limit theorem.  The case \(h=0\) gives the unweighted statement.
\end{proof}

We now pass from the tilted grand-canonical law to the canonical law.

\begin{lemma}[Canonical effective bridge convergence]
\label{lem:canonical-effective-bridge-convergence}
Assume the hypotheses of~\cref{thm:main-canonical-bridge-limit}.  Then, under
\(\mathbb P_{L,N_L}^{\mathrm{can}}\),
\[
        \Xi_{L,\mathrm{eff}}
        \Longrightarrow
        \Pi_{\rho_{\mathrm{eff}}}^{\mathrm{br}}
        \qquad
        \text{in }\mathcal N_\ell(E).
\]
\end{lemma}

\begin{proof}
Let \(h\in\mathcal H\).  By the canonical conditioning~\eqref{eq:can},
\[
        \mathbb E_{L,N_L}^{\mathrm{can}}
        \exp\left\{
            -\langle h,\Xi_{L,\mathrm{eff}}\rangle
        \right\}
=
        \frac{
        \mathbf E_L^{(\kappa)}
        \left[
            e^{-\langle h,\Pi_{L,\mathrm{eff}}^{(\kappa)}\rangle}
            \mathbf 1_{\{S_L^{(\kappa)}=N_L\}}
        \right]
        }{
        \mathbf P_L^{(\kappa)}
        \left(
            S_L^{(\kappa)}=N_L
        \right)
        } .
\]
Since
     $   \frac{N_L}{V_L}\to\rho$
and
   $     \rho_{\mathrm{eff}}=\rho-\rho_{\mathrm{bg}}$,
\cref{thm:full-weighted-LLT} gives
\[
        V_L
        \mathbf E_L^{(\kappa)}
        \left[
            e^{-\langle h,\Pi_{L,\mathrm{eff}}^{(\kappa)}\rangle}
            \mathbf 1_{\{S_L^{(\kappa)}=N_L\}}
        \right]\longrightarrow
        e^{-A^{(\kappa)}(h)}
        f_h^{(\kappa)}(\rho_{\mathrm{eff}}),
\]
and, with \(h=0\),
\[
        V_L
        \mathbf P_L^{(\kappa)}
        \left(
            S_L^{(\kappa)}=N_L
        \right)
        \longrightarrow
        f_0^{(\kappa)}(\rho_{\mathrm{eff}}).
\]
By~\cref{ass:canonical-density-regime},
$       f_0^{(\kappa)}(\rho_{\mathrm{eff}})>0$.
Thus
\[
        \mathbb E_{L,N_L}^{\mathrm{can}}
        \exp\left\{
            -\langle h,\Xi_{L,\mathrm{eff}}\rangle
        \right\}
        \longrightarrow
        e^{-A^{(\kappa)}(h)}
        \frac{
            f_h^{(\kappa)}(\rho_{\mathrm{eff}})
        }{
            f_0^{(\kappa)}(\rho_{\mathrm{eff}})
        } .
\]
By~\cref{lem:limiting-bridge-identity}, the right-hand side is exactly
the Laplace functional of $        \Pi_{\rho_{\mathrm{eff}}}^{\mathrm{br}}$.
Therefore
$       \Xi_{L,\mathrm{eff}}
        \Longrightarrow
        \Pi_{\rho_{\mathrm{eff}}}^{\mathrm{br}}$
in \(\mathcal N_\ell(E)\).
\end{proof}

\begin{lemma}[Canonical background and effective mass concentration]
\label{lem:canonical-mass-concentration}
Assume the hypotheses of~\cref{thm:main-canonical-bridge-limit}.  Then, under
$\mathbb P_{L,N_L}^{\mathrm{can}}$,
\[
 \frac{B_L}{V_L}\xrightarrow{\mathbb P}\rho_{\mathrm{bg}},
 \qquad
 \frac{G_L}{V_L}\xrightarrow{\mathbb P}\rho_{\mathrm{eff}}.
\]
\end{lemma}

\begin{proof}
Fix $\varepsilon>0$ and put
\[
 D_{L,\varepsilon}:=
 \left\{\left|B_L^{(\kappa)}/V_L-\rho_{\mathrm{bg}}\right|>
 \varepsilon\right\}.
\]
The conditional representation and independence give
\[
\begin{aligned}
 &\mathbf P_L^{(\kappa)}
   \bigl(D_{L,\varepsilon},S_L^{(\kappa)}=N_L\bigr)\\
 &\quad=\sum_{b\ge0}
 \mathbf P_L^{(\kappa)}
   \bigl(D_{L,\varepsilon},B_L^{(\kappa)}=b\bigr)
 \mathbf P_L^{(\kappa)}\bigl(G_L^{(\kappa)}=N_L-b\bigr).
\end{aligned}
\]
By the uniform effective lattice bound, the last display is at most
\[
 \frac{C}{V_L}\mathbf P_L^{(\kappa)}(D_{L,\varepsilon})=o(V_L^{-1}),
\]
where the last equality follows from
\cref{lem:background-density-concentration-proof}.  On the other hand,
\cref{thm:full-weighted-LLT} and positivity of the endpoint density give
\[
 \mathbf P_L^{(\kappa)}(S_L^{(\kappa)}=N_L)
 \sim \frac{f_0^{(\kappa)}(\rho_{\mathrm{eff}})}{V_L}.
\]
Dividing proves canonical concentration of $B_L/V_L$.  Finally,
\[
 \frac{G_L}{V_L}=\frac{N_L}{V_L}-\frac{B_L}{V_L}
 \xrightarrow{\mathbb P}\rho-\rho_{\mathrm{bg}}
 =\rho_{\mathrm{eff}}.
\]
\end{proof}

\begin{lemma}[Canonical background invisibility]
\label{lem:canonical-background-invisibility}
Assume the hypotheses of~\cref{thm:main-canonical-bridge-limit}.  Then, for every
\(0<\delta<M<\infty\),
\[
        \mathbb P_{L,N_L}^{\mathrm{can}}
        \left(
            \Xi_{L,\mathrm{bg}}
            \bigl(
                [0,1]\times[\delta,M]\times\mathsf M
            \bigr)>0
        \right)
        \longrightarrow0 .
\]
\end{lemma}

\begin{proof}

By canonical conditioning~\eqref{eq:can},
\[
\begin{aligned}
        &
        \mathbb P_{L,N_L}^{\mathrm{can}}
        \left(
            \Xi_{L,\mathrm{bg}}([0,1]\times[\delta,M]\times\mathsf M )>0
        \right)
        \\
        &\qquad =
        \frac{
        \mathbf P_L^{(\kappa)}
        \left(
            \Pi_{L,\mathrm{bg}}^{(\kappa)}([0,1]\times[\delta,M]\times\mathsf M )>0,\,
            S_L^{(\kappa)}=N_L
        \right)
        }{
        \mathbf P_L^{(\kappa)}
        \left(
            S_L^{(\kappa)}=N_L
        \right)
        } .
\end{aligned}
\]

For the numerator, use independence of the effective and background parts:
\[
\begin{aligned}
        &
        \mathbf P_L^{(\kappa)}
        \left(
            \Pi_{L,\mathrm{bg}}^{(\kappa)}([0,1]\times[\delta,M]\times\mathsf M )>0,\,
            S_L^{(\kappa)}=N_L
        \right)
        \\
        &\qquad =
        \sum_{b\ge0}
        \mathbf P_L^{(\kappa)}
        \left(
            \Pi_{L,\mathrm{bg}}^{(\kappa)}([0,1]\times[\delta,M]\times\mathsf M )>0,\,
            B_L^{(\kappa)}=b
        \right)
        \mathbf P_L^{(\kappa)}
        \left(
            G_L^{(\kappa)}=N_L-b
        \right).
\end{aligned}
\]
By the uniform lattice bound from~\cref{lem:effective-lattice-bound} with
\(h=0\), there exists \(C<\infty\) such that
\[
        \sup_{L,n}
        V_L
        \mathbf P_L^{(\kappa)}
        \left(
            G_L^{(\kappa)}=n
        \right)
        \le C .
\]
Therefore
\[
\begin{aligned}
        &
        \mathbf P_L^{(\kappa)}
        \left(
            \Pi_{L,\mathrm{bg}}^{(\kappa)}([0,1]\times[\delta,M]\times\mathsf M )>0,\,
            S_L^{(\kappa)}=N_L
        \right)
        \\
        &\qquad \le
        \frac{C}{V_L}
        \mathbf P_L^{(\kappa)}
        \left(
            \Pi_{L,\mathrm{bg}}^{(\kappa)}([0,1]\times[\delta,M]\times\mathsf M )>0
        \right).
\end{aligned}
\]
By~\cref{lem:grand-canonical-background-invisibility}, the last
probability tends to zero.  Hence the numerator is \(o(V_L^{-1})\).

On the other hand, \cref{thm:full-weighted-LLT} with \(h=0\) gives
\[
        \mathbf P_L^{(\kappa)}
        \left(
            S_L^{(\kappa)}=N_L
        \right)
        \sim
        \frac1{V_L}
        f_0^{(\kappa)}(\rho_{\mathrm{eff}}),
\]
and the limit density is positive by~\cref{ass:canonical-density-regime}.  Therefore the quotient tends
to zero.
\end{proof}
Now we present the proof of our main result and its corollary.
\begin{proof}[Proof of~\cref{thm:main-canonical-bridge-limit}]
\cref{lem:canonical-effective-bridge-convergence} gives
\[
        \Xi_{L,\mathrm{eff}}
        \Longrightarrow
        \Pi_{\rho_{\mathrm{eff}}}^{\mathrm{br}}
        \qquad
        \text{in }\mathcal N_\ell(E).
\]
\cref{lem:canonical-background-invisibility} gives, for every
\(0<\delta<M<\infty\),
\[
        \mathbb P_{L,N_L}^{\mathrm{can}}
        \left(
            \Xi_{L,\mathrm{bg}}
            ([0,1]\times[\delta,M]\times\mathsf M )>0
        \right)
        \longrightarrow0 .
\]
Let \(h\in\mathcal H\), and choose \(0<\delta_h<M_h<\infty\) such that
\(h\) vanishes outside the length window
\([0,1]\times[\delta_h,M_h]\times\mathsf M\).  Then
\[
\begin{aligned}
        &
        \left|
        \exp\{-\langle h,\Xi_L\rangle\}
        -
        \exp\{-\langle h,\Xi_{L,\mathrm{eff}}\rangle\}
        \right|
        \\
        &\qquad\le
        \mathbf 1_{\{
            \Xi_{L,\mathrm{bg}}
            ([0,1]\times[\delta_h,M_h]\times\mathsf M)>0
        \}} .
\end{aligned}
\]
The right-hand side tends to zero in probability and in expectation.  Hence
the Laplace functionals of \(\Xi_L\) and \(\Xi_{L,\mathrm{eff}}\) have the
same limit.  Therefore
\[
        \Xi_L
        =
        \Xi_{L,\mathrm{eff}}+\Xi_{L,\mathrm{bg}}
        \Longrightarrow
        \Pi_{\rho_{\mathrm{eff}}}^{\mathrm{br}}
        \qquad
        \text{in }\mathcal N_\ell(E).
\]
The rest of the statement are exactly~\cref{lem:canonical-background-invisibility} and~\cref{lem:canonical-mass-concentration}.
\end{proof}
\begin{proof}[Proof of~\cref{cor:ranked-effective-length-convergence}]
By~\cref{thm:main-canonical-bridge-limit},
\[
    \Xi_{L,N_L}^{\mathrm{eff}}
    \Longrightarrow
    \Pi_{\rho_{\mathrm{eff}}}^{\mathrm{br}}
    \qquad
    \text{in } \mathcal N_\ell(\mathsf E).
\]
Moreover, by the canonical effective/background decomposition,
\[
    T_L
    :=
    \sum_{i\ge1}\ell_i^L
    =
    \int_{\mathsf E} x\,
    \Xi_{L,N_L}^{\mathrm{eff}}(du,dx,dm)
    =
    \frac{G_L}{V_L}
    \xrightarrow{\mathbb P}
    \rho_{\mathrm{eff}} .
\]
The limiting bridge is conditioned to have total macroscopic mass
\(\rho_{\mathrm{eff}}\), hence
\[
    \sum_{i\ge1}\ell_i
    =
    \rho_{\mathrm{eff}}
    \qquad
    \text{a.s.}
\]

We first prove convergence of finite initial segments.  Fix \(m\ge1\).
By assumption, the limiting bridge has almost surely no ties in the length
coordinate.  On this event, choose
\[
    0<\delta<\ell_m
\]
such that \(\delta\) is not the length of an atom of
\(\Pi_{\rho_{\mathrm{eff}}}^{\mathrm{br}}\).  The restriction of the limiting
point measure to the window
 \[
    \text{}
    [0,1]\times[\delta,R]\times\mathsf M,
  \]
  where $R>\rho+1$ is fixed.  Every finite-volume cycle has rescaled
  length at most $N_L/V_L<R$ for all sufficiently large $L$, and every limiting
  atom has length at most $\rho_{\mathrm{eff}}<R$, so this restriction loses no
  atom relevant to the first $m$ ranks.  It
contains only finitely many atoms.  Since the length-bounded topology gives
convergence of the restricted point measures on such windows, and since
ranking finitely many atoms by distinct length coordinates is continuous, we
obtain
\[
    (\ell_1^L,\ldots,\ell_m^L)
    \Longrightarrow
    (\ell_1,\ldots,\ell_m).
\]
Together with \(T_L\to\rho_{\mathrm{eff}}\) in probability, Slutsky's theorem
gives
\[
    \left(
        T_L,\ell_1^L,\ldots,\ell_m^L
    \right)
    \Longrightarrow
    \left(
        \rho_{\mathrm{eff}},\ell_1,\ldots,\ell_m
    \right).
\]

For the tails, observe that for every fixed \(m\),
\[
    \sum_{i>m}\ell_i^L
    =
    T_L-\sum_{i=1}^m\ell_i^L .
\]
Therefore
\[
    \sum_{i>m}\ell_i^L
    \Longrightarrow
    \rho_{\mathrm{eff}}-\sum_{i=1}^m\ell_i
    =
    \sum_{i>m}\ell_i .
\]
By the Portmanteau theorem, for every \(\varepsilon>0\),
\[
    \limsup_{L\to\infty}
    \mathbb P\left(
        \sum_{i>m}\ell_i^L>\varepsilon
    \right)
    \le
    \mathbb P\left(
        \sum_{i>m}\ell_i\ge\varepsilon
    \right).
\]
Since
\[
    \sum_{i\ge1}\ell_i
    =
    \rho_{\mathrm{eff}}
    <
    \infty
    \qquad
    \text{a.s.},
\]
the right-hand side tends to \(0\) as \(m\to\infty\).  Hence
\[
    \lim_{m\to\infty}
    \limsup_{L\to\infty}
    \mathbb P\left(
        \sum_{i>m}\ell_i^L>\varepsilon
    \right)
    =
    0 .
\]

Let \(\pi_m\) denote truncation after the first \(m\) coordinates.  The
finite-dimensional convergence above gives
\[
    \pi_m\ell^L
    \Longrightarrow
    \pi_m\ell
    \qquad
    \text{in } \ell^1_\downarrow
\]
for every fixed \(m\).  The preceding tail estimate shows that
\[
    \|\ell^L-\pi_m\ell^L\|_1
    =
    \sum_{i>m}\ell_i^L
\]
is negligible uniformly in \(L\) as \(m\to\infty\), while
\[
    \|\ell-\pi_m\ell\|_1
    =
    \sum_{i>m}\ell_i
    \to0
    \qquad
    \text{a.s.}
\]
The standard truncation argument for weak convergence in \(\ell^1\) yields
\[
    \ell^L
    \Longrightarrow
    \ell
    \qquad
    \text{in } \ell^1_\downarrow .
\]
This proves the corollary.
\end{proof}